\documentclass[a4paper, oneside, 10pt, final]{amsart}
\usepackage[top=2.5cm, bottom=2cm, inner=2.5cm, outer=2.5cm, foot=1cm]{geometry}
\usepackage{amssymb,amsfonts,microtype}
\usepackage[utf8]{inputenc}
\usepackage[all, 2cell]{xy} \UseAllTwocells \SilentMatrices
\usepackage{lmodern}
\usepackage{svninfo}
\usepackage[bookmarks=true,colorlinks=true,linkcolor=blue,pdfa=true]{hyperref}
\usepackage[notcite,notref]{showkeys}

\hypersetup{pdftitle={A precise result on the arithmetic of non-principal orders in algebraic number fields},pdfauthor={Andreas Philipp}}
\usepackage{aliascnt}
\let\orgautoref\autoref
\renewcommand{\autoref}
        {%
\def\corollaryautorefname{Corollary}%
\def\definitionautorefname{Definition}%
\def\lemmaautorefname{Lemma}%
\def\propositionautorefname{Proposition}%
\def\exampleautorefname{Example}%
\def\remarkautorefname{Remark}%
\def\chapterautorefname{Chapter}%
\def\sectionautorefname{Section}%
\def\subsectionautorefname{Subsection}%
\def\subsubsectionautorefname{Subsubsection}%
\def\algorithmautorefname{Algorithm}%
         \orgautoref}

\newtheorem*{theorem*}{Theorem}
\newtheorem{theorem}{Theorem}[section]
\newaliascnt{lemma}{theorem}
\newtheorem{lemma}[lemma]{Lemma}
\aliascntresetthe{lemma}
\newaliascnt{corollary}{theorem}
\newtheorem{corollary}[corollary]{Corollary}
\aliascntresetthe{corollary}
\newaliascnt{proposition}{theorem}
\newtheorem{proposition}[proposition]{Proposition}
\aliascntresetthe{proposition}
\theoremstyle{definition}
\newaliascnt{definition}{theorem}
\newtheorem{definition}[definition]{Definition}
\aliascntresetthe{definition}
\newaliascnt{remark}{theorem}

\aliascntresetthe{remark}
\newaliascnt{notation}{theorem}

\aliascntresetthe{notation}
\newaliascnt{situation}{theorem}

\aliascntresetthe{situation}
\newaliascnt{example}{theorem}
\newtheorem{example}[example]{Example}
\aliascntresetthe{example}

\newcommand{\N}{\mathbb N}

\newcommand{\Z}{\mathbb Z}
\newcommand{\Q}{\mathbb Q}

\DeclareMathOperator{\spec}{spec}

\DeclareMathOperator{\ord}{ord}
\DeclareMathOperator{\pic}{Pic}
\DeclareMathOperator{\red}{red}
\newcommand{\cat}{\mathsf c}

\newcommand{\cmon}{\cat_{\mathrm{mon}}}

\newcommand{\Req}{\mathcal R_{\mathrm{eq}}}
\newcommand{\equ}{\approx_{\mathrm{eq}}}
\newcommand{\eq}{\approx_{\mathrm{eq}}}

\newcommand{\punkt}[1]{#1^\bullet}
\newcommand{\mal}[1]{#1^\times}
\newcommand{\mdots}{\cdot\ldots\cdot}

\newcommand{\wmal}[1]{\mal{\widehat{#1}}}
\newcommand{\eps}{\varepsilon}

\def\clap#1{\hbox to 0pt{\hss#1\hss}}
\def\mathclap{\mathpalette\mathclapinternal}
\def\mathclapinternal#1#2{%
\clap{$\mathsurround=0pt#1{#2}$}%
}

\numberwithin{equation}{section}

\makeatletter
\renewcommand{\p@enumii}{}
\makeatother

\begin{document}
\svnInfo $Id: precise_orders.tex 2262 2011-04-20 08:26:24Z aph $
\svnKeyword $HeadURL: svn+ssh://aph@aph.homelinux.net/var/svn/notes/publications/submitted/precise_orders/precise_orders.tex $

\address{Institut f\"ur Mathematik und Wissenschaftliches Rechnen \\
Karl-Franzens-Universit\"at Graz \\
Heinrichstra\ss e 36\\
8010 Graz, Austria} \email{andreas.philipp@uni-graz.at}

\thanks{I thank my Ph.D. thesis advisors Prof. Franz Halter-Koch and Prof. Alfred Geroldinger for all the help, advice, and mathematical discussions during my thesis which led to all results in this article.}

\author{Andreas Philipp}



\begin{abstract}
Let $R$ be an order in an algebraic number field. If $R$ is a principal order, then many explicit results on its arithmetic are available. Among others, $R$ is half-factorial if and only if the class group of $R$ has at most two elements. Much less is known for non-principal orders. Using  a new semigroup theoretical approach, we study half-factoriality and further arithmetical properties for non-principal orders in algebraic number fields.
\end{abstract}

\title[Arithmetic of non-principal orders]{A precise result on the arithmetic of non-principal orders in algebraic number fields}

\maketitle

\bigskip
\section{Introduction and Main Result}
\bigskip

Let $R$ be a noetherian domain. Then every non-zero non-unit  $a \in R$ can be written as a finite product of atoms, say $a = u_1 \cdot\ldots \cdot u_k$. In general, $a$ has many essentially different factorizations into atoms. The non-uniqueness of factorizations of elements in $R$ is measured by arithmetical invariants. For convenience, we briefly recall the definition of two classical invariants, the elasticity and the set of distances (details will be given in \autoref{sec:pre}). In a factorization of an element $a \in R$ as above, the number of factors $k$ is called the length of the factorization. Then the elasticity $\rho (a) \in \mathbb R_{\ge 1}\cup \{\infty\}$ is defined as the supremum over all $k/l$ where $k$ and $l$ are lengths of factorizations of $a$. Suppose that $a = u_1\cdot \ldots\cdot u_k = v_1 \cdot \ldots \cdot v_l$, where $k < l$ and all $u_i$ and all $v_j$ are atoms of $R$. If $a$ has no factorizations of length $m$ with $k < m < l$, then $l-k$ is said to be a distance of two (successive) factorization lengths, and $\Delta (a) \subset \N$ is the set of all such distances. The elasticity $\rho (R)$ is the supremum over all $\rho (a)$, and the set of distances $\Delta (R)$ is the union of all $\Delta (a)$. Then $\rho (R) = 1$ if and only if $\Delta (R) = \emptyset$, and in this case $R$ is said to be half-factorial.

\smallskip
In the last decade, abstract finiteness results for arithmetical
invariants have been derived for large classes of noetherian domains
(see \cite[Theorem 2.11.9]{MR2194494}, or \cite{Ka11,MR2123000} for
recent progress). If the noetherian domain is integrally closed,
then it is a Krull domain, and if in addition every divisor class
contains a prime divisor, then methods from additive and
combinatorial number theory allow one to obtain precise results on the
arithmetic (see \cite{MR2547479} for the role of combinatorial number
theory in this context). By a precise result, we mean an explicit
formula, say for the elasticity, in terms of the group invariants of
the class group, or an explicit characterization of the extremal
cases, say  $\rho (R) = 1$, which asks,  in other words, for an
explicit characterization of half-factoriality.

\smallskip
Half-factoriality has been a central topic ever since the beginning
of factorization theory (see the surveys \cite{MR1858159, MR2140686, MR2228388}, and \cite{MR1460780,MR1966521,MR2202351,MR2483834} for some recent results). A
classical result due to Carlitz states that a ring of integers is
half-factorial if and only if its class group has at most two
elements (see \cite{MR0111741}; there are analogous results for Krull
monoids, but for simplicity we restrict our discussion here to rings
of integers). If $R$ is  a ring of integers in an algebraic number
field, then, for almost all elements $a \in R$, we have $\Delta (a) =
\{1\}$, and hence their sets of lengths are arithmetical
progressions with difference $1$ (see \cite[Theorem 9.4.11]{MR2194494}). Precise results of such a type for non-principal
orders are extremely rare. In contrast to the above density result
for principal orders, it is even open whether a non-principal order
contains a single element $a$ with $1 \in \Delta (a)$. In 1984, F.
Halter-Koch gave a characterization of half-factoriality for
non-principal orders in quadratic number fields (see \cite[Theorem 3.7.15]{MR2194494}, or \cite{MR781157}), but the general case remained
wide open (\cite{MR2140704,MR1770474}).

\smallskip
The present paper is devoted to non-principal orders in algebraic number fields and studies half-factoriality and the question whether $1$ occurs in the set of distances. Here is our main result.

\begin{theorem}
\label{main}
Let $\mathcal O$ be a non-principal, locally half-factorial order in an algebraic number field and set $\mathcal P^*=\lbrace\mathfrak p\in\mathfrak X(\mathcal O)\mid\mathfrak p\supset(\mathcal O:\overline{\mathcal O})\rbrace$.\\
\begin{enumerate}
 \item\label{main1} If $|\pic(\mathcal O)|=1$, then $\mathcal O$ is half-factorial.
 \item\label{main2} If $|\pic(\mathcal O)|\geq 3$, then $(\mathsf D(\pic(\mathcal O)))^2\geq\mathsf c(\mathcal O)\geq 3$, $\min\triangle(\mathcal O)=1$, and $\rho(\mathcal O)>1$.
 \item\label{main3} If $|\pic(\mathcal O)|=2$, then $\rho(\mathcal O)\leq 2$, $2\leq\mathsf c(\mathcal O)\leq 4$, and $\min\triangle(\mathcal O)\leq 1$.\\
If, additionally, all localizations of $\mathcal O$ are finitely primary monoids of exponent $1$, then, setting $k=\#\lbrace\mathfrak p\in\mathcal P^*\mid[\mal{\overline{\mathcal O}_{\mathfrak p}}/\mal{\mathcal O_{\mathfrak p}}]_{\pic(\mathcal O)}=\pic(\mathcal O)\rbrace$, it follows that
\begin{itemize}
 \item $\cmon(\mathcal O)=\mathsf c(\mathcal O)=2+\min\lbrace 2,k\rbrace\in\lbrace 2,3,4\rbrace$;
 \item $\rho(\mathcal O)=\frac{1}{2}\mathsf c(\mathcal O)\in\lbrace 1,\frac{3}{2},2\rbrace$; 
 \item $\triangle(\mathcal O)=[1,\mathsf c(\mathcal O)-2]\subset[1,2]$;
\end{itemize}
and the following are equivalent:
\begin{itemize}
 \item $\cmon(\mathcal O)=2$.
 \item $\mathsf c(\mathcal O)=2$.
 \item $\mathcal O$ is half-factorial.
\end{itemize}
If, additionally, $[\mathfrak p]=\mathbf 0_{\pic(\mathcal O)}$ for all $\mathfrak p\in\mathcal P^*$, then the following is also equivalent:
\begin{itemize}
 \item $\mathsf t(\mathcal O)=2$.
\end{itemize}
\end{enumerate}
In particular, $\min\triangle(\mathcal O)\leq 1$ always holds.
\end{theorem}

\medskip
Recall that $\mathcal O$ is called locally half-factorial if the
localizations $\mathcal O_{\mathfrak p}$ are half-factorial for all
non-zero prime ideals $\mathfrak p $ of $\mathcal O$. It is the
standing conjecture that all half-factorial orders are locally
half-factorial, and this holds true for orders in quadratic and
cubic number fields. In particular, the above theorem yields the
classical result of F. Halter-Koch as a corollary (see \autoref{4.7}). We will see that the most difficult case is $|\pic (\mathcal O)| = 2$, and that the other ones are quite easy.

\medskip
We briefly sketch our approach. We proceed in two steps. The first
one is fairly standard in this area. We consider $\mathcal O$, the
set of invertible ideals $\mathcal I^* ( \mathcal O)$, and construct
the associated $T$-block monoid $\mathcal B (G, T, \iota)$. Then all
questions under consideration can be studied in the $T$-block monoid
instead of in $\mathcal O$ (see \autoref{sec:proof} for this transfer
process). The second step contains the main new idea behind
the present progress. In a series of recent papers (see for example
\cite{garger10,MR2494887,MR2243561}), arithmetical invariants
of a monoid have been characterized in abstract semigroup
theoretical terms, such as the  monoid of relations and
presentations. Of course, these semigroup theoretical invariants are
far beyond reach in the case of non-principal orders. However, the
$T$-block monoid $\mathcal B (G, T, \iota)$ has such simple
constituents that these characterizations can be used to determine
the arithmetical invariants exactly. These local results can be put together
to get information for the whole $T$-block monoid $\mathcal B (G, T,
\iota)$, and then all this is shifted to $\mathcal O$. Our crucial
technical results are formulated in \autoref{3.17} and \autoref{3.18},
which are based on \cite{phil10} and \cite{phil11a}.

\medskip
In \autoref{sec:pre}, we recall the relevant concepts from factorization theory and some abstract concepts from semigroup theory. In \autoref{sec:proof}, we introduce $T$-block monoids and the associated transfer homomorphisms. The main work is to prove the already mentioned technical results \autoref{3.17} and \autoref{3.18}. The proof of our main result, \autoref{main}, will be given at the end of \autoref{sec:proof}.

\bigskip
\section{Preliminaries}
\label{sec:pre}
\bigskip

In this note, our notation and terminology will be consistent with \cite{MR2194494}. Let $\N$ denote the set of positive integers and let $\N_0=\N\uplus\lbrace 0\rbrace$. For integers $n,\,m\in\Z$, we set $[n,m]=\lbrace x\in\Z\mid n\leq x\leq m\rbrace$. By convention, the supremum of the empty set is zero and we set $\frac{0}{0}=1$. The term ``monoid'' always means a commutative, cancellative semigroup with unit element. We will write all monoids multiplicatively. For a monoid $H$, we denote by $H^\times$ the set of invertible elements of $H$. We call $H$ reduced if $H^\times=\lbrace 1\rbrace$ and call $H_{\mathrm{red}}=H/H^\times$ the reduced monoid associated with $H$. Of course, $H_{\mathrm{red}}$ is always reduced.
Note that the arithmetic of $H$ is determined by $H_{\mathrm{red}}$, and therefore we can restrict to reduced monoids whenever convenient. We denote by $\mathcal A(H)$ the \emph{set of atoms} of $H$, by $\mathcal A(H_{\mathrm{red}})$ the set of atoms of the associated reduced monoid $H_{\mathrm{red}}$, by $\mathsf Z(H)=\mathcal F(\mathcal A(H_{\mathrm{red}}))$ the free (abelian) monoid with basis $\mathcal A(H_{\mathrm{red}})$, and by $\pi_H:\mathsf Z(H)\rightarrow H_{\mathrm{red}}$ the unique homomorphism such that $\pi_H|\mathcal A(H_{\mathrm{red}})=\mathrm{id}$. We call $\mathsf Z(H)$ the \emph{factorization monoid} and $\pi_H$ the \emph{factorization homomorphism} of $H$. For $a\in H$, we denote by $\mathsf Z(a)=\pi_H^{-1}(a\mal H)$ the \emph{set of factorizations} of $a$ and denote by $\mathsf L(a)=\lbrace |z|\mid z\in\mathsf Z(a)\rbrace$ the \emph{set of lengths} of $a$, where $|\cdot|$ is the ordinary length function in the free monoid $\mathsf Z(H)$. In this terminology, a monoid $H$ is called \emph{half-factorial} if $|\mathsf L(a)|=1$ for all $a\in H\setminus\mal H$---this coincides with the classical definition of being half-factorial, since then every two factorizations of an element have the same length---and \emph{factorial} if $|\mathsf Z(a)|=1$ for all $a\in H\setminus\mal H$.

With all these notions at hand, for $a\in H$, we set
\[
\rho(a)=\frac{\sup\mathsf L(a)}{\min\mathsf L(a)}\:\mbox{ and call }\:
\rho(H)=\sup\lbrace\rho(a)\mid a\in H\rbrace\:\mbox{ the \emph{elasticity} of }H.
\]
Note that $H$ is half-factorial if and only if $\rho(H)=1$.

For two factorizations $z,\,z'\in\mathsf Z(H)$, we call
\[
\mathsf d(z,z')=\max\left\lbrace\left|\frac{z}{\gcd(z,z')}\right|,\left|\frac{z'}{\gcd(z,z')}\right|\right\rbrace
\quad\mbox{the \emph{distance} between }z\mbox{ and }z'.
\]

\begin{definition}
Let $H$ be an atomic monoid and $a\in H$.
\begin{enumerate}
\item Factorizations $z_0,\ldots,z_n\in\mathsf Z(a)$ with $n\in\N$ and $\mathsf d(z_{i-1},z_i)\leq N$ for some $N\in\N$ and $i\in [1,n]$ are called
\begin{itemize}
\item an \emph{$N$-chain} concatenating $z_0$ and $z_n$ (in $\mathsf Z(H)$).
\item a \emph{monotone $N$-chain} concatenating $z_0$ and $z_n$ (in $\mathsf Z(H)$) if $|z_{i-1}|\leq |z_i|$ for all $i\in [1,n]$.
%\item an \emph{equal-length $N$-chain} concatenating $z_0$ and $z_n$ (in $\mathsf Z(H)$) if $|z_{i-1}|=|z_i|$ for all $i\in [1,n]$.
\end{itemize}
\item The
\begin{itemize}
\item \emph{catenary degree $\cat(a)$}
\item \emph{monotone catenary degree $\cmon(a)$}
%\item \emph{equal catenary degree $\ceq(a)$}
\end{itemize}
denotes the smallest $N\in\mathsf N_0\cup\lbrace\infty\rbrace$ such that for all $z,\,z'\in\mathsf Z(a)$ there is
\begin{itemize}
\item an $N$-chain concatenating $z$ and $z'$.
\item a monotone $N$-chain concatenating $z$ and $z'$.
%\item an equal-length $N$-chain concatenating $z$ and $z'$.
\end{itemize}
Then we call
\begin{itemize}
\item $\cat(H)=\sup\lbrace\cat(a)\mid a\in H\rbrace$ the \emph{catenary degree} of $H$.
\item $\cmon(H)=\sup\lbrace\cmon(a)\mid a\in H\rbrace$ the \emph{monotone catenary degree} of $H$.
%\item $\ceq(H)=\sup\lbrace\ceq(a)\mid a\in H\rbrace$ the \emph{equal-length catenary degree} of $H$.
\end{itemize}
\end{enumerate}
\end{definition}

%Note that $\sup\lbrace\cat(H),\ceq(H)\rbrace\leq\cmon(H)$.
Note that $\cat(H)\leq\cmon(H)$ and that equality holds if $H$ is half-factorial by \cite[Lemma 4.4.1]{phil11a}.

\begin{definition}
\label{def-tame}
Let $H$ be a reduced atomic monoid.
\begin{enumerate}
\item For $a\in H$ and $x\in\mathsf Z(H)$, let $\mathsf t(a,x)$ denote the smallest $N\in\mathbb N_0\cup\lbrace\infty\rbrace$ with the following property:
\begin{itemize}
 \item[] If $\mathsf Z(a)\cap x\mathsf Z(H)\neq\emptyset$ and $z\in\mathsf Z(a)$, then there exists some $z'\in\mathsf Z(a)\cap x\mathsf Z(H)$ such that $\mathsf d(z,z')\leq N$.
\end{itemize}
For subsets $H'\subset H$ and $X\subset\mathsf Z(H)$, we define
\[
 \mathsf t(H',X)=\sup\lbrace\mathsf t(a,x)\mid a\in H',\,x\in X\rbrace,
\]
and we define $\mathsf t(H)=\mathsf t(H,\mathcal A(H))$. This is called the \emph{tame degree} of $H$.
\item If $\mathsf t(H)<\infty$, then we call $H$ \emph{tame}
\end{enumerate}
\end{definition}

Here we recall the exact definitions from \cite[Definition 2.3]{phil11a} for the $\Req$-relation and the $\mathcal R$-relation, the latter one coinciding with the one given in \cite[Section 3]{phil10}.

\begin{definition}
Let $H$ be a reduced atomic monoid and $a\in H$.
\begin{enumerate}
\item Factorizations $z_0,\ldots,z_n\in\mathsf Z(a)$ with $n\in\N$ and $\gcd(z_{i-1},z_i)\neq 1$ for all $i\in [1,n]$ are called
\begin{itemize}
\item an \emph{$\mathcal R$-chain} concatenating $z_0$ and $z_n$ (in $\mathsf Z(H)$).
\item a \emph{monotone $\mathcal R$-chain} concatenating $z_0$ and $z_n$ (in $\mathsf Z(H)$) if $|z_{i-1}|\leq|z_i|$ for all $i\in [1,n]$.
\item an \emph{equal-length $\mathcal R$-chain} concatenating $z_0$ and $z_n$ (in $\mathsf Z(H)$) if $|z_{i-1}|=|z_i|$ for all $i\in [1,n]$.
\end{itemize}
\item Two elements $z,\,z'\in\mathsf Z(H)$ are
\begin{itemize}
\item \emph{$\mathcal R$-related}
\item \emph{$\Req$-related}
\end{itemize}
if there is an
\begin{itemize}
\item $\mathcal R$-chain
\item equal-length $\mathcal R$-chain
\end{itemize}
 concatenating $z$ and $z'$. We then write $z\approx z'$ respectively $z\equ z'$.
\end{enumerate}
Note that with the above definitions $\approx$ and $\equ$ are congruences on $\mathsf Z(H)\times\mathsf Z(H)$.
\end{definition}

\begin{definition}
Let $H\subset D$ be monoids.
\begin{enumerate}
 \item We call $H\subset D$ \emph{saturated} or, equivalently, \emph{a saturated submonoid} if, for all $a,\,b\in H$, $a\mid b$ in $D$ already implies that $a\mid b$ in $H$.
 \item If $H\subset D$ is a saturated submonoid, then we set $D/H=\lbrace a\mathsf q(H)\mid a\in D\rbrace$ and $[a]_{D/H}=a\mathsf q(H)$ and we call $\mathsf q(D)/\mathsf q(H)=\mathsf q(D/H)$ the \emph{class group} of $H$ in $D$.
\end{enumerate}
\end{definition}

\begin{definition}
Let $H$ be an atomic monoid. We call
\begin{align*}
\sim_H &=\lbrace(x,y)\in\mathsf Z(H)\times\mathsf Z(H)\mid\pi(x)=\pi(y)\rbrace
&\mbox{the \emph{monoid of relations} of $H$}, \\
%\sim_{H,\mathrm{eq}} &=\lbrace (x,y)\in\sim_H\mid |x|=|y|\rbrace
%&\mbox{the \emph{monoid of equal-length relations} of $H$}, \\
\sim_{H,\mathrm{mon}} &=\lbrace (x,y)\in\sim_H\mid |x|\leq|y|\rbrace
&\mbox{the \emph{monoid of monotone relations} of $H$},
\end{align*}
and, for $a\in H$, we set
\begin{align*}
\mathcal A_a(\sim_H) &= \mathcal A(\sim_H)\cap(\mathsf Z(a)\times\mathsf Z(a)), \\
%\mathcal A_a(\sim_{H,\mathrm{eq}}) &= \mathcal A(\sim_{H,\mathrm{eq}})\cap(\mathsf Z(a)\times\mathsf Z(a)), \\
\mathcal A_a(\sim_{H,\mathrm{mon}}) &= \mathcal A(\sim_{H,\mathrm{mon}})\cap(\mathsf Z(a)\times\mathsf Z(a)).
\end{align*}
\end{definition}

By \cite[Lemma 11]{phil10}, $\sim_H\subset\mathsf Z(H)\times\mathsf Z(H)$ is a saturated submonoid of a free monoid and thus a Krull monoid by \cite[Theorem 2.4.8.1]{MR2194494}. Unfortunately, $\sim_{H,\mathrm{mon}}\subset\sim_H$ is not saturated.

\medskip
\subsection*{Notions for integral domains}

For an integral domain $R$, we set $\punkt R=R\setminus\lbrace 0\rbrace$ for the commutative, cancellative monoid of non-zero elements of $R$. Additionally, all notions, which were introduced for monoids, are used for domains, too; for example, we write $\mathcal A(R)$ instead of $\mathcal A(\punkt R)$ for the set of atoms.

\begin{definition}
Let $R$ be an integral domain and $K=\mathsf q(R)$ the quotient field of $R$.
\begin{enumerate}
 \item We call $\spec(R)$ the set of all prime ideals of $R$.
 \item We set
\[
 \mathfrak X(R)=\lbrace\mathfrak p\in\spec(R)\mid\mathfrak p\neq 0\mbox{ and }\mathfrak p\mbox{ is minimal}\rbrace
\]
for the \emph{set of minimal prime ideals} of $R$.
 \item Let $L\supset K$ be a field extension. We call $b\in L$ \emph{integral} over $R$ if there is a monic polynomial $f\in R[X]$ such that $f(b)=0$.
 \item We call 
\[
\mathsf{cl}_L(R)=\lbrace b\in L\mid b\mbox{ is integral over }R\rbrace
\quad\mbox{the \emph{integral closure} of $R$ in $L$}
\]
and we set $\bar R=\mathsf{cl}_K(R)$ for the integral closure of an integral domain (in its quotient field).
 \item For non-empty subsets $X,\,Y\subset K$, we define
\[
 (Y:X)=(Y:_K X)=\lbrace a\in K\mid aX\subset Y\rbrace
\mbox{ and } X^{-1}=(R:X).
\]
We denote by $\mathcal I(R)$ the \emph{set of all ideals} of $R$ and we call an ideal $\mathfrak a\in\mathcal I(R)$ \emph{invertible} if $\mathfrak a\mathfrak a^{-1}=R$. Then we denote by $\mathcal I^*(R)$ the \emph{set of all invertible ideals} of $R$.
\end{enumerate}
\end{definition}

\begin{definition}
\label{def:loc-half}
A one-dimensional noetherian domain $R$ is called \emph{locally half-factorial} if $\mathcal I^*(R)$ is half-factorial.\\
Note that this notion of being locally half-factorial does not coincide with the one defined in \cite{MR1687334} but coincides with what is called purely locally half-factorial there.\\
By \cite[Theorem 3.7.1]{MR2194494}, we have $\mathcal I^*(R)\cong\coprod_{\mathfrak p\in\mathfrak X(R)}(\punkt{R_{\mathfrak p}})_{\red}$. Thus $\mathcal I^*(R)$ is half-factorial if and only if $(\punkt{R_{\mathfrak p}})_{\red}$ is half-factorial for all $\mathfrak p\in\mathfrak X(R)$.
\end{definition}

\bigskip
\section{Proof of the main theorem} \label{sec:proof}
\bigskip

Before we can prove the main theorem, we need to gather some additional tools, among these the notion of $T$-block monoids over finite abelian groups, the concept of transfer homomorphism, and some monoid theoretic preliminaries. Once all these things at hand, we will exploit the results from \cite{phil10} and \cite{phil11a} to give the final proof of the main theorem.

\medskip
\subsection*{$T$-block monoids and transfer principles}
First, we briefly fix the notation for $T$-block monoids, which are a generalization of the concept of block monoids, and therefore have their origin in zero-sum theory; for a detailed exposition of these aspects, the reader is referred to \cite[Chapter 3]{MR2194494} . Let $G$ be an additively written finite abelian group, $G_0\subset G$ a subset, and $\mathcal F(G_0)$ the free abelian monoid with basis $G_0$. The elements of $\mathcal F(G_0)$ are called \emph{sequences} over $G_0$. If a sequence $S\in\mathcal F(G_0)$ is written in the form $S=g_1\mdots g_l$, we tacitly assume that $l\in\N_0$ and $g_1,\ldots,g_l\in G_0$. For a sequence $S=g_1\mdots g_l$, we call
\begin{itemize}
 \item[] $|S|=l$ the \emph{length} of $S$ and
 \item[] $\sigma(S)=\sum_{i=1}^lg_i\in G$ the \emph{sum} of $S$.
\end{itemize}
The sequence $S$ is called a \emph{zero-sum sequence} if $\sigma(S)=\mathbf 0$. We set
\[
\mathcal B(G_0)=\lbrace S\in\mathcal F(G_0)\mid\sigma(S)=0\rbrace
\quad\mbox{for the \emph{block monoid} over }G_0
\]
and $\mathcal A(G_0)=\mathcal A(\mathcal B(G_0))$ for its set of atoms.

Then, the \emph{Davenport constant} $\mathsf D(G_0)\in\N$ is defined to be the supremum of all lengths of sequences in $\mathcal A(G_0)$.

Now we are able to give the precise definition of $T$-block monoids.

\begin{definition}
Let $G$ be an additive abelian group, $T$ a monoid, $\iota:T\rightarrow G$ a homomorphism, and $\sigma:\mathcal F(G)\rightarrow G$ the unique homomorphism such that $\sigma(g)=g$ for all $g\in G$. Then we call
\[
\mathcal B(G,T,\iota)=\lbrace St\in\mathcal F(G)\times T\mid\sigma(S)+\iota(t)=\mathbf 0\rbrace
\quad\mbox{the \emph{$T$-block monoid} over $G$ defined by $\iota$}.
\]
If $T=\lbrace 1\rbrace$, then $\mathcal B(G,T,\iota)=\mathcal B(G)$ is the block monoid over $G$.
\end{definition}

Next we give the transfer homomorphism and, then, we use it to transport questions on the arithmetic of our investigated monoids to $T$-block monoids.

\begin{definition}
\label{def:transfer}
A monoid homomorphism $\theta:H\rightarrow B$ is called a \emph{transfer homomorphism} if it has the following properties:
\begin{itemize}
\item[$\mathbf{T1}$] $B=\theta(H)\mal B$ and $\theta^{-1}(\mal B)=\mal H$.
\item[$\mathbf{T2}$] If $a\in H,\,r,\,s\in B$ and $\theta(a)=rs$, then there exist $b,\,c\in H$ such that $\theta(b)\sim r$, $\theta(c)\sim s$, and $a=bc$.
\end{itemize}
\end{definition}

\begin{definition}
Let $\theta:H\rightarrow B$ be a transfer homomorphism of atomic monoids and $\bar\theta:\mathsf Z(H)\rightarrow\mathsf Z(B)$ the unique homomorphism satisfying $\bar\theta(u\mal H)=\theta(u)\mal B$ for all $u\in\mathcal A(H)$. We call $\bar\theta$ the extension of $\theta$ to the factorization monoids.\\
For $a\in H$, the \emph{catenary degree in the fibers} $\mathsf c(a,\theta)$ denotes the smallest $N\in\N_0\cup\lbrace\infty\rbrace$ with the following property:
\begin{itemize}
 \item[] For any two factorizations $z,\,z'\in\mathsf Z(a)$ with $\bar\theta(z)=\bar\theta(z')$, there exists a finite sequence of factorizations $(z_0,z_1,\ldots,z_k)$ in $\mathsf Z(a)$ such that $z_0=z$, $z_k=z'$, $\bar\theta(z_i)=\bar\theta(z)$, and $\mathsf d(z_{i-1},z_i)\leq N$ for all $i\in[1,k]$; that is, $z$ and $z'$ can be concatenated by an $N$-chain in the fiber $\mathsf Z(a)\cap\bar\theta^{-1}((\bar\theta(z)))$.
\end{itemize}
Also, $\mathsf c(H,\theta)=\sup\lbrace\mathsf c(a,\theta)\mid a\in H\rbrace$ is called the \emph{catenary degree in the fibers} of $H$.
\end{definition}

\begin{lemma}
\label{3.5}
Let $D$ be an atomic monoid, $P\subset D$ a set of prime elements, and $T\subset D$ an atomic submonoid such that $D=\mathcal F(P)\times T$. Let $H\subset D$ be a saturated atomic submonoid, let $G=\mathsf q(D/H)$ be its class group, let $\iota:T\rightarrow G$ be a homomorphism defined by $\iota(t)=[t]_{D/H}$, and suppose each class in $G$ contains some prime element from $P$. Then
\begin{enumerate}
\item\label{3.5.1} The map $\beta:H\rightarrow\mathcal B(G,T,\iota)$, given by $\beta(pt)=[p]_{D/H}+\iota(t)=[p]_{D/H}+[t]_{D/H}$, is a transfer homomorphism onto the $T$-block monoid over $G$ defined by $\iota$, and $\mathsf c(H,\beta)\leq 2$
\item\label{3.5.2} The following inequalities hold:
\begin{eqnarray*}
\mathsf c(\mathcal B(G,T,\iota))\leq &\mathsf c(H) &\leq\max\lbrace\mathsf c(\mathcal B(G,T,\iota)),\mathsf c(H,\beta)\rbrace,\\ 
\cmon(\mathcal B(G,T,\iota))\leq &\cmon(H) &\leq\max\lbrace\cmon(\mathcal B(G,T,\iota)),\mathsf c(H,\beta)\rbrace,\mbox{ and}\\
\mathsf t(\mathcal B(G,T,\iota))\leq &\mathsf t(H) &\leq\mathsf t(\mathcal B(G,T,\iota))+\mathsf D(G)+1.
\end{eqnarray*}
In particular, the equality $\mathsf c(H)=\mathsf c(\mathcal B(G,T,\iota))$ holds if $\mathsf c(\mathcal B(G,T,\iota))\geq 2$, and the equality $\cmon(H)=\cmon(\mathcal B(G,T,\iota))$ holds if $\cmon(\mathcal B(G,T,\iota))\geq 2$.
\item\label{3.5.3} $\mathcal L(H)=\mathcal L(\mathcal B(G,T,\iota))$, $\triangle(H)=\triangle(\mathcal B(G,T,\iota))$, $\min\triangle(H)=\min\triangle(\mathcal B(G,T,\iota))$, and $\rho(H)=\rho(\mathcal B(G,t,\iota))$.
\item\label{3.5.4} We set $\mathcal B=\lbrace S\in\mathcal B(G,T,\iota)\mid\mathbf 0\nmid S\rbrace$. Then $\mathcal B$ and $\mathcal B(G,T,\iota)$ have the same arithmetical properties, and
\begin{eqnarray*}
\mathsf c(\mathcal B)\leq &\mathsf c(H) &\leq\max\lbrace\mathsf c(\mathcal B),\mathsf c(H,\beta)\rbrace,\\ 
\cmon(\mathcal B)\leq &\cmon(H) &\leq\max\lbrace\cmon(\mathcal B),\mathsf c(H,\beta)\rbrace,\mbox{ and}\\
\mathsf t(\mathcal B)\leq &\mathsf t(H) &\leq\mathsf t(\mathcal B)+\mathsf D(G)+1.
\end{eqnarray*}
In particular, the equality $\mathsf c(H)=\mathsf c(\mathcal B)$ holds if $\mathsf c(\mathcal B)\geq 2$, and the equality $\cmon(H)=\cmon(\mathcal B)$ holds if $\cmon(\mathcal B)\geq 2$.\\
Additionally,
$\mathcal L(H)=\mathcal L(\mathcal B)$, $\triangle(H)=\triangle(\mathcal B)$, $\min\triangle(H)=\min\triangle(\mathcal B)$, and $\rho(H)=\rho(\mathcal B)$.
\end{enumerate}
\end{lemma}
\begin{proof}
\mbox{}
\begin{enumerate}
\item Follows by \cite[Proposition 3.2.3.3 and Proposition 3.4.8.2]{MR2194494}.
\item The assertion for the catenary degree follows by \cite[Theorem 3.2.5.5]{MR2194494}, the assertion for the monotone catenary degree by \cite[Lemma 3.2.6]{MR2194494}, and the assertion for the tame degree by \cite[Theorem 3.2.5.1]{MR2194494}.
\item Follows by \cite[Proposition 3.2.3.5]{MR2194494}.
\item Since $\mathbf 0\in\mathcal B(G,T,\iota)$ is a prime element, it defines a partition $\mathcal B(G,T,\iota)=[\mathbf 0]\times\mathcal B$ with $\mathcal B=\lbrace S\in\mathcal B(G,T,\iota)\mid\mathbf 0\nmid S\rbrace$. Thus all studied arithmetical invariants coincide for $\mathcal B$ and $\mathcal B(G,T,\iota)$. Now the assertions follow from part~\ref{3.5.2} and part~\ref{3.5.3}.
\qedhere
\end{enumerate}
\end{proof}

\begin{lemma}
\label{3.6}
Let $D$ be an atomic monoid, $P\subset D$ a set of prime elements, and $T\subset D$ an atomic submonoid such that $D=\mathcal F(P)\times T$. Let $H\subset D$ be a saturated atomic submonoid, $G=\mathsf q(D/H)$ its class group, and suppose each class in $G$ contain some $p\in P$.
\begin{enumerate}
\item\label{3.6.1} If $|G|\geq 3$, then $\min\triangle(H)=1$, $\rho(H)>1$, $\mathsf c(H)\geq 3$.
\item\label{3.6.2} $\rho(H)\leq\mathsf D(G)\rho(T)$.
\end{enumerate}
\end{lemma}
\begin{proof}
We define a homomorphism $\iota:T\rightarrow G$ by $\iota(t)=[t]_{D/H}$ and write $\mathcal B(G,T,\iota)$ for the $T$-block monoid over $G$ defined by $\iota$.
\begin{enumerate}
\item Then $\mathcal B(G)\subset\mathcal B(G,T,\iota)$ is a divisor-closed submonoid. By \cite[Theorem 6.7.1.2]{MR2194494}, we have $\min\triangle(G)=1$, and thus $\min\triangle(\mathcal B(G,T,\iota))=1$ and $\mathsf c(\mathcal B(G,T,\iota))\geq 3$ by \cite[Theorem 1.6.3]{MR2194494}. Now the assertions follow by \autoref{3.5}.\ref{3.5.2} and \autoref{3.5}.\ref{3.5.3}.
\item By \cite[Proposition 3.4.7.5]{MR2194494}, we have $\rho(\mathcal B(G,T,\iota))\leq\mathsf D(G)\rho(T)$. Now the assertion again follows by \autoref{3.5}.\ref{3.5.2}.
\qedhere
\end{enumerate}
\end{proof}

\begin{definition}
\label{def:fp}
A monoid $H$ is called \emph{finitely primary} if there exist $s,\,k\in\N$ and a factorial monoid $F=[p_1,\ldots,p_s]\times\mal F$ with the following properties:
\begin{itemize}
 \item $H\setminus\mal H\subset p_1\mdots p_s F$,
 \item $(p_1\mdots p_s)^kF\subset H$, and
 \item $(p_1\mdots p_s)^iF\not\subset H$ for $i\in[0,k)$.
\end{itemize}
If this is the case, then we call $H$ a \emph{finitely primary monoid} of \emph{rank} $s$ and \emph{exponent} $k$.\\
Note that this definition is slightly more restrictive than the one given in \cite[Definition 2.9.1]{MR2194494}. By \cite[Theorem 2.9.2.1]{MR2194494}, we get $F=\widehat H$, and therefore $H\subset\widehat H=[p_1,\ldots,p_s]\times\mal{\widehat H}\subset\mathsf q(H)$.\\
Then, for $i\in[1,s]$, we denote by $\mathsf v_{p_i}:\mathsf q(H)\rightarrow\Z$ the $p_i$-adic valuation of $\mathsf q(H)$.\\
Now let $H\subset\widehat H=[p]\times\mal{\widehat H}$ be a finitely primary monoid of rank $1$ and exponent $k$. Then we set $\mathcal U_i(H)=\lbrace u\in\mal{\widehat H}\mid p^iu\in H\rbrace$ for $i\in\N_0$.
\end{definition}

As a first observation, we find
\[
\mathcal U_i(H)=
\begin{cases}
\mal H & i=0 \\
\mal{\widehat H} & i\geq k
\end{cases}
\quad\mbox{and}\quad
\mathcal U_i(H)\mathcal U_j(H)\subset\mathcal U_{i+j}(H)
\mbox{ for all }i,j\in\N_0.
\]

\begin{definition}
\label{def:ek}
Let $s\in\N$, $\mathbf e=(e_1,\ldots,e_s)\in\N^s$, $\mathbf k=(k_1,\ldots,k_s)\in\N^s$, and $H\subset\widehat H=[p_1,\ldots,p_s]\times\widehat H^\times$ be a finitely primary monoid of rank $s$ and exponent $\max\lbrace k_1,\ldots,k_s\rbrace$.\\
Then $H$ is a \emph{monoid of type $(\mathbf e,\mathbf k)$} if
\begin{itemize}
 \item $\mathsf v_{p_i}(H)=e_i\N_0\cup\N_{\geq k_i}$ for all $i=[1,s]$ and
 \item $p_1^{k_1}\cdot\ldots\cdot p_s^{k_s}\widehat H\subset H$.
\end{itemize}
If $s=1$, i.e., $\mathbf e=(e)\in\N$ and $\mathbf k=(k)\in\N$, then we say that $H$ is a monoid of type $(e,k)$ instead of $(\mathbf e,\mathbf k)$.
\end{definition}

\begin{lemma}
\label{3.9}
Let $H\subset\widehat H=[p_1,\ldots,p_s]\times\mal{\widehat H}$ be a reduced finitely primary monoid of rank $s$ and exponent $k$.
\begin{enumerate}
\item \label{3.9.1} The following statements are equivalent:
\begin{enumerate}
\item \label{3.9.1.1} $H$ is half-factorial.
\item \label{3.9.1.2} $H$ is of rank $1$ and $\mathsf v_{p_1}(\mathcal A(H))=\lbrace 1\rbrace$.
\item \label{3.9.1.3} $H$ is of rank $1$ and $(\mathcal U_1(H))^l=\mathcal U_l(H)$ for all $l\in\N$.
\end{enumerate}
If any of these conditions hold, then $\mathcal A(H)=\lbrace p_1\eps\mid\eps\in\mathcal U_1(H)\rbrace$, $(\mathcal U_1(H))^k=\wmal H$, and $H$ is a monoid of type $(1,k)$.
\item \label{3.9.2} If $H$ is a half-factorial monoid of type $(1,k)$ and $a_1,\ldots,a_{k+1},\,b\in\mathcal A(H)$, then there are some $b_1,\ldots,b_k\in\mathcal A(H)$ such that $a_1\mdots a_{k+1}=bb_1\mdots b_k$.\\
In particular, $\cmon(H)=\mathsf c(H)\leq\mathsf t(H)\leq k+1$.
\end{enumerate}
\end{lemma}
\begin{proof}
\mbox{}
\begin{enumerate}
\item
\textbf{\ref{3.9.1.1}$\,\Rightarrow\,$\ref{3.9.1.2}.}
If $H$ is of rank $s\geq 2$, then we find $\rho(H)=\infty$ by \cite[Theorem 3.1.5.2 (b)]{MR2194494}. Thus $H$ is of rank $1$.
Now we prove $\#\mathsf v_{p_1}(\mathcal A(H))=1$. [Then the assertion follows since $\mathsf v_{p_1}(\mathcal A(H))=\lbrace n\rbrace$ with $n\geq 2$ implies $\mathsf v_{p_1}(H)=n\N_0\not\supset\N_{\geq k}$, a contradiction.] Suppose $\#\mathsf v_{p_1}(\mathcal A(H))>1$. Let $n=\min\mathsf v_{p_1}(\mathcal A(H))$, $m\in\mathsf v_{p_1}(\mathcal A(H))\setminus\lbrace n\rbrace$, and $\eps,\,\eta\in\mal{\widehat H}$ be such that $p_1^n\eps,\,p_1^m\eta\in\mathcal A(H)$. Now we find
\[
(p_1^m\eta)^k=(p_1^n\eps)^k(p_1^{(m-n)k}\eps^{-k}\eta^k).
\]
On the left side there are $k$ atoms and on the right side at least $k+1$---a contradiction to $H$ being half-factorial.\\
\textbf{\ref{3.9.1.2}$\,\Rightarrow\,$\ref{3.9.1.1}.}
Since $\mathsf v_{p_1}(\mathcal A(H))=\lbrace 1\rbrace$, we have $\mathsf L(a)=\lbrace\mathsf v_{p_1}(a)\rbrace$, i.e., $\#\mathsf L(a)=1$ for all $a\in H\setminus\mal H$. Therefore, $H$ is half-factorial.\\
\textbf{\ref{3.9.1.2}$\,\Rightarrow\,$\ref{3.9.1.3}.}
Since $\mathsf v_{p_1}(\mathcal A(H))=\lbrace 1\rbrace$, we have $\mathcal A(H)=\lbrace p_1u\mid u\in\mathcal U_1(H)\rbrace$. Thus, for all $l\in\N$, we have $\mathcal U_l(H)\subset(\mathcal U_1(H))^l$. Since we always have $(\mathcal U_1(H))^l\subset\mathcal U_l(H)$, the assertion follows.\\
\textbf{\ref{3.9.1.3}$\,\Rightarrow\,$\ref{3.9.1.2}.}
Let $l\in\N_{\geq 2}$ and let $\eps\in\mathcal U_l(H)$. By assumption, there are $\eps_1,\ldots,\eps_l\in\mathcal U_1(H)$ such that $(p_1\eps_1)\mdots(p_1\eps_l)=p_1^l\eps$, and therefore $p_1^l\eps\notin\mathcal A(H)$; thus $\mathsf v_{p_1}(\mathcal A(H))=\lbrace 1\rbrace$.\\
Now we prove the additional statement. $\mathcal A(H)=\lbrace p_1\eps\mid\eps\in\mathcal U_1(H)\rbrace$ has already been shown and $(\mathcal U_1(H))^k=\mathcal U_k(H)=\mal{\widehat H}$ is obvious. The last statement follows immediately by considering the definition of a monoid of type $(1,k)$; see \autoref{def:ek}.
\item Let $H\subset [p_1]\times\wmal H=\widehat H$ be a half-factorial monoid of type $(1,k)$ and let $a_1,\ldots,a_{k+1},\,b\in\mathcal A(H)$. Since $H$ is half-factorial, we have $\mathsf c(H)=\cmon(H)$ by \cite[Lemma 4.4.1]{phil11a}. By part~\ref{3.9.1}, we have $\mathcal A(H)=\lbrace p_1\eps\mid\eps\in\mathcal U_1(H)\rbrace$. Then there are $\eps_1,\ldots,\eps_{k+1},\,\eta\in\mathcal U_1(H)$ such that $a_i=p_1\eps_i$ for $i\in [1,k+1]$ and $b=p_1\eta$. Now we find
\[
a_1\mdots a_{k+1}=(p_1\eps_1)\mdots (p_1\eps_{k+1})=(p_1\eta)(p_1^k\eta^{-1}\eps_1\mdots\eps_{k+1}).
\]
By part~\ref{3.9.1}, $(\mathcal U_1(H))^k=\wmal H$, and thus there are $\eta_1,\ldots,\eta_k\in\mathcal U_1(H)$ such that $\eta^{-1}\eps_1\mdots\eps_{k+1}=\eta_1\mdots\eta_k$. Now we finish the proof by setting $b_i=p_1\eta_i$ for $i\in [1,k]$.
\qedhere
\end{enumerate}
\end{proof}

The result of \autoref{3.9}.\ref{3.9.2} is sharp as the following example shows.

\begin{example}
Let $H\subset\widehat H=[p]\times\mal{\widehat H}$ be a half-factorial, reduced, finitely primary monoid of rank $1$ and exponent $k-1$, with $k\geq 2$, such that $\mal{\widehat H}=\mathsf C_k^2=\langle e_1\rangle\times\langle e_2\rangle$ and $\mathcal U_1(H)=\lbrace 1,e_1,e_2\rbrace$.\\
Then $\mathsf c(H)=k$.
\end{example}
\begin{proof}
By \autoref{3.9}.\ref{3.9.2} we find $\mathsf c(H)\leq k$; thus the assertion follows from the equations
\[
(pe_1)^k=(pe_2)^k=p^k
\quad\mbox{and}\quad
e_1^k=e_2^k=1
\quad\mbox{and}\quad
\ord(e_1)=\ord(e_2)=k,
\]
since one cannot construct any shorter steps in between because of the minimality of the order of $e_1$ respectively $e_2$.
\end{proof}

\begin{definition}[cf. {\cite[Definition 3.6.3]{MR2194494}}]
Let $D$ be an atomic monoid.
\begin{enumerate}
 \item If $H\subset D$ is an atomic submonoid, then we define
\[
 \rho(H,D)=\sup\left\lbrace\left.\frac{\min\mathsf L_H(a)}{\min\mathsf L_D(a)}\right| a\in H\setminus\mal D\right\rbrace\in\mathbb R_{\geq 0}\cup\lbrace\infty\rbrace.
\]
 \item Let $H\subset D$ be a submonoid and $G_0=\lbrace[u]_{D/H}|u\in\mathcal A(D)\rbrace\subset\mathsf q(D/H)$. We say that $H\subset D$ is \emph{faithfully saturated} if $H$ is atomic, $H\subset D$ is saturated and cofinal, $\rho(H,D)<\infty$, and $\mathsf D(G_0)<\infty$.
\end{enumerate}
\end{definition}

\begin{lemma}
\label{3.12}
Let $D$ be a half-factorial monoid and $H\subset D$ an atomic saturated submonoid.\\
Then $\rho(H,D)\leq 1$.
\end{lemma}
\begin{proof}
Let $\eps\in\mal D\cap H$. Then $\eps\mid 1$ in $D$, and thus $\eps\mid 1$ in $H$, and therefore $\eps\in\mal H$. Now we find $\rho(H,D)\leq\rho(D)=1$ by \cite[Proposition 3.6.6]{MR2194494}.
\end{proof}

\begin{lemma}
\label{3.13}
Let $D$ be a monoid, $P\subset D$ a set of prime elements, $r\in\N$, and let $D_i\subset\widehat{D_i}=[p_i]\times\mal{\widehat D_i}$ be reduced finitely primary monoids such that $D=\mathcal F(P)\times D_1\times\ldots\times D_r$. Let $H\subset D$ be a saturated submonoid, $G=\mathsf q(D/H)$ its class group, and let $G$ be finite.\\
Then
\begin{enumerate}
 \item\label{3.13.1} $D$ is a reduced BF-monoid.
 \item\label{3.13.2} $H\subset D$ is a faithfully saturated submonoid and $H$ is also a reduced BF-monoid.
\end{enumerate}
\end{lemma}
\begin{proof}\mbox{}
\begin{enumerate}
 \item Since $D$ is the direct product of reduced BF-monoids, $D$ is a reduced BF-monoid.
 \item Since, by part~\ref{3.13.1}, $D$ is a reduced BF-monoid, $H$ is a reduced BF-monoid by \cite[Proposition 3.4.5.5]{MR2194494}. Since $G$ and $r$ are finite, $H\subset D$ is faithfully saturated by \cite[Theorem 3.6.7]{MR2194494}.
\qedhere
\end{enumerate}
\end{proof}

The following lemma offers a refinement of \cite[Theorem 3.6.4]{MR2194494} for faithfully saturated submonoids $H\subset D$ such that $\rho(H,D)=1$. In our application, this new result yields a crucial refinement from $\mathsf c(H)\leq 6$ to $\mathsf c(H)\leq 4$.

\begin{lemma}
\label{3.14}
Let $D$ be a reduced atomic half-factorial monoid, $H\subset D$ a faithfully saturated submonoid with $\rho(H,D)=1$, $G=\mathsf q(D/H)$ its class group, $\mathsf D=\mathsf D(G)$ its Davenport constant, and suppose each class in $G$ contain some $u\in\mathcal A(D)$.\\
Then
\[
\mathsf c(H)\leq\max\left\lbrace\left\lfloor\frac{(\mathsf D+1)}{2}\mathsf c(D)\right\rfloor,\mathsf D^2\right\rbrace.
\]
\begin{enumerate}
 \item \label{3.14.1} $\mathsf c(H)\leq\max\left\lbrace\left\lfloor\frac{(\mathsf D+1)}{2}\mathsf c(D)\right\rfloor,\mathsf D^2\right\rbrace$.
 \item \label{3.14.2} If $a,\,c\in H$ and $x\in\mathsf Z_H(c)$, then
\[
\mathsf t_H(a,x)\leq|x|\left(1+\mathsf D\frac{D-1}{2}\right)+\mathsf D\mathsf t_D(a,\mathsf Z_D(c)).
\]
\end{enumerate}
\end{lemma}
\begin{proof}
We start by developing the same machinery to compare the factorizations in $H$ with those in $D$ as in \cite[Proof of Theorem 3.6.4]{MR2194494}. Let $\pi_H:\mathsf Z(H)\rightarrow H$ and $\pi_D:\mathsf Z(D)\rightarrow D$ be the factorization homomorphisms and let $Y=\pi_D^{-1}(H)\subset\mathsf Z(D)$. Let $f:\mathsf Z(D)\rightarrow D/H$ be defined by $f(z)=[\pi_D(z)]_{D/H}$. Then $f$ is an epimorphism and $Y=f^{-1}(0)$. Now \cite[Proposition 2.5.1]{MR2194494} implies that $Y\subset\mathsf Z(D)$ is saturated, that $Y$ is a Krull monoid, and that $f$ induces an isomorphism $f^*:\mathsf Z(D)/Y\rightarrow D/H$, since $Y\subset\mathsf Z(D)$ is cofinal. By \cite[Theorem 3.4.10.5]{MR2194494}, we have $\mathsf c(Y)\leq\mathsf D$, and by \cite[Proposition 3.4.5.3]{MR2194494} it follows that $|v|\leq\mathsf D$ for all $v\in\mathcal A(Y)$. If $v\in Y$, then there exists a factorization $y\in\mathsf Z_H(\pi_D(v))$ such that $|y|\leq |v|$.\\
If $\tilde z\in Y$ and $z\in\mathsf Z(H)$, then we say that $z$ is induced by $\tilde z$ if $z=z_1\mdots z_m$ and $\tilde z=\tilde z_1\mdots\tilde z_m$, where $\tilde z_j\in\mathcal A(Y)\subset\mathsf Z(D)$, $z_j\in\mathsf Z_H(\pi_D(\tilde z_j))$ and $|z_j|\leq |\tilde z_j|$ for all $j\in[1,m]$. If $z$ is induced by $\tilde z$, then $\pi_H(z)=\pi_D(\tilde z)$ and $|z|\leq |\tilde z|$. By definition, every factorization $\tilde z\in Y$ induces some factorization $z\in\mathsf Z(H)$. Also, if $z$ is induced by $\tilde z$ and $z'$ is induced by $\tilde z'$, then $zz'$ is induced by $\tilde z\tilde z'$.\\
If $x=u_1\mdots u_m\in\mathsf Z(H)$, where $u_j\in\mathcal A(H)$ and $\tilde u_j\in\mathsf Z_D(u_j)$, then $\tilde u_j\in\mathcal A(Y)$ and $|\tilde u_j|\leq\mathsf D$ for all $j\in [1,m]$ by \cite[Proposition 3.4.5.3]{MR2194494}. Hence $x$ is induced by $\tilde x=\tilde u_1\mdots\tilde u_m$, and $|\tilde x|\leq \mathsf D|x|$.\\
We prove the following assertions:
\begin{itemize}
 \item[\textbf{A0}] Let $\tilde z\in Y$ with $\tilde z=a_1\mdots a_mb_1\mdots b_n$, where $a_1,\,\ldots,a_m,\,b_1,\ldots,b_n\in\mathcal A(H)$, $[a_1]_{D/H}=\ldots=[a_m]_{D/H}=\mathbf 0_{D/H}$, and $[b_1]_{D/H},\ldots,[b_n]_{D/H}\neq\mathbf 0_{D/H}$. For any $z\in\mathsf Z(H)$ such that $z$ is induced by $\tilde z$, we have $|z|=m+\left\lfloor\frac{n}{2}\right\rfloor$.
 \item[\textbf{A1}] For any $\tilde z,\,\tilde z'\in Y$, there exist $z,\,z'\in\mathsf Z(H)$ such that $z$ is induced by $\tilde z$, $z'$ is induced by $\tilde z'$, and $\mathsf d(z,z')\leq\left\lfloor\frac{\mathsf D+1}{2}\mathsf d(\tilde z,\tilde z')\right\rfloor$.
 \item[\textbf{A2}] If $a\in H$, $\tilde z\in Y$, and $z,\,z'\in\mathsf Z_H(a)$ are both induced by $\tilde z$, then there exists a $\mathsf D^2$-chain of factorizations in $\mathsf Z_H(a)$ concatenating $z$ and $z'$.
\end{itemize}
\begin{proof}[Proof of A0]
Let $\tilde z\in Y$ with $\tilde z=a_1\mdots a_mb_1\mdots b_n$, where $a_1,\,\ldots,a_m,\,b_1,\ldots,b_n\in\mathcal A(H)$, $[a_1]_{D/H}=\ldots=[a_m]_{D/H}=0$, and $[b_1]_{D/H}=\ldots=[b_n]_{D/H}\neq\mathsf 0_{D/H}$. Let now $z\in\mathsf Z(H)$ be induced by $\tilde z$. We have $a_i\in\mathcal A(H)$ for all $i\in [1,m]$ and---after renumbering if necessary---$b_1\mdots b_{j_1},\,b_{j_1+1}\mdots b_{j_2},\ldots,b_{j_{k-1}+1}\mdots b_{j_k}\in\mathcal A(H)$ for some $k\in\N$ and $1<j_1+1<j_2<j_2+1<\ldots<j_{k-1}+1<j_k<n$ such that $a_1\mdots a_m(b_1\mdots b_{j_1})(b_{j_1+1}\mdots b_{j_2})\mdots(b_{j_{k-1}+1}\mdots b_{j_k})=z$. Then we have $|z|=m+\left\lfloor\frac{n}{2}\right\rfloor$.
\end{proof}
\begin{proof}[Proof of A1]
Suppose that $\tilde z,\,\tilde z'\in Y$, $\tilde w=\gcd(\tilde z,\tilde z')\in\mathsf Z(D)$, $\tilde z=\tilde w\tilde y$, and $\tilde z'=\tilde w\tilde y'$, where $\tilde y,\,\tilde y'\in\mathsf Z(D)$.
By \cite[Proposition 3.4.5.6]{MR2194494}, there exists some $\tilde w_0\in\mathsf Z(D)$ such that $\tilde w_0\mid\tilde w$, $\tilde w_0\tilde y\in Y$, and $|\tilde w_0|\leq(\mathsf D-1)|\tilde y|$. We may assume that there is no $a\in\mathcal A(D)$ with $a\mid\tilde w_0$ and $[a]_{D/H}=0$. We set $\tilde w_1=\tilde w_0^{-1}\tilde w$. Since $\tilde z=\tilde w_1(\tilde w_0\tilde y)\in Y$ and $\tilde w_0\tilde y\in Y$, we obtain $\tilde w_1\in Y$, and since $\tilde z'=\tilde w_1(\tilde w_0\tilde y')\in Y$ it follows that $\tilde w_0\tilde y'\in Y$. Let $v,\,u,\,u'\in\mathsf Z(H)$ be such that $v$ is induced by $\tilde w_0^{-1}\tilde w$, $u$ is induced by $\tilde w_0\tilde y$ and $u'$ is induced by $\tilde w_0\tilde y'$. Then $z=uv$ is induced by $\tilde w\tilde y=\tilde z$, $z'=u'v$ is induced by $\tilde w\tilde y'=\tilde z'$, and, by part \textbf{A1},
\[
\mathsf d(z,z')\leq\max\lbrace |u|,|u'|\rbrace\leq\max\lbrace |\tilde y|,|\tilde y'|\rbrace+\left\lfloor\frac{|\tilde w_0|}{2}\right\rfloor\leq\left\lfloor\frac{\mathsf D+1}{2}\mathsf d(\tilde z,\tilde z')\right\rfloor.
\qedhere
\]
\end{proof}
\begin{proof}[Proof of A2]
For every $\tilde v\in\mathcal A(Y)$, we fix a factorization $\tilde v^*\in\mathsf Z(H)$ which is induced by $\tilde v$, and, for $\bar y=\tilde v_1\mdots\tilde v_s\in\mathsf Z(Y)$, we set $\bar y^*=\tilde v_1^*\mdots\tilde v_s^*\in\mathsf Z(H)$. Then $\bar y^*$ is induced by $\pi_Y(\bar y)$, $|\bar y^*|\leq |\pi_Y(\bar y)|\leq \mathsf D|\bar y|$, and if $\bar y_1,\,\bar y_2\in\mathsf Z(Y)$, then $\mathsf d(\bar y_1^*,\bar y_2^*)\leq\left\lfloor\frac{\mathsf D+1}{2}\mathsf d(\bar y_1,\bar y_2)\right\rfloor$ by \textbf{A1}.\\
Let now $z,\,z'\in\mathsf Z_H(a)$ be both induced by $\tilde z$. Then $\tilde z=\tilde v_1\mdots\tilde v_r=\tilde v_1'\mdots\tilde v_{r'}'$, $z=v_1\mdots v_r$, and $z'=v_1'\mdots v_{r'}'$, where $\tilde v_i,\,\tilde v_i'\in\mathcal A(Y)$, $v_i$ is induced by $\tilde v_i$, and $v'_i$ is induced by $\tilde v_i'$. Since $\bar y=\tilde v_1\mdots\tilde v_r\in\mathsf Z_Y(\tilde z)$, $\bar y'=\tilde v_1'\mdots\tilde v_{r'}'\in\mathsf Z_Y(\tilde z)$, and $\mathsf c(Y)\leq\mathsf D$, there exists a $\mathsf D$-chain $\bar y=\bar y_0,\bar y_1,\ldots,\bar y_l=\bar y'$ in $\mathsf Z_Y(\tilde z)$ concatenating $\bar y$ and $\bar y'$ in $\mathsf Z_Y(\tilde z)$. Then $\bar y_0^*,\bar y_1^*,\ldots,\bar y_l^*$ is a $\mathsf D$-chain in $\mathsf Z_H(a)$ concatenating $\bar y^*$ and $\bar y'^*$. We have $\bar y_0^*=\tilde v_1^*\mdots\tilde v_r^*$, $z=v_1\mdots v_r$, and since both $v_i$ and $v_i^*$ are induced by $\tilde v_i$, it follows that $\max\lbrace |v_i|,|v_i^*|\rbrace\leq |\tilde v_i|\leq \mathsf D$. For $i\in [0,r]$, we set $z_i=\tilde v_1^*\mdots\tilde v_i^*v_{i+1}\mdots v_r\in\mathsf Z_H(a)$. Then $z=z_0,z_1,\ldots,z_r=\bar y^*$ is a $\mathsf D^2$-chain concatenating $z$ and $\bar y^*$. In the same way, we get a $\mathsf D$-chain concatenating $\bar y'^*$ and $z'$. Connecting these three chains, we get a $\mathsf D^2$-chain in $\mathsf Z_H(a)$ concatenating $z$ and $z'$.
\end{proof}
\begin{enumerate}
 \item Assume $a\in H$ and $z,\,z'\in\mathsf Z_H(a)$. Let $\tilde z,\,\tilde z'\in Y$ be such that $z$ is induced by $\tilde z$ and $z'$ is induced by $\tilde z'$. Then $\tilde z,\,\tilde z'\in\mathsf Z_D(a)$, and therefore there exists a $\mathsf c(D)$-chain $\tilde z=\tilde z_0,\tilde z_1,\ldots,\tilde z_l=\tilde z_l'$ in $\mathsf Z_D(a)$. For $i\in [0,l-1]$, \textbf{A1} gives the existence of factorizations $z'_i,\,z''_i\in\mathsf Z_H(a)$ such that $z'_i$ is induced by $\tilde z_i$, $z_i''$ is induced by $\tilde z_{i+1}$, and $\mathsf d(z_i',z_i'')\leq\left\lfloor\frac{\mathsf D+1}{2}\mathsf c(D)\right\rfloor$. By \textbf{A2}, there exist $\mathsf D^2$-chains of factorizations in $\mathsf Z_H(a)$ concatenating $z$ and $z'_0$, $z_i''$ and $z_{i+1}'$ for all $i\in [0,l-1]$, and $z_{l-1}$ and $z'$. Connecting all these chains, we obtain a $\max\left\lbrace\left\lfloor\frac{(\mathsf D+1)}{2}\mathsf c(D)\right\rfloor,\mathsf D^2\right\rbrace$-chain concatenating $z$ and $z'$.
 \item Suppose that $a,\,c\in H$, $x\in\mathsf Z_H(c)$, $z\in\mathsf Z_H(a)$, and $\mathsf Z_H(a)\cap x\mathsf Z(H)\neq\emptyset$. We set $\mathsf t=\mathsf t_D(a,\mathsf Z_D(c))$, and we must prove that there exists some $z'\in\mathsf Z_H(a)\cap x\mathsf Z(H)$ such that
\[
 \mathsf d(z,z')\leq|x|\left(1+\mathsf D\frac{D-1}{2}\right)+\mathsf D\mathsf t.
\]
Let $\tilde x\in Y$ be such that $x$ is induced by $\tilde x$ and $|\tilde x|\leq\mathsf D|x|$. Suppose that $z=u_1\mdots u_m$ and $\tilde z=\tilde u_1\mdots\tilde u_m$, where $u_j\in\mathcal A(H)$ and $\tilde u_j\in\mathsf Z_D(u_j)$ for all $j\in[1,m]$. Then $z$ is induced by $\tilde z$. Since $\mathsf Z_H(a)\cap x\mathsf Z(H)\neq\emptyset$, we obtain $\pi_D(\tilde x)=\pi_H(x)\mid a$, hence $\mathsf Z_D(a)\cap\tilde x\mathsf Z(D)\neq\emptyset$, and therefore there exists some $\tilde z'\in\mathsf Z_D(a)\cap\tilde x\mathsf Z(D)$ such that $\mathsf d(\tilde z,\tilde z')\leq\mathsf t_D(a,\tilde x)\leq\mathsf t$. After renumbering (if necessary), we may assume that
\[
 \gcd(\tilde z,\tilde z')=\prod_{j=1}^k\tilde u_j\prod_{j=k+1}^my_j,\quad\mbox{and we set}\quad\tilde y'=\tilde z'\prod_{j=1}^k\tilde u_j^{-1},
\]
where $k\in[0,m]$, $y_j\in\mathsf Z(D)$, $y_j\mid\tilde u_j$, $y_j\neq\tilde u_j$, and thus $|y_j|\leq|\tilde u_j|-1\leq\mathsf D-1$ for all $j\in[k+1,m]$. Hence we obtain
\[
 \mathsf t\geq\mathsf d(\tilde z,\tilde z')=\mathsf d\left(\prod_{j=k+1}^m\tilde u_jy_j^{-1},\tilde y'\prod_{j=k+1}^my_j^{-1}\right)\geq\max\lbrace m-k,|\tilde y'|-(m-k)(\mathsf D-1)\rbrace,
\]
and therefore $|\tilde y'|\leq\mathsf t+(m-k)(\mathsf D-1)\leq\mathsf t\mathsf D$. After renumbering again (if necessary), we may suppose that $\tilde x_1=\gcd(\tilde u_1\mdots\tilde u_k,\tilde x)=y'_{l+1}\mdots y'_k$, where $l\in[0,k]$, $y_j'\in\mathsf Z(D)$ and $1\neq y_j'\mid\tilde u_j$ for all $j\in[l+1,k]$. Then we have $k-l\leq|\tilde x_1|\leq|\tilde x|\leq\mathsf D|x|$.

Since $\tilde x\mid\tilde z'$, it follows that $\tilde x_1^{-1}\tilde x\mid\tilde x_1^{-1}\tilde z'=\tilde y'(\tilde x_1^{-1}\tilde u_1\mdots\tilde u_k)$, and since $\gcd(\tilde x_1^{-1}\tilde u_1\mdots\tilde u_k,\tilde x_1^{-1}\tilde x)=1$, we deduce $\tilde x_1^{-1}\tilde x\mid\tilde y'$. Hence $\tilde x\mid\tilde y'\tilde x_1\mid\tilde y'\tilde u_{l+1}\mdots\tilde u_k$, and we set
\[
 \tilde y=\tilde x^{-1}\tilde y'\tilde u_{l+1}\mdots\tilde u_k=(\tilde x\tilde u_1\mdots\tilde u_l)^{-1}\tilde z'\in\mathsf Z(D).
\]
Since $\tilde z',\,\tilde x,\,\tilde u_1,\ldots,\tilde u_l\in Y$ and $Y\subset\mathsf Z(D)$ is a saturated submonoid, we get $\tilde y\in Y$. Now we set $\tilde y=\tilde y_1\tilde y_2$ with $\tilde y_1=(\tilde x_1^{-1}\tilde x)^{-1}\tilde y'$ and $\tilde y_2=(\tilde u_{l+1}y_{l+1}'^{-1})\mdots(\tilde u_ky_k'^{-1})$. Let $y\in\mathsf Z(H)$ be induced by $\tilde y$. then $z'=xyu_1\mdots u_l\in\mathsf Z_H(a)\cap x\mathsf Z(H)$ is induced by $\tilde z'$, and $\mathsf d(z,z')=\mathsf d(u_{l+1}\mdots u_m,xy)\leq\max\lbrace m-l,|x|+|y|\rbrace$. Now we start by computing $|y|$. Since, for all $j\in[l+1,k]$, there is no $a\in\mathcal A(D)$ such that $[a]_{D/H}=\mathbf 0_{D/H}$ and $a\mid\tilde u_jy_j^{-1}$, we find using part \textbf{A0}
\[
 |y|\leq |\tilde y_1|+\left\lfloor\frac{|\tilde y_2|}{2}\right\rfloor\leq |\tilde y'|+\left\lfloor\frac{(k-l)(\mathsf D-1)}{2}\right\rfloor\leq\mathsf t\mathsf D+\frac{\mathsf D(\mathsf D-1)}{2}|x|=\mathsf D\left(\mathsf t+\frac{\mathsf D-1}{2}|x|\right).
\]
Furthermore, we have
\[
 m-l=(m-k)+(k-l)\leq\mathsf t+\mathsf D|x|,\quad
 |x|+|y|\leq|x|+\mathsf D\left(\mathsf t+\frac{\mathsf D-1}{2}|x|\right),
\]
and $\mathsf D\geq 2$ implies
\[
 \mathsf t+\mathsf D|x|\leq |x|\left(1+\mathsf D\frac{D-1}{2}\right)+\mathsf D\mathsf t.
\]
Hence we obtain the asserted bound for $\mathsf d(z,z')$.
\qedhere
\end{enumerate}
\end{proof}

\begin{lemma}
\label{3.15}
Let $D$ be a monoid, $P\subset D$ a set of prime elements, $r\in\N$, and let $D_i\subset\widehat{D_i}=[p_i]\times\mal{\widehat D_i}$ be reduced half-factorial but not factorial monoids of type $(1,k_i)$ with $k_i\in\lbrace 1,2\rbrace$ for $i\in[1,r]$ such that $D=\mathcal F(P)\times D_1\times\ldots\times D_r$. Let $H\subset D$ be a saturated submonoid, $G=\mathsf q(D/H)$ its class group, and suppose $G$ is finite with each class in $G$ containing some $p\in P$.\\
Then
\begin{enumerate}
 \item \label{3.15.1} $2\leq\mathsf c(D)=\max\lbrace\mathsf c(D_1),\ldots,\mathsf c(D_r)\rbrace\leq\max\lbrace k_1,\ldots,k_r\rbrace+1\leq 3$ and $D$ is half-factorial.\\
In particular, $\mathsf c(D)=2$ and $\mathsf t(D)=2$ if $k_1=\ldots=k_r=1$.
 \item \label{3.15.2} If $|G|=1$, then $\mathsf c(H)=\mathsf c(D)$, $\mathsf t(H)=\mathsf t(D)$, and $H$ is half-factorial.
 \item \label{3.15.3} If $|G|\geq 3$, then $(\mathsf D(G))^2\geq\mathsf c(H)\geq 3$ and $\min\triangle(H)=1$.
 \item \label{3.15.4} If $|G|=2$, then $\mathsf c(H)\leq 4$ and $\rho(H)\leq 2$.
\end{enumerate}
\end{lemma}
\begin{proof}\  
\begin{enumerate}
 \item By \autoref{3.13}.\ref{3.13.1}, $D$ is atomic. Trivially, we have $\mathsf c(\mathcal F(P))=0$. By \autoref{3.9}.\ref{3.9.2} and the fact that $D_i$ is not factorial, we find $2\leq\mathsf c(D_i)\leq k_i+1\leq 3$ for all $i\in[1,r]$. By \cite[Proposition 1.6.8]{MR2194494}, we find
\begin{align*}
 \mathsf c(D) &= \max\lbrace\mathsf c(\mathcal F(P)),\mathsf c(D_1),\ldots,\mathsf c(D_r)\rbrace \\
&= \max\lbrace \mathsf c(D_1),\ldots,\mathsf c(D_r)\rbrace \\
&=\max\lbrace k_1,\ldots,k_r\rbrace+1 \leq 3.
\end{align*}
Thus the first part of the assertion follows. Since $D$ is the direct product of half-factorial monoids, $D$ is half-factorial by \cite[Proposition 1.4.5]{MR2194494}. We have $\mathsf t(D_i)=2$ if $k_i=1$ for all $i\in [1,r]$ by \autoref{3.9}.\ref{3.9.2}. Now $\mathsf t(D)=2$ follows by \cite[Proposition 1.6.8]{MR2194494}.
 \item Here we have $H=D$ and thus the assertion follows from part~\ref{3.15.1}.
\end{enumerate}
Before the proof of the two remaining parts, we make the following observations. By \autoref{3.13}.\ref{3.13.2}, $H$ is atomic, $H\subset D$ is a faithfully saturated submonoid, and, by \autoref{3.12}, we have $\rho(H,D)\leq 1$.
\begin{enumerate}
 \item[3.] By part~\ref{3.15.1}, we have $\mathsf c(D)\leq 3$, by \autoref{3.6}.\ref{3.6.1}, we have $\min\triangle(H)=1$, and, by \cite[Lemma 1.4.9.2]{MR2194494}, we have $\mathsf D(G)\geq 3$. Using \cite[Theorem 3.6.4.1]{MR2194494}, we find
\[
 3\leq\mathsf D(G)\leq\mathsf c(H)\leq\rho(H,D)\max\lbrace\mathsf c(D),\mathsf D(G)\rbrace\mathsf D(G)=(\mathsf D(G))^2.
\]
 \item[4.] Since $|G|=2$, we have $\mathsf D(G)=2$, and since $D_1\times\ldots\times D_r$ is half-factorial, i.e., $\rho(D_1\times\ldots\times D_r)=1$, we find $\rho(H)\leq 2$ by \autoref{3.6}.\ref{3.6.2}. When we apply \autoref{3.14}.\ref{3.14.1}, we find
\[
\mathsf c(H)\leq\max\left\lbrace\left\lfloor(\mathsf D(G)+1)\frac{\mathsf c(D)}{2}\right\rfloor,\mathsf D(G)^2\right\rbrace\leq\left\lbrace\left\lfloor\frac92\right\rfloor,4\right\rbrace=4.
\qedhere
\]
\end{enumerate}
\end{proof}

For the rest of this section, we define additional shorthand notation. Let $D$ be a monoid, $P\subset D$ a set of prime elements, and $T\subset D$ a submonoid such that $D=\mathcal F(P)\times T$. Let $H\subset D$ be a saturated submonoid, $G=\mathsf q(D/H)=\mathsf q(D)/\mathsf q(H)$ its class group, suppose each $g\in G$ contains some $p\in P$, and let $\mathcal B(G,T,\iota)$ be the $T$-block monoid over $G$ defined by the homomorphism $\iota:T\rightarrow G$, $\iota(t)=[t]_{D/H}$.\\
For a subset $S\subset\mathcal B(G,T,\iota)$ and an element $g\in G$, we set $S_g=S\cap\iota^{-1}(\lbrace g\rbrace)$.

\begin{lemma}
\label{3.16}
Let $D$ be a monoid, $P\subset D$ a set of prime elements, $r\in\N$, and let $D_i\subset\widehat{D_i}=[p_i]\times\mal{\widehat{D_i}}$ be reduced half-factorial monoids of type $(1,k_i)$ with $k_i\in\lbrace 1,2\rbrace$ for all $i\in[1,r]$ such that $D=\mathcal F(P)\times D_1\times\ldots\times D_r$. Let $H\subset D$ be a saturated submonoid, $G=\mathsf q(D/H)$ its class group with $|G|=2$, say $G=\lbrace\mathbf 0,g\rbrace$, suppose each class in $G$ contains some $p\in P$, and define a homomorphism $\iota:D_1\times\ldots\times D_r\rightarrow G$ by $\iota(t)=[t]_{D/H}$.\\
Then we find the following for the atoms of the $(D_1\times\ldots\times D_r)$-block monoid over $G$ defined by $\iota$, i.e., $\mathcal B(G,D_1\times\ldots\times D_r,\iota)$:
\begin{align*}
&\mathcal A(\mathcal B(G,D_1\times\ldots\times D_r,\iota)) \\
&= \lbrace\mathbf 0,g^2\rbrace \\
&\cup \lbrace p_i\eps\mid i\in[1,r],\,\eps\in\mathcal U_1(D_i),\,\iota(p_i\eps)=\mathbf 0\rbrace \\
&\cup \lbrace p_i\eps g\mid i\in[1,r],\,\eps\in\mathcal U_1(D_i),\,\iota(p_i\eps)=g\rbrace \\
&\cup \lbrace p_i^2\eps\mid i\in[1,r],\,\eps\in(\mal{\widehat{D_i}})_{\mathbf 0}\setminus(\mathcal U_1(D_i)_{\iota(p_i)})^2\rbrace \\
&\cup \lbrace p_ip_j\eps_i\eps_j\mid i,\,j\in[1,r],\,i\neq j,\,\eps_i\in\mathcal U_1(D_i),\,\eps_j\in\mathcal U_1(D_j),\,\iota(p_i\eps_i)=\iota(p_j\eps_j)=g\rbrace\,.
\end{align*}
\end{lemma}
\begin{proof}
For short, we write $\mathcal B=\mathcal B(G,D_1\times\ldots\times D_r,\iota)$. Since $|G|=2$, we have $\mathsf D(G)=2$, and thus every atom of $\mathcal B$ is a product of at most two atoms of $\mathcal F(G)\times D_1\times\ldots\times D_r$. First, we write down all atoms of $\mathcal F(G)\times D_1\times\ldots\times D_r$, namely,
\[
\mathcal A(\mathcal F(G)\times D_1\times\ldots\times D_r)=\lbrace\mathbf 0,g\rbrace\cup\bigcup_{\mathclap{i\in[1,r]}}\lbrace p_i\eps\mid\eps\in\mathcal U_1(D_i)\rbrace\,,
\]
by \autoref{3.9}.\ref{3.9.1}.
Now, we find
\begin{align*}
\mathcal A(\mathcal F(G)\times D_1\times\ldots\times D_r)\cap\mathcal B
&= \lbrace\mathbf 0\rbrace
\cup \lbrace p_i\eps\mid i\in[1,r],\,\eps\in\mathcal U_1(D_i),\,\iota(p_i\eps)=\mathbf 0\rbrace\,\mbox{, and} \\
\mathcal A(\mathcal F(G)\times D_1\times\ldots\times D_r)\setminus\mathcal B
&= \lbrace\mathbf g\rbrace
\cup \lbrace p_i\eps\mid i\in[1,r],\,\eps\in\mathcal U_1(D_i),\,\iota(p_i\eps)=g\rbrace\,.
\end{align*}
By \autoref{3.13}, $D$ and $H$ are reduced, and therefore $\eps_i\eps_j\notin\mathcal B$ for all $i,\,j\in[1,r]$, $i\neq j$, $\eps_i\in\mathcal U_1(D_i)$, and $\eps_j\in\mathcal U_1(D_i)$. Thus the following products of two atoms of $\mathcal F(G)\times D_1\times\ldots\times D_r$ are atoms of $\mathcal B$:
\begin{align*}
\mathcal A(\mathcal B)\supset & \lbrace g^2\rbrace \\
\cup & \lbrace p_i\eps g\mid i\in[1,r],\,\eps\in\mathcal U_1(D_i),\,\iota(p_i\eps)=g\rbrace \\
\cup & \lbrace p_i^2\eps\mid i\in[1,r],\,\eps\in(\mal{\widehat{D_i}})_{\mathbf 0}\setminus(\mathcal U_1(D_i)_{\iota(p_i)})^2\rbrace \\
\cup & \lbrace p_ip_j\eps_i\eps_j\mid i,\,j\in[1,r],\,i\neq j,\,\eps_i\in\mathcal U_1(D_i),\,\eps_j\in\mathcal U_1(D_j),\,\iota(p_i\eps_i)=\iota(p_j\eps_j)=g\rbrace\,.
\end{align*}
Since we have run through all possible combinations, the assertion follows.
\end{proof}

\begin{lemma}
\label{3.17}
Let $D=\mathcal F(P)\times D_1\times\ldots\times D_r$ be a monoid, where $P\subset D$ is a set of prime elements, $r\in\N$, and, for all $i\in [1,r]$, $D_i\subset [p_i]\times\wmal{D_i}$ is a reduced half-factorial but not factorial monoid of type $(1,1)$. Let $H\subset D$ be a saturated submonoid, $G=\mathsf q(D/H)$ its class group with $|G|=2$, and suppose each class in $G$ contains some $p\in P$. Let $\iota:D_1\times\ldots\times D_r\rightarrow G$ be defined by $\iota(t)=[t]_{D/H}$, denote by $\mathcal B(G,D_1\times\ldots\times D_r,\iota)$ the $(D_1\times\ldots\times D_r)$-block monoid over $G$ defined by $\iota$, and set $|\cdot|_{\mathcal B}=|\cdot|_{\mathcal B(G,D_1\times\ldots\times D_r,\iota)}$.
\begin{enumerate}
\item\label{3.17.1} If $(x,y)\in\sim_{\mathcal B(G,D_1\times\ldots\times D_r,\iota)}$ with $|y|_{\mathcal B}\geq |x|_{\mathcal B}$ and $|y|_{\mathcal B}\geq 5$, then there is a monotone $\mathcal R$-chain concatenating $x$ and $y$; in particular, $x\approx y$, and if $|x|_{\mathcal B}=|y|_{\mathcal B}$, then $x\eq y$.
\item\label{3.17.2} Additionally,
\[
\cmon(\mathcal B(G,D_1\times\ldots\times D_r,\iota)\leq\sup\lbrace |y|_{\mathcal B}\mid(x,y)\in\mathcal A(\sim_{\mathcal B(G,D_1\times\ldots\times D_r,\iota)}),\,|x|_{\mathcal B}\leq |y|_{\mathcal B}\leq 4\rbrace.
\]
\end{enumerate}
\end{lemma}
\begin{proof}
Let $|G|=2$, say $G=\lbrace\mathbf 0,g\rbrace$.
By \autoref{3.5}.\ref{3.5.4}, we set $\mathcal B(G,D_1\times\ldots\times D_r,\iota)=[\mathbf 0]\times\mathcal B$ with $\mathcal B=\lbrace S\in\mathcal B(G,D_1\times\ldots\times D_r,\iota)\mid\mathbf 0\nmid S\rbrace$.
Before we start the actual proof, we establish some machinery to deal with factorizations in $\mathcal B$ and their lengths more systematically.\\
We set $D_0=\mathcal F(\lbrace g\rbrace)$, whence $\mathcal A(D_0)=\lbrace g\rbrace$ and $\mathsf Z(D_0)=D_0$. We define $\iota:D_0\rightarrow G$ by $\iota(g^k)=kg$ for all $k\in\Z$. For $i\in [0,r]$, let $\pi_i:\mathsf Z(D_i)\rightarrow D_i$ be the factorization homomorphism. We set $D'=D_0\times\ldots\times D_r$ and obtain $\mathcal A(D')=\mathcal A(D_0)\cup\ldots\cup\mathcal A(D_r)$. If $a=a_0\mdots a_r\in D'$, where $a_i\in D_i$ for all $i\in [0,r]$, then we set $\iota(a)=\iota(a_0)+\iota(a_1\mdots a_r)=\iota(a_0)+\ldots+\iota(a_r)$. Then $\iota:D'\rightarrow G$ is a homomorphism and $\mathcal B=\iota^{-1}(\mathbf 0)\subset D'$ is a saturated submonoid, whose atoms are given by the following assertion $\mathbf{A1}$.
\begin{itemize}
 \item[$\mathbf{A1}$] An element $x\in D_0\times\mdots\times D_r$ is an atom of $\mathcal B$ if and only if it is of one of the following forms:
 \begin{itemize}
  \item[$\bullet$] $x=a\in\mathcal A(D_i)$ for some $i\in [1,r]$ and $\iota(a)=\mathbf 0$.
  \item[$\bullet$] $x=a_1a_2$, where $a_1\in\mathcal A(D_i)$, $a_2\in\mathcal A(D_j)$, for some $i,\,j\in [0,r]$, $i\neq j$, and $\iota(a_1)=\iota(a_2)=g$.
  \item[$\bullet$] $x=a_1a_2$, where $a_1,\,a_2\in\mathcal A(D_i)$ for some $i\in [0,r]$ such that $\iota(v)=g$ for all $v\in\mathcal A(D_i)$.
 \end{itemize}
 We will call the atoms of the third form \emph{pure in $i$}.
\end{itemize}
\begin{proof}[Proof of $\mathbf{A1}$]
By the listing of all atoms of $\mathcal B(G,D_1\times\ldots\times D_r,\iota)$ in \autoref{3.16} and the fact that $\mathcal A(\mathcal B)=\mathcal A(\mathcal B(G,D_1\times\ldots\times D_r,\iota))\setminus\lbrace\mathbf 0\rbrace$, we must only show the last statement in the case $i\in [1,r]$. Suppose there are $a_1,\,a_2\in\mathcal A(D_i)$ such that $a=a_1a_2\in\mathcal A(\mathcal B)$. Then, obviously, $\iota(a_1)=\iota(a_2)=g$. Now we assume there is some $v\in\mathcal A(D_i)$ with $\iota(v)=\mathbf 0$. By \autoref{3.9}.\ref{3.9.2}, there is $v'\in\mathcal A(D_i)$ such that $a_1a_2=vv'$, and then $\iota(v')=\mathbf 0$, a contradiction.
\end{proof}
Let $F=\mathsf Z(D')=\mathsf Z(D_0)\times\ldots\times\mathsf Z(D_r)=\mathcal F(\mathcal A(D'))$ be the factorization monoid of $D'$. Then $\pi=\pi_0\times\ldots\times\pi_r:F\rightarrow D'$ is the factorization homomorphism of $D'$. We denote by $|\cdot|=|\cdot|_F$ the length function in the free monoid $F$, and for $x,\,y\in F$, we write $x\mid y$ instead of $x\mid_F y$.
For $a\in\mathcal A(\mathcal B)$, let $\theta_0(a)\in\pi^{-1}(a)\subset\mathsf Z(D')$ be a factorization of $a$ in $D'$. If $a\in\mathcal A(D')$, then $\theta_0(a)=a$; otherwise $\theta_0(a)=a_1a_2\in F$ for some $a_1,\,a_2\in\mathcal A(D')$ such that $a=a_1a_2$ in $D'$. By $\mathbf{A1}$, $\#\pi^{-1}(a)=1$ unless $a$ is pure in $i$ for some $i\in [1,r]$. Let $\theta:\mathsf Z(\mathcal B)\rightarrow F$ be the unique monoid homomorphism satisfying $\theta|\mathcal A(\mathcal B)=\theta_0$. Then $\theta$ induces the following commutative diagram
\[
\xymatrix{
\mathsf Z(\mathcal B)\ar[d]^{\pi_{\mathcal B}}\ar[r]^{\theta\quad\;} & F=\mathsf Z(D')\ar[d]^{\pi_{D'}} \\
\mathcal B\ar@{^{(}->}[r] & D',
}
\]
where $\pi_{\mathcal B}$ denotes the factorization homomorphism of $\mathcal B$ and the bottom arrow denotes the inclusion. For $x\in\mathsf Z(\mathcal B)$, we set $|x|=|\theta(x)|$.
For $x\in\mathsf Z(\mathcal B)$, we define its components $x_i\in\mathsf Z(D_i)$ for $i\in [0,r]$ by $\theta(x)=x_0\mdots x_r$. Then $\pi\circ\theta(x)\in\mathcal B$ implies $\iota\circ\pi_0(x_0)+\ldots+\iota\circ\pi_r(x_r)=\mathbf 0$. For $i\in [0,r]$, we set $x_i=u_{i,1}\mdots u_{i,k_i}v_{i,1}\mdots v_{i,l_i}$, where $u_{i,j},\,v_{i,j}\in\mathcal A(D_i)$, $\iota(u_{i,j})=\mathbf 0$, and $\iota(v_{i,j})=g$. We define $x_i',\,x_i''\in\mathsf Z(D_i)$ by $x_i'=u_{i,1}\mdots u_{i,k_i}$ and $x_i''=v_{i,1}\mdots v_{i,l_i}$, whence $x_i=x_i'x_i''$. In particular, $|x_0'|=0$, $x_0=x_0''$, and $\iota\circ\pi_i(x_i)=l_ig=|x_i''|g$. Therefore we obtain $|x_0''|+\ldots+|x_r''|\equiv 0\mod 2$. If $i\in [1,r]$ and $a\in D_i$ is such that $a\mid x_i'$, then $a\mid_{\mathcal B}x$. In $\mathsf Z(\mathcal B)$, there is a factorization $x=u_1\mdots u_mv_1\mdots v_n$, where $u_j,\,v_j\in\mathcal A(\mathcal B)$, $|u_j|=1$ for all $j\in [1,m]$, $|v_j|=2$ for all $j\in [1,n]$, and we obtain
\[
 m=\sum_{i=1}^r|x_i'|,\quad
 n=\frac{1}{2}\sum_{i=0}^r|x_i''|,
 \quad\mbox{and}\quad
 |x|_{\mathcal B}=m+n=\frac{1}{2}\sum_{i=0}^r(|x_i|+|x_i'|)\leq\sum_{i=0}^r|x_i|.
\]
Assume now that $x=x_0\mdots x_r$, $y=y_0\mdots y_r\in\mathsf Z(\mathcal B)$ are as above, and suppose that $(x,y)\in\sim_{\mathcal B}$. Then $x_0=y_0$, $|x_i|=|y_i|$ (since each $D_i$ is half-factorial), $\pi_i(x_i)=\pi_i(y_i)\in D_i$, and thus $\iota\circ\pi_i(x_i)=\iota\circ\pi_i(y_i)\in G$, and therefore $|x_i''|\equiv |y_i''|\mod 2$ and $|x_i'|\equiv |y_i'|\mod 2$ for all $i\in [1,r]$. Consequently, it follows that the following are all equivalent:
\begin{itemize}
 \item $|x|_{\mathcal B}\leq |y|_{\mathcal B}$
 \item $\sum_{i=1}^r|x_i'|\leq\sum_{i=1}^r|y_i'|$
 \item $\sum_{i=1}^r|x_i''|\geq\sum_{i=1}^r|y_i''|$
\end{itemize}
Additionally, we find
\[
2|x|_{\mathcal B}=\sum_{i=0}^r(|x_i|+|x_i'|)\geq\sum_{i=0}^r(|y_i|)\geq |y|_{\mathcal B},
\]
and thus $|y|_{\mathcal B}\geq 5$ implies $|x|_{\mathcal B}\geq 3$.\\
Before we start with the actual proof of part 1 of \autoref{3.17}, we prove the following reduction step.
\begin{itemize}
 \item[$\mathbf{A2}$] In the proof of part 1 of \autoref{3.17}, we may assume that $|x_i|=|y_i|\geq 2$ for all $i\in [1,r]$.
\end{itemize}
\begin{proof}[Proof of $\mathbf{A2}$]
If $i\in[1,r]$, then $|x_i|=0$ if and only if $|y_i|=0$, and in this case we may neglect this component. If $|x_i|=0$ for all $i\in[1,r]$, then there is nothing to do. If $i\in[1,r]$, then $|x_i|=1$ if and only if $|y_i|=1$, and then $x_i=y_i\in\mathcal A(D_i)$. Suppose that $i\in[1,r]$ and $|x_i|=1$. If $\iota(x_i)=\mathbf 0$, then $x_i\in\mathcal A(\mathcal B)$ and $x_i$ is a greatest common divisor of $x$ and $y$ in $\mathsf Z(\mathcal B)$; hence $(x,\,y)$ is a monotone $\mathcal R$-chain concatenating $x$ and $y$. If $\iota(x_i)=g$, then we set $\tilde x=gx_i^{-1}x$, $\tilde y=gy_i^{-1}y$, and then $(\tilde x,\tilde y)\in\,\sim_{\mathcal B}$, $|\tilde x_i|=|\tilde y_i|=0$, and whenever there is a monotone $\mathcal R$-chain concatenating $\tilde x$ and $\tilde y$, then there is a monotone $\mathcal R$-chain concatenating $x$ and $y$.
\end{proof}
Now we are ready to do the actual proof of the lemma. Suppose that $(x,y)\in\sim_{\mathcal B}$ with $|y|_{\mathcal B}\geq 5$, $|y|_{\mathcal B}\geq |x|_{\mathcal B}$, $x=x_0\mdots x_r$, $y=y_0\mdots y_r$, $x_i=x_i'x_i''$, and $y_i=y_i'y_i''$ as above, and $|x_i|=|y_i|\geq 2$ for all $i\in [0,r]$. We shall use $\mathbf A1$ and \autoref{3.9}.\ref{3.9.2} again and again without mentioning this explicitly. Of course, we may assume that there is no $a\in\mathcal A(\mathcal B)$ such that $a\mid_{\mathcal B}x$ and $a\mid_{\mathcal B}y$, since then  there is, trivially, a monotone $\mathcal R$-chain concatenating $x$ and $y$. For now, assume $|x|_{\mathcal B}\geq 4$; the remaining case, where $|x|_{\mathcal B}=3$, will be studied at the end of the proof after Case 3.
\smallskip

\noindent\textbf{Case 1.} There is some $i\in [1,r]$ such that $|x_i'|\geq 1$ and $|y_i'|\geq 1$.\\
\indent\textbf{Case 1.1.} There is some $i\in [1,r]$ such that $|x_i|'\geq 2$ and $|y_i'|\geq 1$.\\
Let $a_1,\,a_2,\,b\in\mathcal A(D_i)$ be such that $a_1a_2\mid x_i'$ and $b\mid y_i'$. Then there is some $b'\in\mathcal A(D_i)$ such that $a_1a_2=bb'$. Thus $\iota(b')=\mathbf 0$, and if $x^*\in\mathsf Z(\mathcal B)$ is such that $x=a_1a_2x^*$, then $x,\,bb'x^*,\,y$ is a monotone $\mathcal R$-chain concatenating $x$ and $y$.\\
\indent\textbf{Case 1.2.} There is some $i\in [1,r]$ such that $|x_i'|=1$ and $|y_i'|\geq 1$.\\
Then $x_i'\in\mathcal A(\mathcal B)$. Let $a,\,b\in\mathcal A(D_i)$ be such that $a\mid x_i''$ and $b\mid y_i'$. Let $u\in\mathcal A(F)$ be such that $au\in\mathcal A(\mathcal B)$ and $au\mid_{\mathcal B} x$. Since $x_i'\in\mathcal A(D_i)$, we obtain $u\notin\mathcal A(D_i)$. Let $b'\in\mathcal A(D_i)$ be such that $x_i'a=bb'$, whence $\iota(b')=g$ and $b'u\in\mathcal A(\mathcal B)$. If $x=x_i'(au)x^*$, where $x^*\in\mathsf Z(\mathcal B)$, then $|x^*|_{\mathcal B}\geq 1$, and $x,\,b(b'u)x^*,\,y$ is a monotone $\mathcal R$-chain concatenating $x$ and $y$.
\smallskip

\noindent\textbf{Reduction 1.} By Case 1, we may now assume that, for all $i\in [1,r]$, either $|x_i'|=0$ or $|y_i'|=0$. In particular, if $|x_i'|\geq 1$, then $|y_i'|=0$, and therefore $|x_i'|\geq 2$, since $|x_i'|\equiv |y_i'|\mod 2$. Similarly, if $|y_i'|\geq 1$, then $|y_i'|\geq 2$.
\smallskip

\noindent\textbf{Case 2.} There is some $i\in [1,r]$ such that $|y_i'|\geq 1$.\\
In this case, $|x_i'|=0$ by Reduction 1, hence $|y_i'|\geq 2$ and $|x_i''|\geq 2$. Let $b\in\mathcal A(D_i)$ be such that $b\mid y_i'$. Now we must distinguish a few more cases based on $|x|_{\mathcal B}$ and $|y|_{\mathcal B}$.\\
\indent\textbf{Case 2.1.} $|x|_{\mathcal B}=|y|_{\mathcal B}$.\\
Note that in this case $|x|_{\mathcal B}=|y|_{\mathcal B}\geq 5$.
We assert that there is some $j\in [1,r]\setminus\lbrace i\rbrace$ such that $|y_j'|<|x_j'|$. Indeed, if $|y_j'|\geq |x_j'|$ for all $j\in [1,r]\setminus\lbrace i\rbrace$, then
\[
\sum_{\nu=1}^r|x_\nu'|\leq\sum_{\substack{\nu=1\\ \nu\neq i}}^r|y_\nu'|<\sum_{\nu=1}^r|y_\nu'|,
\]
and therefore $|x|_{\mathcal B}<|y|_{\mathcal B}$, a contradiction. By Reduction 1, we obtain $|y_j'|=0$. Hence $|y_j''|\geq 2$, and $|x_j'|\geq 2$. We write $x$ in the form
\[
 x=(a_1u_1)\mdots (a_ku_k)(a_{k+1}u_1^*)\mdots (a_{k+t}u_t^*)(e_1u_{k+1})\mdots (e_su_{k+s})\tilde x,
\]
where $k,\,s,\,t\in\N_0$, $x_i''=a_1\mdots a_{k+t}$, $x_j''=u_1\mdots u_{k+s}$, $u_1^*,\ldots,u_t^*,\,e_1,\ldots,e_s\in\mathcal A(F)\setminus(\mathcal A(D_i)\cup\mathcal A(D_j))$, $k+t\geq 2$, and
\[
\tilde x=(e_1\mdots e_s)^{-1}\prod_{\substack{\nu=1\\ \nu\neq i}}^r x_\nu'\prod_{\substack{\nu=1\\ \nu\neq i,j}}^r x_\nu''\in\mathsf Z(\mathcal B).
\]
Let $c_1,\,c_2,\,d_1\in\mathcal A(D_j)$ be such that $c_1c_2\mid x_j'$, $d_1\mid y_j''$, and choose $d_2\in\mathcal A(D_j)$ such that $c_1c_2=d_1d_2$, whence $\iota(d_2)=g$.\\
\indent\textbf{Case 2.1a.} $t\geq 2$.\\
Choose some $b'\in\mathcal A(D_i)$ such that $a_{k+1}a_{k+2}=bb'$. Then $\iota(b')=\mathbf 0$, $d_1u_1^*,\,d_2u_2^*\in\mathcal A(\mathcal B)$, and we set $x=(a_{k+1}u_1^*)(a_{k+2}u_2^*)c_1c_2x^*$, where $x^*\in\mathsf Z(\mathcal B)$ and $|x^*|_{\mathcal B}\geq 1$. Now $x,\,bb'(d_1u_1^*)(d_2u_2^*)x^*,\,y$ is a monotone $\mathcal R$-chain concatenating $x$ and $y$.\\
\indent\textbf{Case 2.1b.} $t=1$.\\
Note that $|x_i''|=k+1\geq 2$ implies $k\geq 1$. Assume first that there is some $v\in\mathcal A(\mathcal B)$ such that $|v|=2$ and $v\mid\tilde x$, say $v=v'v''$, where $v',\,v''\in\mathcal A(F)\setminus(\mathcal A(D_i)\cup\mathcal A(D_j))$ and $\iota(v')=\iota(v'')=g$. Then it follows that $a_1v',\,u_1v''\in\mathcal A(\mathcal B)$, and we set $x=(a_1u_1)(a_{k+1}v_1)(v'v'')x^*$, where $x^*\in\mathsf Z(\mathcal B)$ and $|x^*|_{\mathcal B}\geq 1$. We set $x'=(a_1v')(a_{k+1}v_1)(u_1v'')x^*$. Then we find $x'\in\mathsf Z(\mathcal B)$, $(x,x')\in\sim_{\mathcal B}$, and $x\eq x'$. Hence, $(x',y)\in\sim_{\mathcal B}$, $x'\eq y$ by Case 2.1a, and therefore $x\eq y$.\\
Now we set $u_1^*=u$, and we let $m\in [0,r]\setminus\lbrace i,j\rbrace$ be such that $u\in\mathcal A(D_m)$. We write $x$ in the form
\[
 x=(a_1u_1)\mdots(a_ku_k)(a_{k+1}u)(e_1u_{k+1})\mdots (e_su_{k+s})\prod_{\substack{\nu=1\\ \nu\neq i}}^r x_\nu',
\quad\mbox{where }
s+1=\sum_{\substack{\nu=0\\ \nu\neq i,j}}^r|x_\nu''|.
\]
We may assume that $|x_n' | =0$ for all $n\in [1,r]\setminus\lbrace m,j\rbrace$. Indeed, let $n\in [1,r]\setminus\lbrace m,j\rbrace$ be such that $|x_n'|\geq 1$. Then $|x_n'|\geq 2$, $|y_n'| =0$, and $|y_n'' |\geq 2$. Let $v_1,\,v_2,\,w_1\in \mathcal A(D_n)$ be such that $v_1v_2\mid x_n'$ and $w_1\mid y_n''$, and choose $b_1\in\mathcal A(D_i)$ and $w_2\in\mathcal A(D_n)$ such that $a_1a_{k+1} = bb_1$ and $v_1v_2 = w_1w_2$. Then $\iota(b_1)=\mathbf 0$, $\iota(w_2) = g$, $u_1w_1,\, uw_2\in\mathcal A(\mathcal B)$, and if $x =(a_1u_1)(a_{k+1}u)v_1v_2x^*$, where $x^* \in \mathsf Z(\mathcal B)$, then $|x^*|_{\mathcal B}\geq 1$, and $x,\,bb_1(u_1w_1)(uw_2)x^*,\,y$ is a monotone $\mathcal R$-chain concatenating $x$ and $y$.\\
Thus suppose that $|x_n' | =0$ for all $n \in [1,r] \setminus\lbrace m,j\rbrace$, and consequently
\[
x = (a_1u_1)\mdots(a_ku_k)(a_{k+1}u)(e_1u_{k+1})\mdots(e_su_{k+s})x_j'x_m'.
\]
We assert that there exist $v_1,\,v_2,\,v_3\in\mathcal A(D_m)$ and $w_1,\,w_2,\,w_3 \in\mathcal A(D_j)$ such that $v_1v_2v_3\mid y_m$, $w_1w_2w_3\mid y_j$, $\iota(v_\nu)=\iota(w_\nu)=g$, $v_\nu w_\nu\in\mathcal A(\mathcal B)$ and $v_\nu w_\nu\mid_{\mathcal B}y$ for all $\nu\in [1,3]$. Indeed, observe that
\[
|y_i''| = |y_i| -|y_i'| \le |y_i| -2 = |x_i| -2 = |x_i''|-2 = k-1,
\]
\[
|y_j''| = |y_j| = |x_j'|+|x_j''|\geq 2 + |x_j''| = k+s+2,
\]
and set $y_j'' = y_{j,1}\mdots y_{j,\mu}$, where $\mu=|y_j''|$, and, for all $\alpha\in [1,\mu]$, $y_{j,\alpha}\in\mathcal A(D_j)$ and $\iota(y_{j,\alpha}) = g$. For $\alpha \in [1,\mu]$, let $u_{j,\alpha}\in\mathcal A(F)$ be such that $y_{j,\alpha}u_{j,\alpha}\in\mathcal A(B)$ and $y_{j,\alpha}u_{j,\alpha}\mid_{\mathcal B} y$. Since $|x_j'|\geq 1$, it follows that $u_{j,\alpha}\notin\mathcal A(D_j)$ for all $\alpha\in [1,\mu]$. For $\nu\in [0,r]\setminus\lbrace j\rbrace$, we set $N_\nu=|\lbrace\alpha\in [1,\mu]\mid y_{\nu,\alpha}\in\mathcal A(D_\nu\rbrace|$, and we obtain
\[
\mu=\sum_{\substack{\nu =0\\ \nu\neq j}}^r N_\nu = N_m+N_i+ \sum_{\substack{\nu =0\\ \nu\neq i,j,m}}^r N_\nu\leq N_m + |y_i''|+ \sum_{\substack{\nu =0\\ \nu\neq i,j,m}}^r |y_\nu|.
\]
Since $|y_\nu| = |x_\nu| = |x_\nu''|$ for all $\nu\in [0,r]\setminus\lbrace i,j,m\rbrace$ and $|x_m''|\geq 1$, it follows that
\[
k+s+2\leq\mu\leq N_m + k - 1 + \sum_{\substack{\nu =0\\ \nu\neq i,j,m}}^r |x_\nu''|\leq N_m + k - 1 +  \sum_{\substack{\nu =0\\ \nu\neq i,j}}^r |x_\nu''| - |x_m''| = N_m + k +s -1,
\]
and therefore $N_m\geq 3$, which implies the existence of $v_1,\,v_2,\,v_3$ and $w_1,\,w_2,\,w_3$ as asserted. In particular, it follows that $|x_m| = |y_m|\geq |y_m''|\geq 3$ and $|x_j| = |y_j|\geq |y_j''|\geq 3$. Let $u_1'\in\mathcal A(D_j)$ be such that $u_1u_{k+1} = u_1'w_1$. Then  $\iota(u_1')=g$ and $a_1u_1'\in\mathcal A(\mathcal B)$.\\
\indent\textbf{Case 2.1b$\boldsymbol{'}$.} $s\geq 1$.\\
We assume first that $|x_m'|\geq 1$. Let $u'\in\mathcal A(D_m)$ be such that $u'\mid x_m'$. Then there exists some $v\in\mathcal A(D_m)$ such that $uu' = v_1v$. Hence $\iota(v)=\mathbf 0$, and $x = (a_1u_1)(a_{k+1}u)(e_1u_{k+1}) u'x^*$, where $x^* \in\mathsf Z(\mathcal B)$ and $|x^*|_{\mathcal B}\geq 1$. Since $a_1u_1'\in\mathcal A(\mathcal B)$ and $a_{k+1}e_1\in\mathcal A(\mathcal B)$, we conclude that $x,\,(a_1u_1')(a_{k+1}e_1)(v_1w_1)vx^*,\,y$ is a monotone $\mathcal R$-chain concatenating $x$ and $y$.\\
Assume now that $|x_m'| =0$. Then $|x_m''| = |x_m|\geq 3$, and (after renumbering if necessary) we may assume that $e_1\in\mathcal A(D_m)$. Let $v\in\mathcal A(D_m)$ be such that $ue_1=v_1v$. Then $\iota(v) = g$, $a_{k+1}v\in\mathcal A(\mathcal B)$ and $x = (a_1u_1)(a_{k+1}u)(e_1u_{k+1})x^*$, where $x^*\in\mathsf Z(\mathcal B)$ and $|x^*|_{\mathcal B}\geq 1$. Hence $x,\,(a_1u_1')(a_{k+1}v)(v_1w_1)x^*,\,y$ is a monotone $\mathcal R$-chain concatenating $x$ and $y$.\\
\indent\textbf{Case 2.1b$\boldsymbol{''}$.} $s=0$.\\
We assert that $k\geq 2$. Indeed, assuming to the contrary that $k = 1$, then $x_i = x_i'' = a_1a_2$, $x_j'' = u_1$, $u_m'' = u$, $3\leq |x_j| = 1 + |x_j'|$, $3\leq |x_m| = 1 + |x_m'|$, hence $|x_j'|\geq 2$, $|x_m'|\geq 2$, and therefore $|y_j'|=|y_m'| =0$. Hence
\[
\sum_{\nu =1}^r |y_\nu'| = |y_i'|\leq |y_i| = 2\quad\mbox{and}\quad\sum_{\nu =1}^r |x_\nu'| = |x_j'| + |x_m'|\geq 4,
\]
a contradiction to $|x|_{\mathcal B}=|y|_{\mathcal B}$.\\
As $k\geq 2$, it follows that $u_2\in\mathcal A(D_j)$, hence $u_2v_1\in\mathcal A(\mathcal B)$, and we choose $b_2\in\mathcal A(D_i)$ such that $a_1a_2 = bb_2$, whence $\iota(b_2)=\mathbf 0$. Since $3\leq |x_m| = 1 + |x_m'|$, we get $|x_m'|\geq 2$, and there exist $v_1',\,v_2'\in\mathcal A(D_m)$ such that $v_1'v_2'\mid x_m'$. Let $v\in\mathcal A(D_m)$ be such that $v_1'v_2' = v_1v$. Then $\iota(v) = g$ and $u_1v\in\mathcal A(\mathcal B)$.\\
Assume first that $k\geq 2$, and set $x =(a_1u_1)(a_2u_2)v_1'v_2'x^*$, where $x^*\in\mathcal A(\mathcal B)$ and $|x^*|_{\mathcal B}\geq 1$. Then $u_2v_1\in\mathcal A(\mathcal B)$, and therefore $x,\,bb_2(u_1v)(u_2v_1)x^*,\,y$ is a monotone $\mathcal R$-chain concatenating $x$ and $y$.\\
\indent\textbf{Case 2.1c.} $t=0$.\\
Observe that $|x_i''| = k\geq2$ and \ $x= (a_1u_1)\mdots (a_ku_k)(e_1u_{k+1})\mdots (e_su_{k+s})\,\tilde x$. We may assume that there is no $v \in\mathcal A(\mathcal B)$ such that $|v| =2$ and $v \mid \tilde x$. Indeed, if $v \in\mathcal A(B)$ is such that $|v| =2$ and $v \mid \tilde x$. Then $v = v'v''$, where $v',\,v'' \in\mathcal A(F) \setminus (\mathcal A(D_i) \cup\mathcal A(D_j))$, \ $\iota (v') = \iota(v'')=g$, and $a_2v',\, u_2v'' \in\mathcal A(\mathcal B)$. We set $x = (a_1u_1)(a_2u_2)(v'v'')x^*$, where $x^* \in \mathsf Z(\mathcal B)$, and $x' = (a_1u_1)(a_2v')(u_2v'')x^*$. Then it follows that $x' \in \mathsf Z(\mathcal B)$, \ $(x,x') \in \,\sim_{\mathcal B}$ and $x \eq x'$. Hence $(x',y) \in \,\sim_\mathcal B$, \ $x' \eq y$ by \,Case 2.1b, and therefore $x \eq y$.\\
Next we prove that there is some $n \in [1,r] \setminus \{j\}$ such that $|x_n'|\geq1$. Assume the contrary. Then $x = (a_1u_1)\mdots (a_ku_k)(e_1u_{k+1 })\mdots (e_su_{k+s})x_j'$, \ $x_i = x_i'' = a_1\mdots a_k$, \ $x_j'' = u_1\mdots u_{k+s}$, and
\[
e_1\mdots e_s = \prod_{\substack{\nu =0\\ \nu \neq i,j}}^r x_\nu\,.
\]
Moreover, we obtain \,$|y_i''| = |y_i| - |y_i'| \le |x_i| -2 = |x_i''| - 2 = k-2$\, and \,$|y_j''| = |y_j| = |x_j'|+|x_j''|\geq2+k+s$. We set $y_j'' = y_{j,1}\mdots y_{j,\mu}$, where $\mu = |y_j''|$, and, for all $\alpha \in [1,\mu]$, \ $y_{j,\alpha} \in\mathcal A(D_j)$ and $\iota(y_{j,\alpha}) = g$. For $\alpha \in [1,\mu]$, let \ $u_{j,\alpha} \in\mathcal A(F)$ be such that $y_{j,\alpha}u_{j,\alpha} \in\mathcal A(B)$ and $y_{j,\alpha}u_{j,\alpha} \mid_{\mathcal B} y$. Since $|x_j'|\geq1$, it follows that $u_{j,\alpha} \notin\mathcal A(D_j)$ for all $\alpha \in [1,\mu]$. For $\nu \in [0,r] \setminus \{j\}$, we set $N_\nu = |\{ \alpha \in [1,\mu] \mid y_{\nu,\alpha} \in\mathcal A(D_\nu \}|$, and we obtain
\[
2+k+s \ \le \ \mu \ = \ \sum_{\substack{\nu =0\\ \nu \neq j}}^r N_\nu \ \le \ \sum_{\substack{\nu =0\\ \nu \neq j}}^r |y_\nu''| \ \le \ |y_i''| + \sum_{\substack{\nu =0\\ \nu \neq i,j}}^r |y_\nu| \ = \ |y_i''| + \sum_{\substack{\nu =0\\ \nu \neq i,j}}^r |x_\nu| \ \le \ k-2+s\,,
\]
a contradiction.\\
Thus now let $n \in [1,r] \setminus \{j\}$ be such that $|x_n'|\geq1$. Then $|x_n'|\geq2$, \ $|y_n'|=0$ and $|y_n''|\geq2$. Let $v_1,\,v_2,\,w_1 \in\mathcal A(D_n)$ be such that $v_1v_2 \mid x_n'$, \ $w_1 \mid y_n''$, and choose some $x_2 \in\mathcal A(D_n)$ such that $v_1v_2 = w_1w_2$. Then $x = (a_1u_1)(a_2u_2)v_1v_2 x^*$, where $x^* \in\mathcal A(\mathcal B)$ and $|x^*|\geq1$. Let $b_2 \in\mathcal A(D_i)$ be such that $a_1a_2 = bb_2$, whence $\iota (b_2)=\boldsymbol 0$. Then $x,\, bb_2(u_1w_1)(u_2w_2)x^*,\,y$ is a monotone $\mathcal R$-chain concatenating $x$ and $y$.\\
\indent\textbf{Case 2.2.} $|y|_{\mathcal B}\geq|x|_{\mathcal B}+1$, and we are in the following special situation.
\begin{itemize}
 \item[$\mathbf{S1}$] There exist $a_1,\,a_2\in\mathcal A(D_i)$ and $u_1,\,u_2\in\mathcal A(F)\setminus\mathcal A(D_i)$ such that $a_1u_1,\,a_2u_2,\,u_1u_2\in\mathcal A(\mathcal B)$ and $(a_1u_1)(a_2u_2)\mid_{\mathcal B} x$.
\end{itemize}
We set $x = (a_1u_1)(a_2u_2)x^*$, where $x^*\in\mathcal A(\mathcal B)$ and $|x|_{\mathcal B}\geq 1$, and we let $b'\in\mathcal A(D_i)$ be such that $a_1a_2 = bb'$, whence $\iota(b')=0$. Then $x,\,bb'(u_1u_2)x^*,\,y$ is a monotone $\mathcal R$-chain concatenating $x$ and $y$.\\
\indent\textbf{Case 2.3.} $|y|_{\mathcal B}=|x|_{\mathcal B}+1$, and we are not in the special situation $\mathbf{S1}$.\\
We set $x_i'' = a_1\mdots a_k$, where $a_1,\ldots,a_k\in\mathcal A(D_i)$ and $k \ge 2$. For $\nu \in [1,k]$, let $u_\nu\in\mathcal A(F)$ be such that $a_\nu u_\nu\in\mathcal A(\mathcal B)$ and $(a_1u_1)\mdots (a_ku_k)\mid_{\mathcal B} x$. Since $|y_i'| \ge 1$ and we are not in the special situation $\mathbf{S1}$, there exists some $j \in [1,r] \setminus \{i\}$ such that $u_1\mdots u_k\mid x_j''$. Suppose that $x_j'' = u_1\mdots u_{k+s}$, where $s \in \N_0$, and let $c_1,\ldots, c_s\in\mathcal A(F) \setminus (\mathcal A(D_i)\cup\mathcal A(D_j))$ be such that \ $x = (a_1u_1)\mdots (a_ku_k)(c_1u_{k+1})\mdots (c_su_{k+s})\tilde x$ for some $\tilde x \in \mathsf Z(\mathcal B)$.\\
We may assume that there is no $v\in\mathcal A(\mathcal B)$ such that $|v| =2$ and $v\mid\tilde x$. Indeed, suppose that $v\in\mathcal A(B)$ is such that $|v| =2$ and $v\mid\tilde x$. Then $v = v'v''$, where $v',\,v''\in\mathcal A(F)\setminus (\mathcal A(D_i)\cup\mathcal A(D_j))$, $\iota (v') = \iota(v'')=g$, and $a_2v',\,u_2v''\in\mathcal A(\mathcal B)$. We set $x = (a_1u_1)(a_2u_2)(v'v'')x^*$, where $x^*\in\mathsf Z(\mathcal B)$, and $x' = (a_1u_1)(a_2v')(u_2v'')x^*$. Then it follows that $x'\in\mathsf Z(\mathcal B)$, $(x,x')\in\sim_{\mathcal B}$ and $x \eq x'$. Hence $(x',y) \in\sim_\mathcal B$, and, by Case 2.2, there is a monotone $\mathcal R$-chain concatenating $x'$ and $y$, and therefore there is a monotone $\mathcal R$-chain concatenating $x$ and $y$.\\
Hence $x$ is of the form
\[
x = (a_1u_1)\mdots (a_ku_k)(c_1u_{k+1})\mdots (c_su_{k+s})x_1'\mdots x_r',
\]
and we assert that there exists some $m\in [1,r]\setminus\lbrace j\rbrace$ such that $|x_m'|\geq 2$. Indeed, if we assume to the contrary that $|x_m'|=0$ for all $m\in [1,r]\setminus\lbrace j\rbrace$, then we obtain $|x|=2(k+s)+|x_j'|$, and since $|y|_{\mathcal B}=|x|_{\mathcal B}+1$, it follows that
\[
 \sum_{\nu=0}^r|y_\nu'|=\sum_{\nu=0}^r|x_\nu'|+2=|x_j'|+2.
\]
If $|y_j'|\geq 1$, then we find $|x_j'|=0$ and $|y_j'|\geq 2$, and therefore
\[
 4\leq |y_j'|+|y_i'|\leq\sum_{\nu=0}^r|y_\nu'|=2,
\]
a contradiction. Hence it follows that $|y_j'|=0$, and then $|y_j''|=|y_j|=|x_j|=k+s+|x_j'|$. Now we find
\[
 \sum_{\substack{\nu=0\\ \nu\neq j}}^r|y_\nu''|\leq |y|-|y_j''|-|y_i'|\leq |x|-(k+s+|x_j'|)-2=k+s-2\leq |y_j''|-2.
\]
We set $y_j''=y_{j_1}\mdots y_{j,\mu}$, where $\mu=|y_j''|$ and, for all $\alpha\in [1,\mu]$, $y_{j,\alpha}\in\mathcal A(D_j)$ and $\iota(y_{j,\alpha})=g$. For $\alpha\in [1,\mu]$, let $u_{j,\alpha}\in\mathcal A(F)$ be such that $y_{j,\alpha}u_{j,\alpha}\in\mathcal A(\mathcal B)$ and $y_{j,\alpha}u_{j,\alpha}\mid_{\mathcal B}y$. For $\nu\in [0,r]$, we set $N_\nu=\#\lbrace\alpha\in [1,\mu]\mid y_{\nu,\alpha}\in\mathcal A(D_\nu)\rbrace$, and we obtain
\[
 0\leq\sum_{\substack{\nu=0\\ \nu\neq j}}^r|y_\nu''|\leq |y_j''|-2=\sum_{\nu=0}^r N_\nu-2,
\]
and therefore $N_j\geq 2$. Hence, there exist $w_1,\,w_2\in\mathcal A(D_j)$ such that $\iota(w_1)=\iota(w_2)=g$ and $w_1w_2\in\mathcal A(\mathcal B)$. On the other hand, $u_1u_2\notin\mathcal A(\mathcal B)$, since we are not in the special situation $\mathbf{S1}$, and therefore $u_1u_2=u_1'u_2'$, where $u_1',\,u_2'\in\mathcal A(D_j)$ and $\iota(u_1')=\iota(u_2')=\mathbf 0$. Hence the existence of $w_1w_2\in\mathcal A(\mathcal B)$ contradicts $\mathbf{A1}$.\\
Let now $m\in [1,r]\setminus\lbrace j\rbrace$ be such that $|x_m'|\geq 2$ and let $b'\in\mathcal A(D_i)$ be such that $a_1a_2=bb'$. By Reduction 1, we obtain $|y_m'|=0$, hence $|y_m''|\geq 2$, there exist $v',\,v''\in\mathcal A(D_m)$ such that $v'v''\mid x_m'$, and there exists some $u'\in\mathcal A(D_m)$ such that $u'\mid y_m''$. Let $u''\in\mathcal A(D_m)$ be such that $v'v''=u'u''$, whence $\iota(u'')=g$, and set $x=(a_1u_1)(a_2u_2)v'v''x^*$, where $x^*\in\mathsf Z(\mathcal B)$. If $|x|_{\mathcal B}=4$, then $x=(a_1u_1)(a_2u_2)v'v''$ and thus $y=y_i'y_j'y_m''$, where $|y_i'|=|y_j'|=|y_m''|=2$, and thus there is a pure atom in $m$ dividing $y$. Since $v',\,v''\in\mathcal A(D_m)$ and $\iota(v')=\iota(v'')=\mathbf 0$, this contradicts $\mathbf{A1}$. Now we may assume $|x|_{\mathcal B}\geq 5$. Then we have $|x^*|_{\mathcal B}\geq 1$ and it follows that $u_1u', u_2u''\in\mathcal A(\mathcal B)$, and $x,\,bb'(u_1u')(u_2u'')x^*,\,y$ is a monotone $\mathcal R$-chain concatenating $x$ and $y$.\\
\indent\textbf{Case 2.4.} $|y|_{\mathcal B}\geq|x|_{\mathcal B}+2$, and we are not in the special situation $\mathbf{S1}$.\\
Let $a_1,\,a_2\in\mathcal A(D_i)$ be such that $a_1a_2\mid x_i''$. Since $|y_i'|>0$, there are $u_1,\,u_2\in\mathcal A(F)\setminus\mathcal A(D_i)$ such that $a_1u_1,\,a_2u_2\in\mathcal A(\mathcal B)$ and $(a_1u_1)(a_2u_2)\mid_{\mathcal B} x$.
We set $x = (a_1u_1)(a_2u_2)x^*$, where $x^*\in\mathcal A(\mathcal B)$ and $|x^*|\geq 1$, and $u_1 u_2 =v_1 v_2$ for some $v_1,\,v_2\in\mathcal A(D_i)$ such that $\iota(v_1)=\iota(v_2)=\mathbf 0$. Again we set $a_1a_2 = bb'$, where $b'\in\mathcal A(D_i)$ and $\iota(b')=\mathbf 0$, and then $x,\, bb'v_1v_2x^*,\,y$ is a monotone $\mathcal R$-chain concatenating $x$ and $y$, since $|y|_{\mathcal B}\geq |x|_{\mathcal B}+2=|x^*|+4$.
\smallskip

\noindent\textbf{Reduction 2.} By Case 2, we may now assume that $|y_i'|=0$ for all $i\in [1,r]$, and since $|x|_{\mathcal B}\leq |y|_{\mathcal B}$, this implies that $|x_i'|=0$. Therefore $|x_i''|\geq 2$ and $|y_i''|\geq 2$ for all $i\in [1,r]$. Since $x_0''=x_0=y_0=y_0''$, we have $x_i=x_i''$ for all $i\in [0,r]$.
\smallskip

\noindent\textbf{Case 3.} $x_i=x_i''$, $y_i=y_i''$, and $|x_i|=|y_i|\geq 2$ for all $i\in [0,r]$.\\
\indent{\textbf{Case 3a.}} There is some $i\in [0,r]$ such that
\[
 \sum_{\substack{\nu=0\\ \nu\neq i}}^r|x_\nu|<|x_i|
\quad\Bigg[
\mbox{and thus also }
 \sum_{\substack{\nu=0\\ \nu\neq i}}^r|y_\nu|<|y_i|
\Bigg].
\]
There exist $a_1,\,a_2,\,b_1,\,b_2\in\mathcal A(D_i)$ such that $a_1a_2\in\mathcal A(\mathcal B)$, $b_1b_2\in\mathcal A(\mathcal B)$, $a_1a_2\mid_{\mathcal B}x$, and $b_1b_2\mid_{\mathcal B}y$. Let $b\in\mathcal A(D_i)$ be such that $a_1a_2=b_1b$. Since $5\leq |x|_{\mathcal B}\leq 2|x_i''|=2|x_i|$, there exists some $a_w\in\mathcal A(D_i)$ such that $a_1a_2a_3\mid x_i$. Let $c\in\mathcal A(F)$ be such that $a_3c\in\mathcal A(\mathcal B)$ and $a_3c\mid_{\mathcal B}x$, and let $b_3\in\mathcal A(D_i)$ be such that $ba_3=b_2b_3$. If $x=(b_1b)(a_3c)x^*$, where $x^*\in\mathcal A(\mathcal B)$ and $|x^*|_{\mathcal B}\geq 1$, then $x,\,(b_1b_2)(b_3c)x^*,\,y$ is a monotone $\mathcal R$-chain concatenating $x$ and $y$.\\
\indent{\textbf{Case 3b.}} For all $i\in [0,r]$, we have
\[
 \sum_{\substack{\nu=0\\ \nu\neq i}}^r|x_\nu|\geq|x_i|
\quad\Bigg[
\mbox{and thus also }
 \sum_{\substack{\nu=0\\ \nu\neq i}}^r|y_\nu|\geq|y_i|
\Bigg].
\]
We shall prove the following reduction step.
\begin{itemize}
\item[$\mathbf{R1}$] We may assume that, for each $i\in [0,r]$, there is no pure atom in $i$ dividing either $x$ or $y$ in $\mathcal B$.
\end{itemize}
\begin{proof}[Proof of $\mathbf{R1}$]
Let $\tilde x\in\mathsf Z(\mathcal B)$ be such that $(x,\tilde x)\in\sim_{\mathcal B}$, $x\eq\tilde x$, and the number of pure atoms dividing $\tilde x$ is minimal. Assume there is at least one pure atom in $i\in [0,r]$ dividing $\tilde x$, say $a_1a_2\in\mathcal A(\mathcal B)$ with $a_1,\,a_2\in\mathcal A(D_i)$ and $a_1a_2\mid_{\mathcal B}\tilde x$. Now we find
\[
\sum_{\substack{\nu=0\\ \nu\neq i}}^r|\tilde x_\nu|\geq|\tilde x_i|-2,
\]
and thus there are $c_1,\,c_2\in\mathcal A(F)\setminus\mathcal A(D_i)$ with $c_1c_2\in\mathcal A(\mathcal B)$ and $c_1c_2\mid_{\mathcal B}\tilde x$. If $\tilde x=(a_1a_2)(c_1c_2)x^*$, where $x^*\in\mathcal A(\mathcal B)$ and $|x^*|_{\mathcal B}\geq 1$, then we set $x'=(a_1c_1)(a_2c_2)x^*$. Now we find $(\tilde x,x')\in\sim_{\mathcal B}$ and $\tilde x\eq x'$, and thus $x\eq x'$. Since there is one pure atom less dividing $x'$ than $\tilde x$, this is a contradiction.\\
The same argument applies to $y$. Therefore there exist $\tilde x,\,\tilde y\in\mathsf Z(\mathcal B)$ both not divisible by any pure atom such that $(x,\tilde x),\,(y,\tilde y)\in\sim_{\mathcal B}$, $x\eq\tilde x$, and $y\eq\tilde y$. Hence it follows that $(\tilde x,\tilde y)\in\sim_{\mathcal B}$ and if $\tilde x\eq\tilde y$, then $x\eq y$.
\end{proof}
Next we prove the following reduction step.
\begin{itemize}
 \item[$\mathbf{R2}$] We may assume that, for each $i\in [0,r]$, $x_i=y_i$.
\end{itemize}
\begin{proof}[Proof of $\mathbf{R2}$]
Trivially, we have $x_0=y_0$. Now let $i\in [1,r]$. We assert that there is some $\tilde x\in\mathsf Z(\mathcal B)$ such that $(x,\tilde x)\in\sim_{\mathcal B}$, $x\eq\tilde x$, and $z=\gcd(\tilde x_i,y_i)$ (in $F$) is maximal. Now assume that $\tilde x_i=z\tilde z$ and $y_i=z\tilde z'$, where $\tilde z,\,\tilde z'\in\mathsf Z(D_i)$ and $|\tilde z|=|\tilde z'|\geq 1$. If $|\tilde z|=|\tilde z'|=1$, then there are some $v,\,v'\in\mathcal A(D_i)$ such that $\tilde z=v$ and $\tilde z'=v'$. Now we find
\[
 \pi_i(z)v=\pi_i(z\tilde z)=\pi_i(\tilde x_i)=\pi_i(y_i)=\pi_i(z\tilde z')=\pi_i(z)v',
\]
and thus $v=v'$. But then $\gcd(\tilde x_i,y_i)=vz$, a contradiction. If $|\tilde z|=|\tilde z'|\geq 2$, then there are $a_1,\,a_2,\,b\in\mathcal A(D_i)$ with $a_1a_2\mid\tilde x_i$ and $b\mid y_i$. By $\mathbf{R1}$, there are $c_1,\,c_2\in\mathcal A(F)\setminus\mathcal A(D_i)$ such that $a_1c_1,\,a_2c_2\in\mathcal A(\mathcal B)$ and $a_1c_1,\,a_2c_2\mid_{\mathcal B}x$. There is some $b'\in\mathcal A(D_i)$ such that $a_1a_2=b'b$. If $\tilde x=(a_1c_1)(a_2c_2)x^*$, where $x^*\in\mathsf Z(\mathcal B)$ and $|x^*|_{\mathcal B}\geq 1$, then we set $\bar x=(bc_1)(b'c_2)x^*$ and find $(\tilde x,\bar x)\in\sim_{\mathcal B}$ and $\tilde x\eq\bar x$, and thus $x\eq\bar x$. Since $bz=b\gcd(\tilde x_i,y_i)\mid\gcd(\bar x_i,y_i)$, this is a contradiction.
\end{proof}
Now we fix---again arbitrarily---some $i\in [0,r]$ and choose $a\in\mathcal A(D_i)$ such that $a\mid x_i''$. Then $a\mid y_i''$, too. By $\mathbf{R1}$, there are $c,\,d\in\mathcal A(F)\setminus\mathcal A(D_i)$ such that $ac\mid_{\mathcal B}x$ and $ad\mid_{\mathcal B}y$. Again by $\mathbf{R1}$, there are $e,\,f\in\mathcal A(F)$ such that $de\mid_{\mathcal B}x$ and $cf\mid_{\mathcal B}y$.\\
Then $x$ and $y$ are of the following forms
\[
x=(ac)(de)x^*
\quad\mbox{and}\quad
y=(ad)(cf)y^*,
\]
where $x^*,\,y^*\in\mathsf Z(\mathcal B)$ and $|x^*|_{\mathcal B}=|y^*|_{\mathcal B}\geq 1$.\\
\indent\textbf{Case 3.b$\mathbf{'}$.} $ce\in\mathcal A(\mathcal B)$.\\
Then $x,\,(ad)(ce)x^*,\,y$ is a monotone $\mathcal R$-chain concatenating $x$ and $y$.\\
\indent\textbf{Case 3.b$\mathbf{''}$} $df\in\mathcal A(\mathcal B)$.\\
Then $x,\,(ac)(df)y^*,\,y$ is a monotone $\mathcal R$-chain concatenating $x$ and $y$.\\
\indent\textbf{Case 3.b$\mathbf{'''}$} We are neither in Case 3.b$\mathbf{'}$ nor in Case 3.b$\mathbf{''}$, and thus there are $j_1,\,j_2\in [0,r]\setminus\lbrace i\rbrace$ with $j_1\neq j_2$ such that $c,\,e\in\mathcal A(D_{j_1})$ and $d,\,f\in\mathcal A(D_{j_2})$. Then $ae,\,af,\,cd\in\mathcal A(\mathcal B)$ and hence
$x,\,(ae)(cd)x^*,\,(af)(cd)y^*,\,y$ is a monotone $\mathcal R$-chain concatenating $x$ and $y$.
\smallskip

Now it remains to prove the special case, where $|x|_{\mathcal B}=3$. By the length formulas from the beginning of the proof, we find that $|y|_{\mathcal B}\in\lbrace 5,6\rbrace$. If $|y|_{\mathcal B}=5$, then the length formulas imply that that there is some $i\in [1,r]$ such that $|x_i'|=1$ and $|y_i'|\geq 1$, and thus we are in the situation of Case 1.2. When we inspect the monotone $\mathcal R$-chain constructed there, we find that the same monotone $\mathcal R$-chain concatenating $x$ and $y$ exists in our situation. If $|y|_{\mathcal B}=6$, then we find that $|x_i'|=0$ and $|y_i''|=0$ for all $i\in [1,r]$. Since $6=|y|_\mathcal B\geq |x|_\mathcal B+2=5$, we are either in Case 2.2 or in Case 2.4. Again we inspect the monotone $\mathcal R$-chains constructed there and we find that the same monotone $\mathcal R$-chains concatenating $x$ and $y$ exist in our special situation.\\
Now it remains to show part 2. By \cite[Lemma 3.4]{phil11a}, we have
\begin{align*}
\cmon(\mathcal B)\leq\sup\lbrace &|y|\mid(x,y)\in\mathcal A(\sim_{\mathcal B,\mathrm{mon}}),\,\mbox{there is no monotone $\mathcal R$-chain from $x$}\\ &\mbox{to $y$},\,\mbox{and either }|x|=|y|\mbox{ or }|x|,|y|\in\mathsf L(\pi_{\mathcal B}(x))\mbox{ are adjacent}\rbrace.
\end{align*}
By part \ref{3.17.1}, there is a monotone $\mathcal R$-chain concatenating $x$ and $y$ for all $(x,y)\in\sim_\mathcal B$ with $|y|_{\mathcal B}\geq 5$ and $|y|_{\mathcal B}\geq |x|_{\mathcal B}$. Thus it suffices to consider relations $(x,y)\in\sim_\mathcal B$ with $|x|_\mathcal B\leq |y|_\mathcal B\leq 4$. By definition, we have
$\lbrace (x,y)\in\mathcal A(\sim_\mathcal B)\mid |x|_\mathcal B\leq |y|_\mathcal B\rbrace\subset\mathcal A(\sim_{\mathcal B,\mathrm{mon}})$. Since the shortest possible atom $(x,y)\in\mathcal A(\sim_{\mathcal B})\setminus\mathcal A(\sim_{\mathcal B,\mathrm{mon}})$ satisfies $|x|_{\mathcal B}>|y|_{\mathcal B}\geq 2$, we find $|xy|_{\mathcal B}\geq 5$. Hence, we may restrict to elements of $\mathcal A(\sim_{\mathcal B})$, and the assertion follows.
\end{proof}

Using \autoref{3.15}.\ref{3.15.4}, \autoref{3.16}, and \autoref{3.17} above, we can now calculate the catenary degree and the minimum distance (when $|G|=2$), and in a slightly more special but still interesting situation, we can compute the elasticity, the monotone catenary degree and the tame degree.

\begin{proposition}
\label{3.18}
Let $D$ be a monoid, $P\subset D$ a set of prime elements, $r\in\N_0$, $s\in\N_0$, $r+s\geq 1$, and let $D_i\subset[p_i]\times\wmal{D_i}=\widehat{D_i}$ be reduced half-factorial but not factorial monoids of type $(1,k_i)$ for $i\in[1,r+s]$ with $k_1=\ldots=k_r=1$ and $k_{r+1}=\ldots=k_s=2$ such that $D=\mathcal F(P)\times D_1\times\ldots\times D_{r+s}$.
Let $H\subset D$ be a saturated submonoid, let $G=\mathsf q(D/H)$ be its class group with $|G|=2$, say $G=\lbrace\mathbf 0,g\rbrace$, suppose each class in $G$ contains some $p\in P$, and define a homomorphism $\iota:D_1\times\ldots\times D_{r+s}\rightarrow G$ by $\iota(t)=[t]_{D/H}$.
Furthermore, set $I=\lbrace i\in[1,r+s]\mid(\mathcal U_1(D_i)_{\mathbf 0})^2\cap(\mathcal U_1(D_i)_g)^2\neq\emptyset\rbrace$ and $J=\lbrace i\in[r+1,r+s]\mid\mathsf c(D_i)=3\rbrace$.\\
Then
\begin{enumerate}
 \item \label{3.18.1} If $I=J=\emptyset$, then $H$ is half-factorial and $\mathsf c(H)=2$.
 \item \label{3.18.2} If $I=\emptyset$ and $J\neq\emptyset$, then $\mathsf c(H)\in\lbrace 2,3\rbrace$, and $\triangle(H)\subset\lbrace 1\rbrace$.
 \item \label{3.18.3} If $\#I=1$, then $\rho(H)\geq\frac{3}{2}$, $\mathsf c(H)=3$, and $\triangle(H)=\lbrace1\rbrace$.
 \item \label{3.18.4} If $\#I\geq 2$, then $\rho(H)=2$, $\mathsf c(H)=4$, and $\triangle(H)=\lbrace 1,2\rbrace$.
 \item \label{3.18.5} If $s=0$, then $\cmon(H)=\mathsf c(H)$. Additionally, if $\#I=1$, then $\rho(H)=\frac32$.
 \item \label{3.18.6} If $s=0$ and $\iota(p_i)=\mathbf 0$ for all $i\in [1,r]$, then $H$ is half-factorial if and only if $\mathsf t(H)=2$.
\end{enumerate}
In particular, $\min\triangle(H)\leq 1$ always holds.
\end{proposition}
\begin{proof}
We set $\mathcal B=\lbrace S\in\mathcal B(G,D_1\times\ldots\times D_{r+s}\mid\mathbf 0\nmid S\rbrace$. By \autoref{3.5}, $H$ and $D$ are reduced BF-monoids, and $H\subset D$ is a faithfully saturated submonoid. By \autoref{3.5}.\ref{3.5.4}, we obtain $\triangle(H)=\triangle(\mathcal B)$, $\rho(H)=\rho(\mathcal B)$, $\mathsf c(H)=\mathsf c(\mathcal B)$, and $\cmon(H)=\cmon(\mathcal B)$. \autoref{3.15}.\ref{3.15.1} implies $\mathsf c(D)\leq 3$, and, by \autoref{3.15}.\ref{3.15.4}, we obtain $\mathsf c(\mathcal B)=\mathsf c(H)\leq 4$.\\
By \cite[Proposition 14.1]{phil10}, we obtain $\mathsf c(\mathcal B)\leq\sup\lbrace |y|_{\mathcal B}\mid(x,y)\in\mathcal A(\sim_{\mathcal B})$, and since $\mathsf c(\mathcal B)\leq 4$, it follows that
\[
\mathsf c(\mathcal B)\leq\max\lbrace |y|_{\mathcal B}\mid (x,y)\in\mathcal A(\sim_\mathcal B),\,|x|_{\mathcal B}\leq|y|_\mathcal B\leq 4\rbrace;
\]
indeed we can replace the supremum with a maximum since we have a bounded set of integers on the right hand side.

If $(x,y)\in\mathcal A(\sim_{\mathcal B})$, then $(x,y)=(u_1\mdots u_k,v_1\mdots v_l)$, where $k=|x|_{\mathcal B}$, $l=|y|_{\mathcal B}$, and $u_i,\,v_j\in\mathcal A(\mathcal B)$ for all $i\in [1,k]$ and $j\in [1,l]$. In this case, we call the atom $(x,y)$ of type $(k,l)$ and describe it by the defining relation $u_1\mdots u_k=v_1\mdots v_l$ in $\mathcal B$. Now the equation from above reads as follows:
\[
\mathsf c(\mathcal B)=\max\lbrace |x|_{\mathcal B}\mid (x,y)\in\mathcal A(\sim_{\mathcal B})\mbox{ is of type $(k,l)$, where }2\leq k\leq l\leq 4\rbrace.
\]
Hence we proceed with a list of defining relations for all atoms of type $(k,l)$, where $2\leq k\leq l\leq 4$. An atom will be called of character $\mathcal C\in [1,15]$ if it is defined by the relation $(3.1.\mathcal C)$ in the list below.\\
Let $i,\,j\in [1,r+s]$, $i\neq j$. Then
\begin{equation}
\label{1}
g^2(p_ip_j\eps_i\eps_j)=(p_i\eps_ig)(p_j\eps_jg)
\end{equation}
describes an atom of type $(2,2)$ if and only if $\eps_i\in\mathcal U_1(D_i)$, $\eps_j\in\mathcal U_1(D_j)$, and $\iota(p_i\eps_i)=\iota(p_j\eps_j)=g$;
\begin{equation}
\label{2}
(p_ip_j\eps_i^{(1)}\eps_j^{(1)})(p_ip_j\eps_i^{(2)}\eps_j^{(2)})=(p_i^2\eps_i^{(1)}\eps_i^{(2)})(p_j^2\eps_j^{(1)}\eps_j^{(2)})
\end{equation}
describes an atom of type $(2,2)$ if and only if $\iota(p_i\eps_i^{(1)})=\iota(p_i\eps_i^{(2)})=\iota(p_j\eps_j^{(1)})=\iota(p_j\eps_j^{(2)})=g$, $\eps_i^{(1)}\eps_i^{(2)}\notin\mathcal U_1(D_i)^2_{\iota(p_i)}$, and $\eps_j^{(1)}\eps_j^{(2)}\notin\mathcal U_1(D_j)^2_{\iota(p_j)}$;
\begin{equation}
\label{3}
g^2(p_i^2\eps_1\eps_2)=(p_i\eps_1g)(p_i\eps_2g)
\end{equation}
describes an atom of type $(2,2)$ if and only if either $\eps_1,\,\eps_2\in\mathcal U_1(D_i)_{\mathbf 0}$, $\eps_1\eps_2\notin(\mathcal U_1(D_i)_g)^2$, and $\iota(p_i)=g$ or $\eps_1,\,\eps_2\in\mathcal U_1(D_i)_g$, $\eps_1\eps_2\notin(\mathcal U_1(D_i)_{\mathbf 0})^2$, and $\iota(p_i)=\mathbf 0$;
\begin{equation}
\label{4}
(p_i\eps_1)(p_i\eps_2)=(p_i\eta_1)(p_i\eta_2)
\end{equation}
describes an atom of type $(2,2)$ if and only if $\eps_1,\,\eps_2,\,\eta_1,\,\eta_2\in\mathcal U_1(D_i)$, $\iota(p_i)=\iota(\eps_1)=\iota(\eps_2)=\iota(\eta_1)=\iota(\eta_2)$, and $\eps_1\eps_2=\eta_1\eta_2$;
\begin{equation}
\label{5}
(p_i\eps_1g)(p_i\eps_2g)=(p_i\eta_1g)(p_i\eta_2g)
\end{equation}
describes an atom of type $(2,2)$ if and only if $\eps_1,\,\eps_2,\,\eta_1,\,\eta_2\in\mathcal U_1(D_i)$, $\iota(p_i\eps_1)=\iota(p_i\eps_2)=\iota(p_i\eta_1)=\iota(p_i\eta_2)=g$, and $\eps_1\eps_2=\eta_1\eta_2$;
\begin{equation}
\label{6}
(p_i\eps_1)(p_i\eps_2g)=(p_i\eta_1)(p_i\eta_2g)
\end{equation}
describes an atom of type $(2,2)$ if and only if $\eps_1,\,\eps_2,\,\eta_1,\,\eta_2\in\mathcal U_1(D_i)$, $\iota(p_i)=\iota(\eps_1)=\iota(\eta_1)$, $\iota(p_i\eps_2)=\iota(p_i\eta_2)=g$, and $\eps_1\eps_2=\eta_1\eta_2$; and
\begin{equation}
\label{7}
(p_i\eps_1g)(p_i\eps_2g)=(p_i\eta_1)(p_i\eta_2)g^2\,,
\end{equation}
describes an atom of type $(2,3)$ if and only if $\eps_1,\,\eps_2,\,\eta_1,\,\eta_2\in\mathcal U_1(D_i)$, $\eps_1\eps_2=\eta_1\eta_2$, $\iota(p_i\eps_1)=\iota(p_i\eps_2)=g$, and $\iota(p_i\eta_1)=\iota(p_i\eta_2)=\mathbf 0$. If these conditions are fulfilled, then $\eps_1\eps_2\in\mathcal U_1(D_i)_{\mathbf 0}^2\cap\mathcal U_1(D_i)_g^2$ and therefore $i\in I$. Conversely, if $i\in I$, then $\mathcal U_1(D_1)^2_{\mathbf 0}\cap\mathcal U_1(D_i)_g^2\neq\emptyset$. If $\iota(p_i)=g$, let $\eps_1,\,\eps_2\in\mathcal U_1(D_i)_{\mathbf 0}$ and $\eta_1,\,\eta_2\in\mathcal U_1(D_i)_g$ be such that $\eps_1\eps_2=\eta_1\eta_2$. If $\iota(p_i)=\mathbf 0$, let $\eps_1,\,\eps_2\in\mathcal U_1(D_i)_g$ and $\eta_1,\,\eta_2\in\mathcal U_1(D_i)_{\mathbf 0}$ be such that $\eps_1\eps_2=\eta_1\eta_2$. In any case, \eqref{7} holds.
\\
Now let $i\in I$, $j\in[1,r+s]$, and $i\neq j$. Then
\begin{equation}
\label{8}
(p_ip_j\eps_i^{(1)}\eps_j^{(1)})(p_ip_j\eps_i^{(2)}\eps_j^{(2)})=(p_i\eta_i^{(1)})(p_i\eta_i^{(2)})(p_j^2\eps_j^{(1)}\eps_j^{(2)})
\end{equation}
describes an atom of type $(2,3)$ if and only if $\eps_i^{(1)},\eps_i^{(2)},\eta_i^{(1)},\eta_i^{(2)}\in\mathcal U_1(D_i)$, $\eps_j^{(1)},\,\eps_j^{(2)}\in\mathcal U_1(D_j)$, $\eps_i^{(1)}\eps_i^{(2)}=\eta_i^{(1)}\eta_i^{(2)}$, $\iota(p_i\eps_i^{(1)})=\iota(p_i\eps_i^{(2)})=\iota(p_j\eps_j^{(1)})=\iota(p_j\eps_j^{(2)})=g$, $\iota(p_i\eta_i^{(1)})=\iota(p_i\eta_i^{(2)})=\mathbf 0$, and $\eps_j^{(1)}\eps_j^{(2)}\notin\mathcal U_1(D_j)_{\iota(p_j)}^2$. If these conditions are fulfilled, then $\eps_i^{(1)}\eps_i^{(2)}\in\mathcal U_1(D_i)_{\mathbf 0}^2\cap\mathcal U_1(D_i)_g^2$ and therefore $i\in I$. However, if $i\in I$, then a relation \eqref{8} need not hold, since we cannot guarantee that there exist $\eps_j^{(1)},\,\eps_j^{(2)}\in\mathcal U_1(D_j)$ such that $\eps_j^{(1)}\eps_j^{(2)}\notin\mathcal U_1(D_j)_{\iota(p_j)}^2$.
\\
Now let $i,\,j\in I$ and $i\neq j$. Then
\begin{equation}
\label{9}
(p_ip_j\eps_i^{(1)}\eps_j^{(1)})(p_ip_j\eps_i^{(2)}\eps_j^{(2)})=(p_i\eta_i^{(1)})(p_i\eta_i^{(2)})(p_j\eta_j^{(1)})(p_j\eta_j^{(2)})
\end{equation}
describes an atom of type $(2,4)$ if and only if $\eps_i^{(1)},\,\eps_i^{(2)},\,\eta_i^{(1)},\,\eta_i^{(2)}\in\mathcal U_1(D_i)$, $\eps_j^{(1)},\,\eps_j^{(2)},$ $\eta_j^{(1)},$ $\eta_j^{(2)}\in\mathcal U_1(D_j)$, $\eps_i^{(1)}\eps_i^{(2)}=\eta_i^{(1)}\eta_i^{(2)}$, $\eps_j^{(1)}\eps_j^{(2)}=\eta_j^{(1)}\eta_j^{(2)}$, $\iota(p_i\eps_i^{(1)})=\iota(p_i\eps_i^{(2)})=\iota(p_j\eps_j^{(1)})=\iota(p_j\eps_j^{(2)})=g$, and $\iota(p_i\eta_i^{(1)})=\iota(p_i\eta_i^{(2)})=\iota(p_j\eps_j^{(1)})=\iota(p_j\eta_j^{(2)})=\mathbf 0$. If these conditions are fulfilled, then $\eps_i^{(1)}\eps_i^{(2)}\in\mathcal U_1(D_i)_{\mathbf 0}^2\cap\mathcal U_1(D_i)_g^2$ and $\eps_j^{(1)}\eps_j^{(2)}\in\mathcal U_1(D_j)_{\mathbf 0}^2\cap\mathcal U_1(D_j)_g^2$, and therefore $i,\,j\in I$. Conversely, if $i,\,j\in I$, then a relation \eqref{9} holds (see the arguments for \eqref{7}).
Let $i\in J$, $\eps_1,\,\eps_2,\,\eps_3,\,\eta_1,\,\eta_2,\,\eta_3\in\mathcal U_1(D_i)$, and $\mathsf c_{D_i}((p_i\eps_1)(p_i\eps_2)(p_i\eps_3),(p_i\eta_1)(p_i\eta_2)(p_i\eta_3))=3$. Then
\begin{equation}
\label{10}
 (p_i\eps_1)(p_i\eps_2)(p_i\eps_3)=(p_i\eta_1)(p_i\eta_2)(p_i\eta_3)
\end{equation}
describes an atom of type $(3,3)$ if and only if $\iota(p_i)=\iota(\eps_1)=\iota(\eps_2)=\iota(\eps_3)=\iota(\eta_1)=\iota(\eta_2)=\iota(\eta_3)$;
\begin{equation}
\label{11}
 (p_i\eps_1)(p_i\eps_2)(p_i\eps_3g)=(p_i\eta_1)(p_i\eta_2)(p_i\eta_3g)
\end{equation}
describes an atom of type $(3,3)$ if and only if $\iota(p_i)=\iota(\eps_1)=\iota(\eps_2)=\iota(\eta_1)=\iota(\eta_2)$ and $\iota(p_i\eps_3)=\iota(p_i\eta_3)=g$;
\begin{equation}
\label{12}
 (p_i^2\eps_1\eps_2)(p_i\eps_3)=(p_i\eta_1)(p_i\eta_2)(p_i\eta_3)
\end{equation}
describes an atom of type $(2,3)$ if and only if $\eps_1\eps_2\notin(\mathcal U_1(D_i)_{\mathbf 0})^2\cap(\mathcal U_1(D_i)_g)^2$, $\iota(p_i\eps_1)=\iota(p_i\eps_2)=g$, and $\iota(p_i)=\iota(\eps_1)=\iota(\eta_1)=\iota(\eta_2)=\iota(\eta_3)$;
\begin{equation}
\label{13}
 (p_i^2\eps_1\eps_2)(p_i\eps_3g)=(p_i\eta_1)(p_i\eta_2)(p_i\eta_3g)
\end{equation}
describes an atom of type $(2,3)$ if and only if $\eps_1\eps_2\notin(\mathcal U_1(D_i)_{\mathbf 0})^2\cap(\mathcal U_1(D_i)_g)^2$, $\iota(p_i\eps_1)=\iota(p_i\eps_2)=\iota(p_i\eps_3)=\iota(p_i\eta_3)=g$, and $\iota(p_i)=\iota(\eta_1)=\iota(\eta_2)$;
\begin{equation}
\label{14}
 (p_i^2\eps_1\eps_2)(p_i\eps_3)=(p_i^2\eta_1\eta_2)(p_i\eta_3)\,,
\end{equation}
describes an atom of type $(2,2)$ if and only if $\eps_1\eps_2,\,\eta_1\eta_2\notin(\mathcal U_1(D_i)_{\mathbf 0})^2\cap(\mathcal U_1(D_i)_g)^2$, $\iota(p_i\eps_1)=\iota(p_i\eps_2)=\iota(p_i\eta_1)=\iota(p_i\eta_2)=g$, and $\iota(p_i)=\iota(\eps_3)=\iota(\eta_3)$; and
\begin{equation}
\label{15}
 (p_i^2\eps_1\eps_2)(p_i\eps_3g)=(p_i^2\eta_1\eta_2)(p_i\eta_3g)
\end{equation}
describes an atom of type $(2,2)$ if and only if $\eps_1\eps_2,\,\eta_1\eta_2\notin(\mathcal U_1(D_i)_{\mathbf 0})^2\cap(\mathcal U_1(D_i)_g)^2$, and $\iota(p_i\eps_1)=\iota(p_i\eps_2)=\iota(p_i\eps_3)=\iota(p_i\eta_1)=\iota(p_i\eta_2)=\iota(p_i\eta_3)=g$.\\
Now we can do the actual proof.
\begin{enumerate}
\item If $I=J=\emptyset$, then only atoms of characters $[1,6]$ exist, and they are all of type $(2,2)$. Hence, we obtain $\mathsf c(H)=\mathsf c(\mathcal B)=2$, and thus $H$ is half-factorial.
\item If $I=\emptyset$ and $J\neq\emptyset$, then there are atoms of characters $[1,6]\cup [10,15]$, and they are of types $(2,2)$, $(2,3)$, and $(3,3)$. Hence, it follows that $\mathsf c(H)\in\lbrace 2,3\rbrace$ and $\triangle(H)\subset\lbrace 1\rbrace$.
\item If $\#I=1$, then atoms of characters $[1,7]$ exist, and atoms of characters $\lbrace 8\rbrace\cup [10,15]$ might exist. The atoms of characters $[1,7]$ are of types $(2,2)$ and $(2,3)$, and the atoms of characters $\lbrace 8\rbrace\cup [10,15]$ are of types $(2,3)$, $(3,3)$, and $(2,2)$. Thus we have $\rho(H)\geq\frac32$ and $\mathsf c(H)=3$, and therefore $\triangle(H)=\lbrace 1\rbrace$ by \cite[Theorem 1.6.3]{MR2194494}.
\item If $\#I\geq 2$, then atoms of characters $[1,7]\cup\lbrace 9\rbrace$ exist and possibly also atoms of characters $\lbrace 8\rbrace\cup [10,15]$ exist, and they are of types $(2,2)$, $(2,3)$, and $(2,4)$. Thus we find $\mathsf c(H)=4$, $\lbrace 1,2\rbrace\subset\triangle(H)$, and $\rho(H)\geq 2$. Since $\rho(H)\leq 2$ by \autoref{3.15}.\ref{3.15.4}, we obtain the equality $\rho(H)=2$ and, by \cite[Theorem 1.6.3]{MR2194494}, we find $\triangle(H)=\lbrace 1,2\rbrace$.
\item Let $s=0$. If $I=\emptyset$, then $H$ is half-factorial by part \ref{3.18.1}, and thus $\cmon(H)=\mathsf c(H)$ by \cite[Lemma 4.4.1]{phil11a}.

If $\#I=1$, then atoms of characters $[1,7]$ exist, and atoms of character $8$ might exist. The atoms of characters $[1,7]$ are of types $(2,2)$ and $(2,3)$, and the atoms of character $8$ are also of type $(2,3)$. By \autoref{3.17}.\ref{3.17.2}, we have $\cmon(H)=\cmon(\mathcal B)\leq 3$. By part \ref{3.18.3}, we find $3=\mathsf c(H)\leq\cmon(H)$, and thus $\cmon(H)=3$.\\
It remains to show that $\rho(H)=\rho(\mathcal B)=\frac32$. By part \ref{3.18.3}, we have $\rho(H)\geq\frac32$. Thus it suffices to show that $\rho(H)\leq\frac32$. Now let $(x,y)\in\sim_\mathcal B$ with $|y|_{\mathcal B}\geq|x|_{\mathcal B}$. Then there is a monotone $3$-chain concatenating $x$ and $y$, say $x=z_0,\,z_1,\ldots,z_n=y$ with $z_1,\ldots,z_n\in\mathsf Z(\pi_\mathcal B(x))$ and $n\in\N$. Whenever $|z_{i-1}|_{\mathcal B}<|z_i|_{\mathcal B}$ for some $i\in [1,n]$, then $\mathsf d(z_{i-1},z_i)=3$ and there is an atom $(z_{i-1}',z_i')\in\mathcal A(\sim_H)$ of character $7$ or $8$ such that $z_{i-1}=d_iz_{i-1}'$ and $z_i=d_iz_i'$, where $d_i=\gcd(z_{i-1},z_i)$. Since atoms of both characters replace two very special atoms in $\mathcal A(\mathcal B)$ (on the left side) by three different atoms (on the right side) and there is no atom of character $x\in [1,6]$, which generates the first special atoms, there are at most $\frac12|x|_{\mathcal B}$ such steps, and thus $|y|_{\mathcal B}\leq\frac32|x|_{\mathcal B}$, which proves $\rho(H)\leq\frac32$.

If $\#I\geq 2$, then atoms of characters $[1,7]\cup\lbrace 9\rbrace$ exist, and possibly also atoms of character $8$ exist. The atoms of characters $x\in [1,7]\cup\lbrace 9\rbrace$ are of types $(2,2)$, $(2,3)$, and $(2,4)$, and the atoms of character $8$ are of type $(2,3)$. By \autoref{3.17}.\ref{3.17.2}, we have $\cmon(H)=\cmon(\mathcal B)\leq 4$ and, by part \ref{3.18.4}, we obtain $4=\mathsf c(H)\leq\cmon(H)$, and thus $\cmon(H)=4$.
\end{enumerate}
\smallskip

\noindent
In order to finish the proof, we need an additional Lemma.

\begin{lemma}
\label{3.19}
Let $D$ be a monoid, $P\subset D$ be a set of prime elements, $r\in\N$, and let $D_i\subset\widehat{D_i}=[p_i]\times\wmal{D_i}$ be reduced half-factorial monoids of type $(1,1)$ for all $i\in[1,r]$ such that $D=\mathcal F(P)\times D_1\times\ldots\times D_r$. Let $H\subset D$ be a saturated submonoid, let $G=\mathsf q(D/H)$ be its class group with $|G|=2$, say $G=\lbrace\mathbf 0,g\rbrace$, suppose each class in $G$ contains some $p\in P$, and define a homomorphism $\iota:D_1\times\ldots\times D_r\rightarrow G$ by $\iota(t)=[t]_{D/H}$. Furthermore, let $\mathcal B(G,D_1\times\ldots\times D_r,\iota)$ be the $(D_1\times\ldots\times D_r)$-block monoid defined by $\iota$ over $G$ and suppose $\mathcal B$ is half-factorial but not factorial.\\
Then $\mathsf t(H)=\mathsf t(\mathcal B)=2$.
\end{lemma}
\begin{proof}
Throughout the proof, we write $\mathcal B=\lbrace S\in\mathcal B(G,D_1\times\ldots\times D_{r+s},\iota)\mid\mathbf 0\nmid S\rbrace$ as in \autoref{3.5}.\ref{3.5.4}.
By \autoref{3.18}.\ref{3.18.1}-\ref{3.18.4}, we find that $\lbrace i\in[1,r]\mid(\mathcal U_1(D_i))_{\mathbf 0}\cap(\mathcal U_1(D_i))_g\neq\emptyset\rbrace=\emptyset$, and thus $\iota(\wmal{D_i})=\lbrace\mathbf 0\rbrace$ for all $i\in[1,r]$. Now let $h\in H$, $z\in\mathsf Z(h)$, and $a\in\mathcal A(H)$ be such that $a\mid h$. Then we prove that $\mathsf d(z,z')\leq 2$ for some $z'\in\mathsf Z(h)\cap a\mathsf Z(H)$. We may assume that $a\nmid z$. We find that $z$ is of the following form:
\[
z=q_1\mdots q_k(q_1^{(1)}q_1^{(2)})\mdots(q_l^{(1)}q_l^{(2)})t_1\mdots t_m,
\]
where $q_1,\ldots,q_k,\,q_1^{(1)},\,q_1^{(2)},\ldots,q_l^{(1)},\,q_l^{(2)}\in P$, $[q_1]_{D/H}=\ldots=[q_k]_{D/H}=\mathbf 0$, $[q_1^{(1)}]_{D/H}=[q_1^{(2)}]_{D/H}=\ldots=[q_l^{(1)}]_{D/H}=[q_l^{(2)}]_{D/H}=g$, and $t_1,\ldots,t_m\in\mathcal A(D_1\times\ldots\times D_r)$. Now we have the following three possibilities for $a$:
\begin{align*}
a &= \bar q & \mbox{ with } \bar q\in P\mbox{ and }[\bar q]_{D/H}=\mathbf 0,\mbox{ or} \\
a &= \bar q^{(1)}\bar q^{(2)} & \mbox{ with } \bar q^{(1)}\mbox{ and }\bar q^{(2)}\in P,\,[\bar q^{(1)}]_{D/H}=[\bar q^{(2)}]_{D/H}=g, \mbox{ or}\\
a &= u & \mbox{ with }u\in\mathcal A(D_i)\mbox{ for some } i\in[1,r].
\end{align*}
We proceed case by case. Let $a=\bar q$, where $\bar q\in P$ and $[\bar q]_{D/H}=\mathbf 0$. Since $q_1,\ldots,q_k,$ $q_1^{(1)},$ $q_1^{(2)},\ldots,q_l^{(1)},$ $q_l^{(2)}\in P$ are prime in $D$ and since $[\bar q]_{D/H}=\mathbf 0$, we find $\bar q\in\lbrace q_1,\ldots,q_k\rbrace$. Thus $a=\bar q\mid z$, a contradiction. Let $a=\bar q^{(1)}\bar q^{(2)}$, where $\bar q^{(1)},\,\bar q^{(2)}\in P$ and $[\bar q^{(1)}]_{D/H}=[\bar q^{(2)}]_{D/H}=g$. By the same arguments as before, we find $\bar q^{(1)},\,\bar q^{(2)}\in\lbrace q_1^{(1)},q_1^{(2)},\ldots,q_l^{(1)},q_l^{(2)}\rbrace$. Since $a\nmid z$, there is no $i\in [1,l]$ such that without loss of generality $\bar q^{(j)}=\bar q_i^{(j)}$ for $j=1,2$. Thus there are $i,\,j\in [1,l]$ with $i\neq j$ such that without loss of generality $\bar q^{(1)}=q_i^{(1)}$ and $\bar q^{(2)}=q_j^{(2)}$. Now we find the factorization $z'\in\mathsf Z(h)$,
\[
z' = q_1\mdots q_k
(q_i^{(1)}q_j^{(2)})(q_j^{(1)}q_i^{(2)})\prod_{s=1\\s\neq i,j}^{s=l}(q_s^{(1)}q_s^{(2)})
t_1\mdots t_m,
\]
such that $\mathsf d(z,z')=2$ and $a\mid z'$. Lastly, we consider the case $a=u$ with $u\in\mathcal A(D_i)$ for some $i\in [1,r]$. Then there are $u_1,\ldots,u_{\bar m}\in\mathcal A(D_1\times\ldots\times D_r)$ such that
\[
t_1\mdots t_m=uu_1\mdots u_{\bar m}\quad\mbox{and }\mathsf d_{D_1\times\ldots\times D_r}(t_1\mdots t_m,uu_1\mdots u_{\bar m})\leq 2.
\]
Now we find a factorization $z'\in\mathsf Z(h)$ by setting
\[
z'=q_1\mdots q_k(q_1^{(1)}q_1^{(2)})\mdots(q_l^{(1)}q_l^{(2)})uu_1\mdots u_{\bar m},
\]
and $\mathsf d(z,z')\leq 2$ follows.
\end{proof}

\begin{enumerate}
\item[6.]
Let $s=0$ and $\iota(p_i)=\mathbf 0$ for all $i\in [1,r]$. If $H$ is not half-factorial, then $\mathsf c(H)\geq 3$ and therefore $\mathsf t(H)\geq 3$. Otherwise, if $H$ is half-factorial, then $\mathsf c(H)=\mathsf c(\mathcal B)=2$, and therefore $I=\emptyset$ by points~\ref{3.18.1}-\ref{3.18.4}. Thus $\iota(u)=\iota(p_i)=\mathbf 0$ for all $u\in\mathcal A(D_i)$ and $i\in[1,r]$, and any $a\in\mathcal A(\mathcal B)$ is either of the form $a=g^2$ or $a=u$ with $u\in\mathcal A(D_i)$ for some $i\in[1,r]$. Since, by \autoref{3.9}.\ref{3.9.2}, $\mathsf t(D_i)=2$ for all $i\in[1,r]$, we have $\mathsf t(\mathcal B)=2$. Now the assertion follows by \autoref{3.19}.
\qedhere
\end{enumerate}
\end{proof}

The following example shows that the very special structure of $D$ in the hypothesis of \autoref{3.19}---in terms of \autoref{3.20}, the structure $T$---is definitely necessary for the assertion of \autoref{3.19} to hold.

\begin{example}
\label{3.20}
Let $P$ be a set of prime elements and let $T$ be an atomic monoid such that $D=\mathcal F(P)\times T$. Let $H\subset D$ be a saturated submonoid with class group $D/H=C_2$ such that each class in $C_2$ contains some $p\in P$. Let $\iota:T\rightarrow C_2$, $t\mapsto [t]_{D/H}$ be a homomorphism and $\mathcal B(C_2,T,\iota)$ the $T$-block monoid over $C_2$ defined by $\iota$. Furthermore let $\mathsf t(\mathcal B(C_2,T,\iota))=2$.\\
This situation does not imply $\mathsf t(H)=2$.
\end{example}
\begin{proof}
We write $C_2=\lbrace\mathbf 0,g\rbrace$ and we set $\mathcal B=\mathcal B(C_2,T,\iota)$ and denote by $\beta:H\rightarrow\mathcal B$ the block homomorphism of $H$ and by $\bar\beta:\mathsf Z(H)\rightarrow\mathsf Z(\mathcal B)$ the canonical extension of the block homomorphism.\\
By definition, it is sufficient to prove $\mathsf t(a,v)\geq 3$ for some $a\in H$ and some $v\in\mathcal A(H)$. Let $a\in H$ and $v\in\mathcal A(H)$.\\
We have the following four types of atoms of $H$ which are not prime:
\begin{align*}
& v=p_1p_2 & \mbox{with }p_1,\,p_2\in P\mbox{ and }[p_1]_{D/H}=[p_2]_{D/H}=g \\
& v=pt & \mbox{with }p\in P,\,t\in T\mbox{ and }[p]_{D/H}=[t]_{D/H}=g \\
& v=t_1t_2 & \mbox{with }t_1,\,t_2\in\mathcal A(T)\mbox{ and }[t_1]_{D/H}=[t_2]_{D/H}=g \\
& v=t & \mbox{with }t\in\mathcal A(T)\mbox{ and }[t]_{D/H}=g
\end{align*}
Let $z\in\mathsf Z_H(a)$. Without loss of generality, we may assume that no prime element divides $a$. Then $z$ is of the following form:
\[
z = (p_1p_2)\mdots (p_{l-1}p_l)(p_{l+1}s_1)\mdots (p_{l+m}s_m)(t_1t_2)\mdots (t_{n-1}t_n)u_1\mdots u_o.
\]
Let $v=q_1q_2$ be of the first type. Since all $p\in P$ are prime in $D$, we find $i,\,j\in[1,l+m]$ such that $p_i=q_1$ and $p_j=q_2$. Assume $i=l+1$ and $j=l+2$. Then we find
\[
z'=(p_{l+1}p_{l+2})(p_1s_1)(p_2s_2)(p_1p_2)^{-1}(p_{l+1}s_1)^{-1}(p_{l+2}s_2)^{-1}z.
\]
Thus $\mathsf d(z,z')=3$. If we apply $\bar\beta$ to $z'$, we find
\[
\bar\beta(z')=g^2(gs_1)(gs_2)(g^2)^{-1}(gs_1)^{-1}(gs_2)^{-1}\bar\beta(z)=\bar\beta(z),
\]
and thus $\mathsf d(\bar\beta(z),\bar\beta(z'))=0$.
\end{proof}

\begin{corollary}
\label{3.21}
Let $D$ be an atomic monoid, $P\subset D$ a set of prime elements, $r\in\N$, and let $D_i\subset[p_i]\times\wmal{D_i}=\widehat{D_i}$ be reduced half-factorial monoids of type $(1,1)$ for all $i\in[1,r]$ such that $D=\mathcal F(P)\times D_1\times\ldots\times D_r$. Let $H\subset D$ be a saturated atomic submonoid, $G=\mathsf q(D/H)$ its class group, and suppose each class in $G$ contains some $p\in P$.\\
Then the following are equivalent:
\begin{itemize}
 \item $\cmon(H)\leq 2$.
 \item $\mathsf c(H)\leq 2$.
 \item $H$ is half-factorial.
\end{itemize}
If, additionally, $[p_i]_{D/H}=\mathbf 0_{D/H}$ for all $i\in [1,r]$---in particular, this is true if $|G|=1$---then the following is also equivalent:
\begin{itemize}
 \item $\mathsf t(H)\leq 2$.
\end{itemize}
\end{corollary}
\begin{proof}
By \autoref{3.15}.\ref{3.15.3}, $|G|\geq 3$ implies $\mathsf c(H)\geq 3$ and thus that $H$ is never half-factorial. Thus we have $|G|\leq 2$. If $|G|=2$, then the assertion follows by \autoref{3.18}. If $|G|=1$, the assertion follows by \autoref{3.15}.\ref{3.15.2} and \autoref{3.15}.\ref{3.15.1}.
\end{proof}

\begin{lemma}
\label{3.22}
Let $\mathcal O$ be a locally half-factorial order in an algebraic number field.\\
Then there is a monoid $D$, a set of prime elements $P\subset D$, $r\in\N$, and reduced half-factorial but not factorial monoids $D_i\subset[p_i]\times\wmal{D_i}=\widehat{D_i}$ of type $(1,k_i)$ with $k_i\in\lbrace 1,2\rbrace$ for all $i\in[1,r]$ such that $D=\mathcal F(P)\times D_1\times\ldots\times D_r$, $\mathcal I^*(\mathcal O)\cong D$, $\punkt{\mathcal O}_{\red}\subset D$ is a saturated submonoid, $\pic(\mathcal O)=\mathsf q(D/\punkt{\mathcal O}_{\red})$ is its class group, and each class contains some $p\in P$.\\
If, additionally, all localizations of $\mathcal O$ are finitely primary monoids of exponent $1$, then $k_i=1$ for all $i\in[1,r]$.
\end{lemma}
\begin{proof}
Let $\mathcal O$ be an order in an algebraic number field and suppose $\mathcal I^*(\mathcal O)$ is half-factorial. We set $\overline{\mathcal O}$ for the integral closure of $\mathcal O$ and set $\mathfrak f=(\mathcal O:\overline{\mathcal O})$, $\mathcal P=\lbrace p\in\mathfrak X(\mathcal O)\mid\mathfrak p\not\supset\mathfrak f\rbrace$, $\mathcal P^*=\lbrace\mathfrak p\in\mathfrak X(\mathcal O)\mid\mathfrak p\supset\mathfrak f\rbrace$, and
\[
T=\prod_{\mathfrak p\in\mathcal P^*}(\punkt{\mathcal O_{\mathfrak p}})_{\mathrm{red}}.
\]
By \cite[Theorem 3.7.1]{MR2194494}, we find that $\mathcal P^*$ is finite, $\punkt{\mathcal O}_{red}\subset\mathcal F(\mathcal P)\times T$ is a saturated and cofinal submonoid, $\pic(\mathcal O)=\mathcal C_{\mathsf v}(\mathcal O)=(\mathcal F(\mathcal P)\times T)/\punkt{\mathcal O}_{red}$, and, for all $\mathfrak p\in\mathfrak X(\mathcal O)$, $\punkt{\mathcal O_{\mathfrak p}}$ is a finitely primary monoid of rank $s_{\mathfrak p}$, where $s_{\mathfrak p}$ is the number of prime ideals $\overline{\mathfrak p}\in\mathfrak X(\overline{\mathcal O})$ such that $\overline{\mathfrak p}\cap\mathcal O=\mathfrak p$. For $\mathfrak p\in\mathcal P^*$, the local domain $\mathcal O_{\mathfrak p}$ is not integrally closed, hence not factorial, and therefore the monoid $(\punkt{\mathcal O_{\mathfrak p}})_{\red}$ is not factorial either. Since $\mathcal I^*(\mathcal O)\cong\prod_{\mathfrak p\in\mathfrak X(\mathcal O)}(\punkt{\mathcal O_{\mathfrak p}})_{\mathrm{red}}$ by \cite[Theorem 3.7.1]{MR2194494}, we find, for all $\mathfrak p\in\mathfrak X(\mathcal O)$, that $\mathcal O_{\mathfrak p}$ is half-factorial, and thus, by the additional statement in \autoref{3.9}.\ref{3.9.1}, $\punkt{\mathcal O_{\mathfrak p}}$ is a half-factorial monoid of type $(1,k_{\mathfrak p})$, where $k_{\mathfrak p}$ is the rank of $\punkt{\mathcal O_{\mathfrak p}}$. By \cite[Corollary 3.5]{MR2140704}, we find $k_{\mathfrak p}\leq 2$. Now we set $D_i=(\punkt{\mathcal O_{\mathfrak p}})_{\mathrm{red}}$ for some $\mathfrak p\in\mathcal P^*$ such that $T=D_1\times\ldots\times D_r$ and we set $P=\mathcal P$. By \cite[Corollary 2.11.16]{MR2194494}, every class contains infinitely many primes $p\in P$. Since, by the above, $k_i$ is the exponent of $D_i$ for all $i\in[1,r]$, the additional statement is obvious.
\end{proof}

\medskip
\subsection*{The final proof of the main theorem}

\begin{proof}[Final proof of \autoref{main}]
By \autoref{3.22}, there is a monoid $D$, a set of prime elements $P\subset D$, $r\in\N$, and reduced half-factorial but not factorial monoids $D_i\subset[p_i]\times\wmal{D_i}=\widehat{D_i}$ of type $(1,k_i)$ with $k_i\in\lbrace 1,2\rbrace$ for all $i\in[1,r]$ such that $D=\mathcal F(P)\times D_1\times\ldots\times D_r$, $\mathcal I^*(\mathcal O)\cong D$, $\punkt{\mathcal O}_{\red}\subset D$ is a saturated submonoid, $\pic(\mathcal O)=\mathsf q(D/\punkt{\mathcal O}_{\red})$ is its class group, and each class contains some $p\in P$.
\begin{enumerate}
 \item If $|\pic(\mathcal O)|=1$, then the assertion follows by \autoref{3.15}.\ref{3.15.2}.
 \item If $|\pic(\mathcal O)|\geq 3$, then the assertion follows by \autoref{3.15}.\ref{3.15.3}.
 \item If $|\pic(\mathcal O)|=2$, then $\rho(\mathcal O)\leq 2$ and $2\leq\mathsf c(\mathcal O)\leq 4$ follow by \autoref{3.15}.\ref{3.15.4}. If, additionally, all localizations of $\mathcal O$ are finitely primary monoids of exponent $1$, then, by \autoref{3.22}, we have $k_i=1$ for all $i\in[1,r]$. If $k=0$, then we are in the situation of \autoref{3.18}.\ref{3.18.1}, and thus $\mathcal O$ is half-factorial, $\mathsf c(\mathcal O)=2$, and $\triangle(\mathcal O)=\emptyset$. If $k\geq 2$, then we are in the situation of \autoref{3.18}.\ref{3.18.4}, and thus $\rho(\mathcal O)=2$, $\mathsf c(\mathcal O)=4$, and $\triangle(\mathcal O)=\lbrace 1,2\rbrace$. If $k=1$, then we are in the situation of \autoref{3.18}.\ref{3.18.3}, and thus $\rho(\mathcal O)\geq\frac32$, $\mathsf c(\mathcal O)=3$, and $\triangle(\mathcal O)=\lbrace 1\rbrace$. Since $k_i=1$ for all $i\in [1,r]$, we may use \autoref{3.18}.\ref{3.18.5}. Thus we find $\rho(\mathcal O)=\frac32$ if $k=1$ and $\cmon(\mathcal O)=\mathsf c(\mathcal O)$ in all cases. Putting all this together, we obtain the formulas in the assertion. The equivalence of the four statements follows by \autoref{3.21}.
\end{enumerate}
In particular, in all situations, we find $\min\triangle(\mathcal O)\leq 1$.
\end{proof}

\bigskip
\section{Consequences and refinements of the main theorem}
\bigskip

In the case of quadratic and cubic number fields, we can do even better.
First, we recall and reformulate a definition and the key result from \cite{MR2140704}.\\
Let $\mathcal O$ be an order in an algebraic number field and $\mathfrak p\in\mathfrak X(\mathcal O)$. Then we call $\mathcal O_{\mathfrak p}$ a local order. Now let $\mathcal O_{\mathfrak p}$ be a local order such that its integral closure $\overline{\mathcal O_{\mathfrak p}}$ is local too. Now we fix the following notations. We denote by $\mathfrak m$ respectively $\overline{\mathfrak m}$ the maximal ideal of $\mathcal O_{\mathfrak p}$ respectively $\overline{\mathcal O_{\mathfrak p}}$, by $k=\mathcal O_{\mathfrak p}/\mathfrak m$ and $\overline{k}=\overline{\mathcal O_{\mathfrak p}}/\overline{\mathfrak m}$ the residue class fields, and by $\pi:\overline{\mathcal O_{\mathfrak p}}\rightarrow\overline{k}$ the canonical homomorphism. For a prime $p\in\overline{\mathcal O_{\mathfrak p}}$ and $i\in\N$, we set
\[
U_{i,p}(\mathcal O_{\mathfrak p})=\lbrace \eps\in\mal{\overline{\mathcal O_{\mathfrak p}}}\mid\eps p^i\in\mathcal O_{\mathfrak p}\rbrace
\quad\mbox{and}\quad
V_{i,p}(\mathcal O_{\mathfrak p})=\pi(U_{i,p}(\mathcal O_{\mathfrak p}))\cup\lbrace 0\rbrace,
\]
as in \cite{MR2140704}. Then $V_{i,p}(\mathcal O_{\mathfrak p})$ is a $k$-subspace of $\overline{k}$ by \cite{MR2140704}.

\begin{lemma}[{\cite[Theorem 3.3]{MR2140704}}]
\label{4.1}
Using the above notations, the following are equivalent:
\begin{enumerate}
 \item\label{4.1.1} $\mathcal O_{\mathfrak p}$ is half-factorial.
 \item\label{4.1.2} $(U_{1,p}(\mathcal O_{\mathfrak p}))^2=\mal{\overline{\mathcal O_{\mathfrak p}}}$.
 \item\label{4.1.3} $\lbrace xy\mid (x,y)\in V_{1,p}(\mathcal O_{\mathfrak p})\times V_{1,p}(\mathcal O_{\mathfrak p})\rbrace=\overline{k}$.
\end{enumerate}
\end{lemma}

\begin{lemma}
\label{4.2}
Let $\mathcal O$ be an order in an algebraic number field and $\mathfrak p\in\mathfrak X(\mathcal O)$ such that $\mathcal O_{\mathfrak p}$ is half-factorial.
\begin{enumerate}
 \item\label{4.2.1} $\overline{\mathcal O_{\mathfrak p}}$ is local and every atom of $\mathcal O_{\mathfrak p}$ is a prime of $\overline{\mathcal O_{\mathfrak p}}$.
 \item\label{4.2.2} Let $\mathfrak m$ respectively $\overline{\mathfrak m}$ be the maximal ideals of $\mathcal O_{\mathfrak p}$ respectively $\overline{\mathcal O_{\mathfrak p}}$ and let $k=\mathcal O_{\mathfrak p}/\mathfrak m$ and $\overline{k}=\overline{\mathcal O_{\mathfrak p}}/\overline{\mathfrak m}$ be the residue class fields.\\
 If $\dim_k\overline{k}\leq 3$, then $\punkt{\mathcal O_{\mathfrak p}}\subset\punkt{\overline{\mathcal O_{\mathfrak p}}}$ is a finitely primary monoid of exponent $1$.\\
 In particular, if $\mathcal O$ is an order in a quadratic or cubic number field, then $\punkt{\mathcal O_{\mathfrak p}}\subset\punkt{\overline{\mathcal O_{\mathfrak p}}}$ is a finitely primary monoid of exponent $1$.
\end{enumerate}
\end{lemma}

Whenever $\mathcal O_{\mathfrak p}$ is a Cohen-Kaplansky domain, i.e., whenever it has up to units only finitely many atoms, the result from \autoref{4.2}.\ref{4.2.1} can be found in \cite[Theorem 6.3]{MR1161563}.

\begin{proof}[Proof of \autoref{4.2}]
\mbox{}
\begin{enumerate}
 \item The assertion follows by \autoref{3.9}.\ref{3.9.1}.
 \item By part~\ref{4.2.1}, $\overline{\mathcal O_{\mathfrak p}}$ is local too. Thus $\mathfrak m$ respectively $\overline{\mathfrak m}$ is well-defined and $k$ respectively $\overline{k}$ is a field. Since $\overline{\mathcal O_{\mathfrak p}}$ has up to units only one prime element by part~\ref{4.2.1}, we write $V_1(\mathcal O_{\mathfrak p})$ instead of $V_{1,p}(\mathcal O_{\mathfrak p})$ and $U_1(\mathcal O_{\mathfrak p})$ instead of $U_{1,p}(\mathcal O_{\mathfrak p})$. Furthermore, we find $\mathcal U_1(\punkt{\mathcal O_{\mathfrak p}})=U_1(\mathcal O_{\mathfrak p})$. For short, we write $m=\dim_k\overline{k}$, $n=\dim_kV_1(\mathcal O_{\mathfrak p})$, and $q=\#k$. Now we distinguish three cases by $m$.\\
\textbf{Case 1} $m=1$.
Here $k=\overline{k}$ and therefore $V_1(\mathcal O_{\mathfrak p})=\overline{k}$. Thus $U_1(\mathcal O_{\mathfrak p})=\mal{\overline{\mathcal O_{\mathfrak p}}}$ by \cite[Lemma 3.2]{MR2140704}, and therefore $\punkt{\mathcal O_{\mathfrak p}}\subset\overline{\punkt{\mathcal O_{\mathfrak p}}}$ is of exponent $1$ by the additional statement of \autoref{3.9}.\ref{3.9.1}.\\
\textbf{Case 2} $m=2$.
If $n=1$, then $V_1(\mathcal O_{\mathfrak p})=k$, and therefore $V_1(\mathcal O_{\mathfrak p})*V_1(\mathcal O_{\mathfrak p})=k\neq\overline{k}$, a contradiction to \autoref{4.1}.\ref{4.1.3}. If $n=2$, then $V_1(\mathcal O_{\mathfrak p})=\overline{k}$, and the assertion follows as in Case 1.\\
\textbf{Case 3} $m=3$.
If $n=1$, then we find the same contradiction as in Case 2 when $n=1$ there. If $n=2$, then $\#(V_1(\mathcal O_{\mathfrak p})*V_1(\mathcal O_{\mathfrak p}))<q^3=\#\overline{k}$ by \cite[Lemma 2.5]{MR2140704}. This is again a contradiction to \autoref{4.1}.\ref{4.1.3}. If $n=3$, then $V_1(\mathcal O_{\mathfrak p})=\overline{k}$, and the assertion follows as in Case 1.\\
Let $K$ be the algebraic number field containing $\mathcal O$. Then we find $m\leq[K:\Q]$ and the assertion follows.
\qedhere
\end{enumerate}
\end{proof}

Now we can prove a slightly refined version of \autoref{main} for orders in quadratic and cubic number fields.

\begin{corollary}
 \label{4.3}
Let $\mathcal O$ be a non-principal, locally half-factorial order in a quadratic or cubic number field and set $\mathcal P^*=\lbrace\mathfrak p\in\mathfrak X(\mathcal O)\mid\mathfrak p\supset(\mathcal O:\overline{\mathcal O})\rbrace$.\\
\begin{enumerate}
 \item\label{4.3.1} If $|\pic(\mathcal O)|=1$, then $\mathcal O$ is half-factorial.
 \item\label{4.3.2} If $|\pic(\mathcal O)|\geq 3$, then $(\mathsf D(\pic(\mathcal O)))^2\geq\mathsf c(\mathcal O)\geq 3$, and $\min\triangle(\mathcal O)=1$.
 \item\label{4.3.3} If $|\pic(\mathcal O)|=2$, then, setting $k=\#\lbrace\mathfrak p\in\mathcal P^*\mid [\mal{\overline{\mathcal O}_{\mathfrak p}}/\mal{\mathcal O_{\mathfrak p}}]_{\pic(\mathcal O)}=\pic(\mathcal O)\rbrace$, it follows that
\begin{itemize}
 \item $\cmon(H)=\mathsf c(\mathcal O)=2+\min\lbrace 2,k\rbrace\in\lbrace 2,3,4\rbrace$;
 \item $\rho(\mathcal O)=\frac{1}{2}\mathsf c(\mathcal O)\in\lbrace 1,\frac{3}{2},2\rbrace$; 
 \item $\triangle(\mathcal O)=[1,\mathsf c(\mathcal O)-2]\subset[1,2]$.
\end{itemize}
\end{enumerate}
In particular, $\min\triangle(\mathcal O)\leq 1$ always holds, and the following are equivalent:
\begin{itemize}
 \item $\cmon(\mathcal O)=2$.
 \item $\mathsf c(\mathcal O)=2$.
 \item $\mathcal O$ is half-factorial.
\end{itemize}
If, additionally, $[\mathfrak p]=\mathbf 0_{\pic(\mathcal O)}$ for all $\mathfrak p\in\mathcal P^*$---this is always true if $|\pic(\mathcal O)|=1$ or if $\mathcal O$ is an order in a quadratic number field---then the following is also equivalent:
\begin{itemize}
 \item $\mathsf t(\mathcal O)=2$.
\end{itemize}
\end{corollary}
\begin{proof}
Part \ref{4.3.1} respectively part~\ref{4.3.2} follows immediately from \autoref{main}.\ref{main1} respectively \autoref{main}.\ref{main2}. By \autoref{4.2}.\ref{4.2.2}, all localizations $\mathcal O_{\mathfrak p}$ for $\mathfrak p\in\mathfrak X(\mathcal O)$ are finitely primary monoids of exponent (at most) $1$. Thus part~\ref{4.3.3} follows by the additional statement of \autoref{main}.\ref{main3}.\\
Now we prove the additional statement. First note $\min\triangle(\mathcal O)\leq 1$ follows by the additional statement of \autoref{main}. If $|\pic(\mathcal O)|\geq 3$, then none of the equivalent conditions holds by part 2. If $|\pic(\mathcal O)|=2$, then the four equivalent conditions are shown in the additional statement of \autoref{main}.\ref{main3}. If $|\pic(\mathcal O)|=1$, then $\mathcal O\cong\mathcal I^*(\mathcal O)$, and therefore $\mathcal O$ is half-factorial. By \autoref{3.22} and \autoref{4.2}.\ref{4.2.2}, there is a monoid $D$, a set of prime elements $P\subset D$, $r\in\N$, and reduced half-factorial but not factorial monoids $D_i\subset[p_i]\times\wmal{D_i}=\widehat{D_i}$ of type $(1,k_i)$ with $k_i\in\lbrace 1,2\rbrace$ for all $i\in[1,r]$ such that $D=\mathcal F(P)\times D_1\times\ldots\times D_r$ and $\mathcal I^*(\mathcal O)\cong D$. Now the other equivalent conditions follow by \autoref{3.9}.\ref{3.9.2}.
\end{proof}

If we compare the equivalent conditions in \autoref{4.3} for non-principal, locally half-factorial orders in quadratic or cubic number fields with the ones given in \cite[Theorem 1.7.3.6]{MR2194494}---see below for principal orders in algebraic number fields---we see that at least these special non-principal orders behave nearly the same as the principal ones.
\begin{theorem}[cf. \protect{\cite[Theorem 1.7.3.6]{MR2194494}}]
\label{4.4}
Let $\mathcal O$ be a principal order in a quadratic or cubic number field.\\
Then the following are equivalent.
\begin{enumerate}
 \item $\mathcal O$ is half-factorial.
 \item $|\pic(\mathcal O)|\leq 2$.
 \item $\mathsf t(\mathcal O)\leq 2$.
 \item $\mathsf c(\mathcal O)\leq 2$.
\end{enumerate}
\end{theorem}

By \autoref{4.3}.\ref{4.3.3}, we get an additional bound on the elasticity of a non-principal order $\mathcal O$ in a quadratic or cubic number field such that its conductor is an inert prime ideal, say $(\mathcal O:\overline{\mathcal O})=\mathfrak p\in\mathfrak X(\mathcal O)$ and $\mathfrak p\overline{\mathcal O}\in\mathfrak X(\overline{\mathcal O})$; then $\rho(\mathcal O)\leq\frac{3}{2}$. Now we revisit the example from \cite[example at the end of the publication]{MR1378586}: Let $\mathcal O=\Z[3i]$. Then $\overline{\mathcal O}=\Z[i]$, $|\pic(\mathcal O)|=2$, $\mathcal O$ is locally half-factorial, and $(\mathcal O:\overline{\mathcal O})=3\mathcal O\in\mathfrak X(\mathcal O)$ is an inert prime ideal in $\overline{\mathcal O}$. We set $\beta=1+2i$ and $\beta'=1-2i$. Then $3\beta$, $3\beta'$, $3$, and $5$ are irreducible elements of $\mathcal O$ satisfying $(3\beta)(3\beta')=3^2\cdot 5$; thus $\rho(\mathcal O)\geq\frac{3}{2}$. Now we have equality by \autoref{4.3}.\ref{4.3.3}.

\medskip
\subsection{Localizations of half-factorial orders}

\begin{proposition}
\label{4.5}
Let $D$ be a monoid, $P\subset D$ be a set of prime elements, and let $T\subset D$ be a reduced atomic submonoid such that $D=\mathcal F(P)\times T$. Let $D_1\subset T$ be a divisor-closed submonoid and let $D_1\subset[p]\times\wmal{D_1}=\widehat{D_1}$ be a finitely primary monoid of rank $1$ and exponent $k$. Let $H\subset D$ be a saturated half-factorial submonoid, $G=\mathsf q(D/H)$ its class group, and suppose each class in $G$ contains some $p'\in P$.\\
Then $|G|\leq 2$ and $D_1$ is either
\begin{itemize}
 \item half-factorial or
 \item $|G|=2$, say $G=\lbrace\mathbf 0,g\rbrace$, $\mathsf v_p(\mathcal A(D_1))=\lbrace 1,2\rbrace$, $[p]_{D/H}=g$, and $[\eps]_{D/H}=\mathbf 0$ for all $\eps\in\wmal{D_1}$.
\end{itemize}
\end{proposition}
\begin{proof}
Define a homomorphism $\iota:T\rightarrow G$ by $\iota(t)=[t]_{D/H}$. Throughout the proof, we write $\mathcal B=\lbrace S\in\mathcal B(G,T,\iota)\mid\mathbf 0\nmid S\rbrace$ as in \autoref{3.5}.\ref{3.5.4}. If $|G|\geq 3$, then it follows by \autoref{3.6}.\ref{3.6.1} that $H$ is not half-factorial. If $|G|=1$, then $H=D$ and, obviously, the first case of the assertion holds. Now, let $|G|=2$, say $G=\lbrace\mathbf 0,g\rbrace$. Since $H$ is half-factorial, $\mathcal B$ is also half-factorial by \autoref{3.5}. By \autoref{3.12}, $\mathsf v_p(\mathcal A(D_1))=\lbrace 1\rbrace$ is equivalent to $D_1$ being half-factorial. We show that either
\begin{itemize}
 \item $\mathsf v_p(\mathcal A(D_1))=\lbrace 1\rbrace$ or
 \item $\mathsf v_p(\mathcal A(D_1))=\lbrace 1,2\rbrace$, $\iota(p)=g$, and $\iota(\wmal{D_1})=\lbrace\mathbf 0\rbrace$.
\end{itemize}
If $\#\mathsf v_p(\mathcal A(D_1))=1$, i.e., $\mathsf v_p(\mathcal A(D_1))=\lbrace n\rbrace$, then we find $n=1$ since $N_{\geq k}\subset n\N_0$.
Suppose we have $\#\mathsf v_p(\mathcal A(D_1))>1$. Then there are $n=\min\mathsf v_p(\mathcal A(D_1))$ and $m=\max\mathsf v_p(\mathcal A(D_1))$ $>n$. Let $\eps,\,\eta\in\wmal{D_1}$ be such that $p^n\eps,$ $p^m\eta\in\mathcal A(D_1)$. Now we distinguish four cases by $\iota(p^n\eps)$ and $\iota(p^m\eta)$.\\
\textbf{Case 1} $\iota(p^n\eps)=\iota(p^m\eta)=\mathbf 0$. Then $p^n\eps,\,p^m\eta\in\mathcal A(\mathcal B)$, and we find
\[
(p^m\eta)^k=(p^n\eps)^k(p^{(m-n)k}\eta^k\eps^{-k}).
\]
There are $k$ atoms on the left side and at least $k+1$ on the right side; clearly a contradiction to $\mathcal B$ being half-factorial.\\
\textbf{Case 2} $\iota(p^n\eps)=\iota(p^m\eta)=g$. Then $p^n\eps g,\,p^m\eta g\in\mathcal A(\mathcal B)$, and we find
\[
(p^m\eta g)^k=(p^n\eps g)^k(p^{(m-n)k}\eta^k\eps^{-k})
\]
There are $k$ atoms on the left side and at least $k+1$ on the right side; clearly a contradiction to $\mathcal B$ being half-factorial.\\
\textbf{Case 3} $\iota(p^n\eps)=\mathbf 0$ and $\iota(p^m\eta)=g$. Then $p^n\eps,\,p^m\eta g\in\mathcal A(\mathcal B)$, and we find
\[
(p^m\eta g)^k=
\begin{cases}
(p^n\eps)^k(g^2)^{\frac{k}{2}}(p^{(m-n)k}\eta^k\eps^{-k}) & k \text{ even} \\
(p^n\eps)^k(g^2)^{\frac{k-1}{2}}(p^{(m-n)k}\eta^k\eps^{-k}g) & k \text{ odd}.
\end{cases}
\]
There are $k$ atoms on the left side and, in both cases, at least $k+1+\frac{k-1}{2}$ on the right side; clearly a contradiction to $\mathcal B$ being half-factorial.\\
\textbf{Case 4} $\iota(p^n\eps)=g$ and $\iota(p^m\eta)=\mathbf 0$. Then $p^n\eps g,\,p^m\eta\in\mathcal A(\mathcal B)$. Now we must again distinguish four cases.\\
\indent\textbf{Case 4.1} $2n<m$ and $k$ even. Here we find
\[
(g^2)^\frac{k}{2}(p^m\eta)^k=(p^n\eps g)^k(p^{(m-n)k}\eps^{-k}\eta^k).
\]
There are $\frac{5}{2}k$ atoms on the left side and at least $k+\left\lceil\frac{(m-n)k}{m}\right\rceil$ atoms on the right side. This is a contradiction to $\mathcal B$ being half-factorial, since $m>2n$ by assumption.\\
\indent\textbf{Case 4.2} $2n<m$ and $k$ odd. Here we find
\[
(g^2)^\frac{k+1}{2}(p^m\eta)^{k+1}=(p^n\eps g)^{k+1}(p^{(m-n)(k+1)}\eps^{-k-1}\eta^{k+1}).
\]
This leads to a contradiction as in the case where $k$ was even.\\
\indent\textbf{Case 4.3} $m<2n$ and $k$ even. We choose $l\in\N$ maximal with $lm\leq(n-1)k$, and we find
\[
(p^n\eps g)^k=(g^2)^{\frac{k}{2}}(p^{nk-lm}\eps^k\eta^{-l})(p^m\eta)^l.
\]
There are $k$ atoms on the left side and at least $\frac{k}{2}+\left\lceil\frac{nk-lm}{m}\right\rceil+l$ on the right. This is a contradiction to $\mathcal B$ being half-factorial, since $m<2n$ by assumption.\\
\indent\textbf{Case 4.4} $m<2n$ and $k$ odd. We choose $l\in\N$ maximal with $lm\leq(n-1)(k+1)$, and we find a contradiction to $\mathcal B$ being half-factorial by looking at
\[
(p^n\eps g)^{k+1}=(g^2)^{\frac{k+1}{2}}(p^{n(k+1)-lm}\eps^{k+1}\eta^{-l})(p^m\eta)^l.
\]
\indent\textbf{Case 4.5} $m=2n$. In this particular case, we must again handle two additional cases.\\
\indent\textbf{Case 4.5.1} $n>1$. Then there is $n'\in(n,2n)$ and $\gamma\in\wmal D_1$ such that $p^{n'}\gamma\in\mathcal A(D_1)$. If $\iota(p^{n'}\gamma)=\mathbf 0$, then the assertion follows with $p^{n'}\gamma$ and $p^m\eta$ as in Case 1. If $\iota(p^{n'}\gamma)=g$, then the assertion follows with $p^n\eps$ and $p^{n'}\gamma$ as in Case 2.\\
\indent\textbf{Case 4.5.2} $n=1$. Then $m=2n=2$. Without loss of generality we may assume that $p\in D_1$. Furthermore, $\iota(p^2\eta)=\mathbf 0$ implies $\iota(\eta)=\mathbf 0$. For the moment, we assume that $\iota(p)=\mathbf 0$ and $\iota(\eps)=g$. Then we are done by Case 1 with $p$ and $p^2\eta$. If now $\iota(p)=g$ and $\iota(\eps)=\mathbf 0$, we show $\iota(\wmal{D_1})=\lbrace\mathbf 0\rbrace$ or $\mathsf v_p(\mathcal A(D_1))=\lbrace 1\rbrace$. If $\iota(\wmal{D_1})=\lbrace\mathbf 0\rbrace$, then the second case in the assertion is fulfilled. Now suppose $\iota(\wmal{D_1})=G$, say there is some $\gamma\in\wmal{D_1}$ with $\iota(\gamma)=g$. Then there is some $k'\in[1,k]$ such that $p^{k'}\gamma\in D_1$. Thus there are $\eps_1,\ldots,\eps_l,\,\eta_1,\ldots,\eta_{l'}\in\wmal{D_1}$ such that $(p\eps_1)\mdots(p\eps_l)(p^2\eta_1)\mdots(p^2\eta_{l'})=p^{k'}\gamma$ is a factorization of $p^{k'}\gamma$ in $D_1$. Thus $\eps_1\mdots\eps_l\eta_1\mdots\eta_{l'}=\gamma$, and therefore either $\iota(\eps_i)=g$ for some $i\in[1,l]$ or $\iota(\eta_j)=g$ for some $j\in[1,l']$. In the first case, we are in the situation of Case 1 with $p\eps_i$ and $p^2\eta$, and in the second case, we are in the situation of Case 2 with $p$ and $p^2\eta_j$.
\end{proof}

\begin{corollary}
\label{4.6}
Let $\mathcal O$ be a half-factorial order in an algebraic number field $K$, $\mathcal O_K$ is integral closure, and let $\mathfrak p\in\mathfrak X(\mathcal O)$ be a prime ideal of $\mathcal O$ such that $\mathfrak p\supset(\mathcal O:\mathcal O_K)$.\\
Then $|\pic(\mathcal O)|\leq 2$ and $\mathcal O_{\mathfrak p}$ is either
\begin{itemize}
 \item half-factorial, and $\punkt{\mathcal O_{\mathfrak p}}\subset\punkt{(\mathcal O_K)_{\mathfrak p}}$ is a half-factorial monoid of type $(1,k)$ with $k\in\lbrace 1,2\rbrace$, or
 \item $\mathfrak p$ ramifies in $\mathcal O_K$ with ramification degree $2$, i.e. there is some $\overline{\mathfrak p}\in(\mathcal O_K)_{\mathfrak p}$ prime such that $\overline{\mathfrak p}^2\sim\mathfrak p$.
\end{itemize}
In particular, if $K$ is a quadratic number field, then $\mathcal O_{\mathfrak p}$ is half-factorial.
\end{corollary}
\begin{proof}
Let $\mathcal O$ be a half-factorial order in an algebraic number field $K$, let $\mathcal O_K$ be its integral closure, let $\mathcal P=\lbrace\mathfrak p\in\mathfrak X(\mathcal O)\mid\mathfrak p\not\supset(\mathcal O:\mathcal O_K)\rbrace$, and let $\mathcal P^*=\lbrace\mathfrak p\in\mathfrak X(\mathcal O)\mid\mathfrak p\supset(\mathcal O:\mathcal O_K)\rbrace$. By \cite[Theorem 3.7.1]{MR2194494}, we find that
\[
\punkt{\mathcal O_{\red}}\subset\mathcal F(\mathcal P)\times T
\quad\mbox{with }T=\prod_{\mathfrak p\in\mathcal P^*}(\punkt{\mathcal O_{\mathfrak p}})_{\red}
\]
is a saturated cofinal submonoid with class group $\pic(\mathcal O)$. Now, we obtain $|\pic(\mathcal O)|\leq 2$ by \autoref{3.6}.\ref{3.6.1}. Since $\mathcal O$ is half-factorial, i.e., $\rho(\mathcal O)=1<\infty$, we find, by \cite[Corollary 4.i]{MR1378586}, that $\mathfrak p$ does not split in $\mathcal O_K$. Thus $(\mathcal O_K)_{\mathfrak p}$ is a discrete valuation domain, in particular, it is local, and thus $\punkt{\mathcal O_{\mathfrak p}}\subset\punkt{(\mathcal O_K)_{\mathfrak p}}$ is a finitely primary monoid of rank $1$. Since $(\punkt{\mathcal O_{\mathfrak p}})_{\red}\subset T$ is a divisor-closed submonoid, the assertion follows immediately by \autoref{4.5}.\\
If $K$ is a quadratic number field, then $\mathcal O$ being half-factorial implies that $\mathfrak p$ is inert by \cite[First paragraph from the Proof of A2 in the Proof of Theorem 3.7.15]{MR2194494}, and therefore $\mathcal O_{\mathfrak p}$ is half-factorial.
\end{proof}

\medskip
\subsection{Characterization of half-factorial orders in quadratic number fields}

\begin{corollary}
\label{4.7}
Let $\mathcal O$ be a non-principal order in a quadratic number field $K$, let $\mathcal O_K$ be its integral closure, and let $\mathcal P^*=\lbrace\mathfrak p\in\mathfrak X(\mathcal O)\mid\mathfrak p\supset(\mathcal O:\mathcal O_K)\rbrace$.\\
Then the following are equivalent:
\begin{enumerate}
 \item\label{4.7.1} $\mathcal O$ is half-factorial.
 \item\label{4.7.2} $\mathsf c(\mathcal O)=2$.
 \item\label{4.7.3} $|\pic(\mathcal O)|\leq 2$, $\mathcal O$ is locally half-factorial and, for all $\mathfrak p\in\mathcal P^*$, $[\mal{(\mathcal O_K)_{\mathfrak p}}/\mal{\mathcal O_{\mathfrak p}}]_{\pic(\mathcal O)}=[\mathbf 0]_{\pic(\mathcal O)}$.
 \item\label{4.7.4} $|\pic(\mathcal O)|\leq 2$ and, for all $\mathfrak p\in\mathcal P^*$,
\begin{itemize}
 \item $[\mal{(\mathcal O_K)_{\mathfrak p}}/\mal{\mathcal O_{\mathfrak p}}]_{\pic(\mathcal O)}=[\mathbf 0]_{\pic(\mathcal O)}$,
 \item $\mathfrak p$ is inert in $\mathcal O_K$, and
 \item $\mathfrak p^2\not\supset(\mathcal O:\mathcal O_K)$.
\end{itemize}
\end{enumerate}
\end{corollary}
\begin{proof}
\textbf{\ref{4.7.1} $\Leftrightarrow$ \ref{4.7.2}}
By \autoref{4.6}, we reason $\mathcal I^*(\mathcal O)$ is half-factorial. Thus the assertion is already shown in the additional statement of \autoref{4.3}.\\
\textbf{\ref{4.7.1} $\Rightarrow$ \ref{4.7.3}}
By \autoref{4.6}, $|\pic(\mathcal O)|\leq 2$ and $\mathcal O_\mathfrak p$ is half-factorial for all $\mathfrak p\in\mathcal P^*$. We get $[\mal{(\mathcal O_K)_{\mathfrak p}}/\mal{\mathcal O_{\mathfrak p}}]_{\pic(\mathcal O)}=[\mathbf 0]_{\pic(\mathcal O)}$ by the same construction as in the proof of \autoref{4.3} and \autoref{main} using \autoref{3.18}.\\
\textbf{\ref{4.7.3} $\Rightarrow$ \ref{4.7.1}}
Since, by assumption, $\mathcal O$ is locally half-factorial, this implication follows, directly, by the same construction as in the proof of \autoref{4.3} and \autoref{main} using \autoref{3.18}.\\
\textbf{\ref{4.7.3} $\Leftrightarrow$ \ref{4.7.4}}
Since, for all $\mathfrak p\in\mathcal P^*$, $\mathcal O_{\mathfrak p}$ is half-factorial if and only if $\mathfrak p$ is inert in $\mathcal O_K$ and $\mathfrak p^2\not\supset(\mathcal O:\mathcal O_K)$, the assertion follows.
\end{proof}

\bigskip
\bibliographystyle{plain}

\end{document}